\documentclass[11pt,reqno]{amsart}
\usepackage{amssymb}
\usepackage{amsmath}
\usepackage{amsfonts}

\newtheorem{theorem}{Theorem}
\theoremstyle{plain}

\newtheorem{definition}{Definition}

\newtheorem{lemma}{Lemma}

\newtheorem{proposition}{Proposition}
\newtheorem{remark}{Remark}

\numberwithin{equation}{section}
\numberwithin{equation}{section}
\input{tcilatex}

\begin{document}
\title[Wave equation with nonlinear damping]{Energy decay rates for
solutions of the wave equations with nonlinear damping in exterior domain}
\author{M. Daoulatli}
\address{Imam Abdulrahman Bin Faisal University, King Saudi Arabia \&
University of Carthage, Tunisia. }
\email[M. Daoulatli]{moez.daoulatli@infcom.rnu.tn}
\date{\today }
\thanks{}

\begin{abstract}
In this paper we study the behaviors of the energy of solutions of the wave
equations with localized nonlinear damping in exterior domains.
\end{abstract}

\maketitle

\section{Introduction and Statement of the results}

Let $O$ be a compact domain of $%
\mathbb{R}
^{d}$ $\left( d\geq 1\right) $ with $C^{\infty }$ boundary $\Gamma $ and $%
\Omega =\mathbb{R}^{d}\backslash O$. Consider the following wave equation
with localized nonlinear damping 
\begin{equation}
\left\{ 
\begin{array}{lc}
\partial _{t}^{2}u-\Delta u+a\left( x\right) \left\vert \partial
_{t}u\right\vert ^{r-1}\partial _{t}u=0 & \text{in }\mathbb{R}_{+}\times
\Omega , \\ 
u=0 & \text{on }\mathbb{R}_{+}\times \Gamma , \\ 
u\left( 0,x\right) =u_{0}\quad \text{ and }\quad \partial _{t}u\left(
0,x\right) =u_{1}, & 
\end{array}%
\right.  \label{system}
\end{equation}%
here $\Delta $ denotes the Laplace operator in the space variables. $a\left(
x\right) $ is a nonnegative function in $L^{\infty }\left( \Omega \right) $.
Throughout this paper we assume that $1<r\leq 1+\frac{2}{d}.$ Below $r_{0}>0$
is a fixed constant such that $O\subset B_{r_{o}}=\{x\in 
\mathbb{R}
^{d};\left\vert x\right\vert <r_{0}\}.$

The existence and uniqueness of global solutions to the problem $\left( \ref%
{system}\right) $\ is standard (see \cite{lions straus}). If $\left(
u_{0},u_{1}\right) $ is in $H_{0}^{1}\left( \Omega \right) \cap H^{2}\left(
\Omega \right) \times H_{0}^{1}\left( \Omega \right) $, then the system (\ref%
{system}), admits a unique solution $u$ in the class%
\begin{equation*}
u\in C^{0}\left( 
\mathbb{R}
_{+},H_{0}^{1}\left( \Omega \right) \cap H^{2}\left( \Omega \right) \right)
\cap C^{1}\left( 
\mathbb{R}
_{+},H_{0}^{1}\left( \Omega \right) \right) .
\end{equation*}%
Let us consider the energy at instant $t$ defined by%
\begin{equation*}
E_{u}\left( t\right) =\frac{1}{2}\int_{\Omega }\left( \left\vert \nabla
u\left( t,x\right) \right\vert ^{2}+\left\vert \partial _{t}u\left(
t,x\right) \right\vert ^{2}\right) dx.
\end{equation*}%
The energy functional satisfies the following identity%
\begin{equation}
E_{u}\left( T\right) +\int_{0}^{T}\int_{\Omega }a\left( x\right) \left\vert
\partial _{t}u\right\vert ^{r+1}dxdt=E_{u}\left( 0\right) ,
\label{energy inequality}
\end{equation}%
for every $T\geq 0$. Moreover we have%
\begin{equation}
\begin{array}{l}
\left\Vert \nabla \partial _{t}u\right\Vert _{L^{\infty }\left( 
\mathbb{R}
_{+},L^{2}\left( \Omega \right) \right) }^{2}+\left\Vert \partial
_{t}^{2}u\right\Vert _{L^{\infty }\left( 
\mathbb{R}
_{+},L^{2}\left( \Omega \right) \right) }^{2} \\ 
\leq 2\left( 1+\left\Vert a\right\Vert _{L^{\infty }}\right) \left(
\left\Vert u_{0}\right\Vert _{H^{2}}^{2}+\left\Vert u_{1}\right\Vert
_{H^{1}}^{2}+\left\Vert u_{1}\right\Vert _{H^{1}}^{2r}\right) .%
\end{array}
\label{hight energy inequality}
\end{equation}

Now we give a summary of results on the asymptotic behavior of the energy of
solutions of the nonlinear system $\left( \ref{system}\right) $ in the free
space $%
\mathbb{R}
^{d}$ and for a globally distributed damping$.~$For the Klein Gordon
equation a polynomial decay rate is derived by Nakao \cite{nakao klein} for
compactly supported initial data and by Mochizuki and Motai \cite{MM} for
weighted initial data. More precisely they show that if $1<r<1+\frac{2}{d}$
the energy decays according to%
\begin{equation*}
E_{u}\left( t\right) \leq C\left( 1+t\right) ^{-\gamma },
\end{equation*}%
where $\gamma <\min \left( 1,\frac{2+d-dr}{r-1}\right) .$ If $r>1+\frac{2}{d}%
,$ Mochizuki and Motai \cite{MM} establishes a complementary non-decay
result for a dense set of initial data in $H^{1}\left( 
\mathbb{R}
^{d}\right) \times L^{2}\left( 
\mathbb{R}
^{d}\right) .$

For the wave equation we first quote the result of Ono \cite{ono 1}, in
which the author consider the wave equation with a damping term equal to $%
\partial _{t}u+g\left( \partial _{t}u\right) $ where $g$ superlinear and has
a polynomial growth. He showed the polynomial decay of the energy. We note
that in this case the $L^{2}$ norm of the time derivative on $%
\mathbb{R}
_{+}\times 
\mathbb{R}
^{d}$ of the solution is bounded by the energy of the initial data.
Mochizuki and Motai in \cite{MM} obtained a logarithmic decay rate when $%
1<r\leq 1+\frac{2}{d}$ and for a kind of weighted initial data. The
corresponding non-decay result in \cite{MM} requires $r>1+\frac{2}{d-1}.$
Todorova and Yordanov in \cite{todo yord} showed that for compactly
supported initial data there exists a positive constant $\tau $ such that $%
E_{u}\left( t\right) \leq C\left( 1+t\right) ^{-\tau },$ when $1<r\leq 1+%
\frac{2}{d+1}$ and $d\geq 3.$ The main idea in this paper is to use the
\textquotedblleft parabolic\textquotedblright\ effects coming from the
presence of the damping term. Recently, Wakasa and Yordanov in \cite{W
yordanov} studied the energy decay for dissipative nonlinear wave equations
in one space dimension. They established polynomial decay estimates for the
energy for compactly supported initial data. More explicetly they show that $%
E_{u}\left( t\right) \leq C\left( 1+t\right) ^{-\tau },$ when $1<r<3$ with $%
\tau <\min \left( \frac{1}{2},\frac{3-r}{r-1}\right) .$

In the case of exterior domain we mention the result of Nakao and Jung \cite%
{nakao jung} which consider a dissipation which is allowed to be nonlinear
only in a ball, but outside that ball the dissipation must be linear. For
the generalized Klein Gordon equation we quote the result of Nakao \cite%
{nakao 2011}.

For another type of total energy decay property we refer the reader to \cite%
{ikehata 1,kawa,Racke,nakao-ba,ono,ikehata2} and references therein.

Before introducing our results we shall state several assumptions:

\begin{description}
\item[Hyp A] There exists $L>r_{0}$ such that 
\begin{equation*}
a\left( x\right) \geq \epsilon _{0}>0\text{ for }\left\vert x\right\vert
\geq L.
\end{equation*}
\end{description}

\begin{definition}
Let $\omega $ be an open set of $\Omega .$

\begin{enumerate}
\item $\left( \omega ,T\right) $ geometrically controls $\Omega $, i.e.
every generalized geodesic travelling with speed $1$ and issued at $t=0$,
enters the set $\omega $ in a time $t<T$.

\item We say that $\omega $ satisfies GCC if there exists $T>0$ such that $%
\left( \omega ,T\right) $ geometrically controls $\Omega $.
\end{enumerate}
\end{definition}

This condition is called Geometric Control Condition (see e.g.\cite{blr} ).
We shall relate the open subset $\omega $ with the damper $a$ by%
\begin{equation*}
\omega \subset \left\{ x\in \Omega ;a\left( x\right) >\epsilon
_{0}>0\right\} .
\end{equation*}

We note that according to \cite{blr} and \cite{bg}\ the Geometric Control
Condition of Bardos et al is a necessary and sufficient condition for the
exponential decay of solutions of the wave equation in bounded domain.

In this paper, we deal with real solutions, the general case can be treated
in the same way. Throughout this paper we use the following notations%
\begin{equation*}
q\left( x\right) =\left( 1+\left\vert x\right\vert ^{2}\right) ^{\frac{1}{2}%
},\text{ for }x\in \Omega .
\end{equation*}

and%
\begin{equation*}
p=\left\{ 
\begin{array}{ll}
2\left( r+1\right) & \text{if }d\leq 3 \\ 
\frac{2d}{d-2} & \text{if }d\geq 4.%
\end{array}%
\right.
\end{equation*}%
Now we state the results of this paper.

\begin{theorem}
We assume that Hyp A holds and $\omega $ satisfies GCC. Let%
\begin{equation*}
\begin{array}{ll}
\gamma >0 & \text{if }1<r<1+\frac{2}{d} \\ 
0<\gamma <\frac{2}{r-1} & \text{if }r=1+\frac{2}{d}.%
\end{array}%
\end{equation*}%
Then there exists $C_{0}>0$ such that the following estimate%
\begin{equation*}
E_{u}\left( t\right) \leq C_{0}\left( \ln \left( 2+t\right) \right)
^{-\gamma }I_{0},\text{ for all }t\geq 0,
\end{equation*}%
holds for every solution $u$ of $\left( \ref{system}\right) $ with initial
data $\left( u_{0},u_{1}\right) $ in $H_{0}^{1}\left( \Omega \right) \cap
H^{2}\left( \Omega \right) \times H_{0}^{1}\left( \Omega \right) $, such that%
\begin{equation*}
\begin{array}{l}
\left\Vert \left( \ln \left( 1+q\right) \right) ^{\frac{\gamma }{2}}\nabla
u_{0}\right\Vert _{L^{2}}^{2}+\left\Vert \left( \ln \left( 1+q\right)
\right) ^{\frac{\gamma }{2}}u_{1}\right\Vert _{L^{2}}^{2}<+\infty ,%
\end{array}%
\end{equation*}%
where%
\begin{equation*}
\begin{array}{l}
I_{0}=\left\Vert u_{0}\right\Vert _{H^{2}}^{2}+\left\Vert u_{1}\right\Vert
_{H^{1}}^{2}+\left\Vert u_{1}\right\Vert _{H^{1}}^{2r}+\left\Vert
u_{0}\right\Vert _{L^{r+1}}^{r+1}+\left\Vert \left( \ln \left( 1+q\right)
\right) ^{\frac{\gamma }{2}}\nabla u_{0}\right\Vert _{L^{2}}^{2} \\ 
+\left\Vert \left( \ln \left( 1+q\right) \right) ^{\frac{\gamma }{2}%
}u_{1}\right\Vert _{L^{2}}^{2}+\left( \left\Vert u_{0}\right\Vert
_{H^{2}}^{2}+\left\Vert u_{1}\right\Vert _{H^{1}}^{2}+\left\Vert
u_{1}\right\Vert _{H^{1}}^{2r}\right) ^{\frac{p}{2}}+1.%
\end{array}%
\end{equation*}
\end{theorem}

In the result above we see that when $1<r<1+\frac{2}{d},$ we can take any $%
\gamma >0,$ so we expect that we can obtain a rate of decay of the energy
for a weight with a polynomial growth.

\begin{theorem}
We assume that Hyp A holds and $\omega $ satisfies GCC. We suppose that $%
1<r<1+\frac{2}{d}$. We set 
\begin{equation*}
\begin{array}{c}
\tau \left( r,\lambda \right) =\frac{r\delta _{0}^{r-1}\left( \lambda
+1\right) ^{r-1}\left( r+1\right) ^{r}}{1+\delta _{0}^{r-1}\left( \lambda
+1\right) ^{r-1}\left( r+1\right) ^{r}\left( r\delta _{0}^{\frac{r-1}{r}%
}\left( \lambda +1\right) \left( r+1\right) +1\right) \allowbreak },%
\end{array}%
\end{equation*}%
$\lambda $ any positive constant$~$and 
\begin{equation*}
\delta _{0}=\left( \lambda +1\right) ^{\frac{r^{2}}{r^{2}-1}}\left(
r+1\right) ^{-\frac{r}{r-1}}.
\end{equation*}

We take 
\begin{equation*}
\begin{array}{c}
\text{ }\gamma <\min \left( \tau \left( r,\lambda \right) ,\frac{d+2-dr}{r-1}%
,\frac{p-2r}{r-1}\right) ,%
\end{array}%
\end{equation*}%
and 
\begin{equation*}
\begin{array}{c}
\alpha \left( r,\lambda \right) =\frac{r\delta _{0}^{\frac{r^{2}-1}{r}%
}\left( 1+\lambda \right) ^{r}\left( r+1\right) ^{r+1}+1}{\delta
_{0}^{r}\left( r-\tau \right) \left( 1+\lambda \right) ^{r}\left( r+1\right)
^{r+1}}.%
\end{array}%
\end{equation*}%
Then there exists $C_{1}>0$ such that the following estimate%
\begin{equation*}
E_{u}\left( t\right) \leq C_{1}\left( 1+\alpha t\right) ^{-\gamma }I_{1},%
\text{ for all }t\geq 0,
\end{equation*}%
holds for every solution $u$ of $\left( \ref{system}\right) $ with initial
data $\left( u_{0},u_{1}\right) $ in $H_{0}^{1}\left( \Omega \right) \cap
H^{2}\left( \Omega \right) \times H_{0}^{1}\left( \Omega \right) $, such that%
\begin{equation*}
\begin{array}{c}
\left\Vert \left( 1+\alpha q\right) ^{\frac{\gamma }{2}}\nabla
u_{0}\right\Vert _{L^{2}}^{2}+\left\Vert \left( 1+\alpha q\right) ^{\frac{%
\gamma }{2}}u_{1}\right\Vert _{L^{2}}^{2}<+\infty ,%
\end{array}%
\end{equation*}%
where%
\begin{equation*}
\begin{array}{l}
I_{1}=\left\Vert u_{0}\right\Vert _{H^{2}}^{2}+\left\Vert u_{1}\right\Vert
_{H^{1}}^{2}+\left\Vert u_{1}\right\Vert _{H^{1}}^{2r}+\left\Vert
u_{0}\right\Vert _{L^{r+1}}^{r+1}+\left\Vert \left( 1+\alpha q\right) ^{%
\frac{\gamma }{2}}\nabla u_{0}\right\Vert _{L^{2}}^{2} \\ 
+\left\Vert \left( 1+\alpha q\right) ^{\frac{\gamma }{2}}u_{1}\right\Vert
_{L^{2}}^{2}+\left( \left\Vert u_{0}\right\Vert _{H^{2}}^{2}+\left\Vert
u_{1}\right\Vert _{H^{1}}^{2}+\left\Vert u_{1}\right\Vert
_{H^{1}}^{2r}\right) ^{\frac{p}{2}}+1.%
\end{array}%
\end{equation*}
\end{theorem}

\begin{remark}
\begin{enumerate}
\item We note that for a fixed $r,$ the best value of \ $\tau \left(
r,\lambda \right) $ is obtained when $\lambda $ goes to zero. In addition
the function $r\longmapsto \tau \left( r,0\right) =\frac{r}{r+2}$ is
increasing on $(1,3],$%
\begin{equation*}
\text{and }\underset{r\rightarrow 1}{\lim }\tau \left( r,0\right) =\frac{1}{3%
}\text{ }.
\end{equation*}%
This fact is natural, since the value of $\tau $ is essentially linked to
the fact that 
\begin{equation*}
\int_{0}^{\infty }\dint_{\Omega }a\left( x\right) \left( 1+\alpha \left(
q\left( x\right) +t\right) \right) ^{\gamma -r-1}\left\vert u\left(
t,x\right) \right\vert ^{r+1}dx
\end{equation*}%
is finite and we cannot expect to obtain a better value of $\tau $ when $r$
decrease.

\item The best rate of decay is obtained, when we take%
\begin{equation*}
\gamma <\min \left( \frac{r}{r+2},\frac{d+2-dr}{r-1},\frac{p-2r}{r-1}\right)
.
\end{equation*}

\item We remark that the function $r\longmapsto \alpha \left( r,0\right) $
is decreasing on $(1,3]$, $\underset{r\rightarrow 1}{\lim }\alpha \left(
r,0\right) =\infty $ and $\alpha \left( 3,0\right) \geq \frac{10}{3}.$
\end{enumerate}
\end{remark}

The case of initial data with compact support

\begin{theorem}
We assume that Hyp A holds and $\omega $ satisfies GCC. We suppose that $%
1<r<1+\frac{2}{d}$. We set 
\begin{equation*}
\begin{array}{c}
\tau _{1}\left( r,\lambda \right) =\frac{2r\delta _{0}^{r-1}\left( \lambda
+1\right) ^{r-1}\left( r+1\right) ^{r}}{1+\delta _{0}^{r-1}\left( \lambda
+1\right) ^{r-1}\left( r+1\right) ^{r}\left( r\delta _{0}^{\frac{r-1}{r}%
}\left( \lambda +1\right) \left( r+1\right) +2\right) \allowbreak },%
\end{array}%
\end{equation*}%
$\lambda $ any positive constant$~$and 
\begin{equation*}
\delta _{0}=\left( \lambda +1\right) ^{\frac{r^{2}}{r^{2}-1}}\left(
r+1\right) ^{-\frac{r}{r-1}}.
\end{equation*}

We take 
\begin{equation*}
\begin{array}{c}
\gamma <\min \left( \tau _{1}\left( r,\lambda \right) ,\frac{d+2-dr}{r-1},%
\frac{p-2r}{r-1}\right) ,%
\end{array}%
\end{equation*}%
and 
\begin{equation*}
\begin{array}{c}
\alpha \left( r,\lambda \right) =\frac{r\delta _{0}^{\frac{r^{2}-1}{r}%
}\left( 1+\lambda \right) ^{r}\left( r+1\right) ^{r+1}+1}{\delta
_{0}^{r}\left( r-\tau _{1}\right) \left( 1+\lambda \right) ^{r}\left(
r+1\right) ^{r+1}}.%
\end{array}%
\end{equation*}%
Then there exists $C_{1}>0$ such that the following estimate%
\begin{equation*}
E_{u}\left( t\right) \leq C_{1}\left( \frac{R}{R+\alpha t}\right) ^{\gamma
}I_{2},\text{ for all }t\geq 0,
\end{equation*}%
holds for every solution $u$ of $\left( \ref{system}\right) $ with initial
data $\left( u_{0},u_{1}\right) $ in $H_{0}^{1}\left( \Omega \right) \cap
H^{2}\left( \Omega \right) \times H_{0}^{1}\left( \Omega \right) $ such that
the support of the initial data is contained in $B_{R}$, where%
\begin{equation*}
\begin{array}{l}
I_{2}=\left\Vert u_{0}\right\Vert _{H^{2}}^{2}+\left\Vert u_{1}\right\Vert
_{H^{1}}^{2}+\left\Vert u_{1}\right\Vert _{H^{1}}^{2r}+\left\Vert
u_{0}\right\Vert _{L^{r+1}}^{r+1} \\ 
+\left( \left\Vert u_{0}\right\Vert _{H^{2}}^{2}+\left\Vert u_{1}\right\Vert
_{H^{1}}^{2}+\left\Vert u_{1}\right\Vert _{H^{1}}^{2r}\right) ^{\frac{p}{2}%
}+1.%
\end{array}%
\end{equation*}
\end{theorem}

\begin{remark}
\begin{enumerate}
\item Our results are also valid for the case $\Omega =\mathbb{R}^{d},$ $%
d\geq 3$, where the boundary condition is dropped.

\item We note that for a fixed $r,$ the best value of \ $\tau _{1}\left(
r,\lambda \right) $ is obtained when $\lambda $ goes to zero. In addition
the function $r\longmapsto \tau _{1}\left( r,0\right) =\frac{2r}{r+3}$ is
increasing on $(1,3],$%
\begin{equation*}
\text{and }\underset{r\rightarrow 1}{\lim }\tau _{1}\left( r,0\right) =\frac{%
1}{2}.
\end{equation*}%
This fact is natural, since the value of $\tau _{1}$ is essentially linked
to the fact that 
\begin{equation*}
\int_{0}^{\infty }\dint_{\Omega }a\left( x\right) \left( R+\alpha t\right)
^{\gamma -r-1}\left\vert u\left( t,x\right) \right\vert ^{r+1}dx
\end{equation*}%
is finite, which depends on the behavior of 
\begin{equation*}
\dint_{\Omega }a\left( x\right) \left\vert u\left( t,x\right) \right\vert
^{r+1}dx,
\end{equation*}%
therefore we cannot expect to obtain a better value of $\tau _{1}$ when $r$
decrease.

\item The best rate of decay is obtained, when we take%
\begin{equation*}
\gamma <\min \left( \frac{2r}{r+3},\frac{d+2-dr}{r-1},\frac{p-2r}{r-1}%
\right) .
\end{equation*}%
When $d=1,$ we obtain that 
\begin{equation*}
\begin{array}{cc}
1/2<\gamma <\frac{2r}{r+3} & \text{if }1<r\leq \frac{2}{3}\sqrt{7}+\frac{1}{3%
} \\ 
\gamma <\frac{3-r}{r-1} & \text{if }\frac{2}{3}\sqrt{7}+\frac{1}{3}<r<3.%
\end{array}%
\end{equation*}%
Our decay rates is better or equal than the one obtained by Wakasa and
Yordanov in \cite{W yordanov}$.$

\item We remark that the function $r\longmapsto \alpha \left( r,0\right) $
is decreasing on $(1,3]$, $\underset{r\rightarrow 1}{\lim }\alpha \left(
r,0\right) =\infty $ and $\alpha \left( 3,0\right) \geq 18.$ In addition we
have%
\begin{equation*}
\alpha \left( r,\lambda \right) \geq \alpha \left( r,0\right) ,\text{ for
all }\left( r,\lambda \right) \in (1,3]\times 
\mathbb{R}
_{+}^{\ast },
\end{equation*}%
so the case $r=1$ cannot be obtained by letting $r$ goes to $1.$ 
\end{enumerate}
\end{remark}

The main difficulty in establishing such results is the lack of control of
the $L^{2}$ norm of the solution. This is an essential difference with the
equation in a bounded domain or the Klein-Gordon equation or in the case of
unbounded domain with finite measure \cite{caval}. The other difficulties is
that the $L^{2}$ norm of the time derivative on $%
\mathbb{R}
_{+}\times \Omega $ is not controlled by the initial energy and the fact
that the domain is with infinite measure.

To prove our results it is sufficient to show the integrability of $\varphi
^{\prime }E_{u}$ over $(0,\infty )$. For this purpose we show an estimate on
a functional $X(t)$ which control the weighted energy functional (see for
example \cite{Daou linear} and \cite{Daou coupled} for similar idea). Also
we prove a weighted observability estimate for the local energy of solutions
the wave equation with external force.

The rest of the paper is organized as follows. In section 2 we present some
results on the weighted energy and we give a weighted observability estimate
for the local energy. The section 3 is devoted to the proof of theorem 1 and
in section 4 \ we give the proof of theorem 2. In the last section we give
the needed results to show the theorem 3.

\section{Weighted observability estimate}

The next result concern the weighted energy estimate for solutions of $%
\left( \ref{system}\right) $ with initial data with finite weighted energy.

\begin{proposition}
\label{proposition weighted energy}Let $\varphi $ be a positive function in $%
C^{2}\left( 
\mathbb{R}
_{+}\right) $ such that $\varphi ^{\prime }$ and $\varphi ^{\prime \prime }$
are in $L^{\infty }\left( 
\mathbb{R}
_{+}\right) .$ Let $u$ be a solution of $\left( \ref{system}\right) $ with
initial data $\left( u_{0},u_{1}\right) $ in $H_{0}^{1}\left( \Omega \right)
\cap H^{2}\left( \Omega \right) \times H_{0}^{1}\left( \Omega \right) .$ We
set 
\begin{equation}
E_{\varphi }\left( u\right) \left( t\right) =\frac{1}{2}\int_{\Omega
}\varphi \left( \eta q\left( x\right) +\alpha t\right) \left( \left\vert
\nabla u\right\vert ^{2}+\left\vert \partial _{t}u\right\vert ^{2}\right) dx.
\label{weighted energy definition}
\end{equation}%
If $E_{\varphi }\left( u\right) \left( 0\right) $ $<\infty $, then%
\begin{equation}
\sqrt{\varphi }\nabla u\in L_{loc}^{\infty }\left( 
\mathbb{R}
_{+},\left( L^{2}\left( \Omega \right) \right) ^{d}\right) \text{ and }\sqrt{%
\varphi }\partial _{t}u\in L_{loc}^{\infty }\left( 
\mathbb{R}
_{+},L^{2}\left( \Omega \right) \right) .  \label{weighted regularity}
\end{equation}%
Moreover, we have%
\begin{equation}
\begin{array}{l}
\medskip E_{\varphi }\left( u\right) \left( t+T\right)
+\dint_{t}^{t+T}\dint_{\Omega }a\left( x\right) \varphi \left( s,x\right)
\left\vert \partial _{t}u\right\vert ^{r+1}dxds \\ 
\leq E_{\varphi }\left( u\right) \left( t\right) +\frac{\alpha +\eta }{2}%
\dint_{t}^{t+T}\dint_{\Omega }\left\vert \varphi ^{\prime }\left( s,x\right)
\right\vert \left( \left\vert \nabla u\left( s\right) \right\vert
^{2}+\left\vert \partial _{t}u\left( s\right) \right\vert ^{2}\right) dxds.%
\end{array}
\label{weighted energy}
\end{equation}%
for every $t\geq 0$ and $T\geq 0$, where $\varphi ^{\left( j\right) }\left(
t,x\right) =\varphi ^{\left( j\right) }\left( \eta q\left( x\right) +\alpha
t\right) ,$ for $j=0,1,2$ and $\alpha ,\eta \geq 0.$
\end{proposition}

\begin{proof}
The first step is to show $\left( \ref{weighted regularity}\right) .$ Let $%
n\in 
\mathbb{N}
^{\ast }.$ We define 
\begin{equation*}
g_{n}\left( s\right) =g\circ \left( I+n^{-1}g\right) ^{-1}\left( s\right)
=n\left( s-\left( I+n^{-1}g\right) ^{-1}\left( s\right) \right) ,
\end{equation*}%
the Yosida approximation of $g:s\longmapsto \left\vert s\right\vert ^{r-1}s.$
$g_{n}$ is monotone increasing, globally Lipschitz and $g_{n}\left( 0\right)
=0$. Let $w$ be the solution of the following system%
\begin{equation}
\left\{ 
\begin{array}{ll}
\partial _{t}^{2}w-\Delta w+a\left( x\right) \left( 1+\varphi \right) ^{%
\frac{1}{2}}g_{n}\left( \left( 1+\varphi \right) ^{-\frac{1}{2}}\partial
_{t}w+h\left( t,x\right) \right) =f\left( t,x\right) & 
\mathbb{R}
_{+}\times \Omega \\ 
w=0 & 
\mathbb{R}
_{+}\times \partial \Omega \\ 
\left( w\left( 0\right) ,\partial _{t}w\left( 0\right) \right) =\left(
w_{0},w_{1}\right) \in H_{D}\left( \Omega \right) \times L^{2}\left( \Omega
\right) & 
\end{array}%
\right.  \label{proof energy weighted system}
\end{equation}%
with $f\in L_{loc}^{2}\left( 
\mathbb{R}
_{+},L^{2}\left( \Omega \right) \right) $ and $\left( 1+\varphi \right) ^{%
\frac{1}{2}}h\in L_{loc}^{2}\left( 
\mathbb{R}
_{+},L^{2}\left( \Omega \right) \right) ,$ where $H_{D}\left( \Omega \right) 
$\ the completion of $C_{c}^{\infty }\left( \Omega \right) $\ with respect
to the norm%
\begin{equation*}
\left\Vert \varphi _{0}\right\Vert _{H_{D}}^{2}=\dint_{\Omega }\left\vert
\nabla \varphi _{0}\right\vert ^{2}dx.
\end{equation*}%
$g_{n}$ is a global Lipschitz. function, therefore 
\begin{equation*}
a\left( x\right) \left( 1+\varphi \right) ^{\frac{1}{2}}g_{n}\left( \left(
1+\varphi \right) ^{-\frac{1}{2}}\partial _{t}w+h\right) \in
L_{loc}^{2}\left( 
\mathbb{R}
_{+},L^{2}\left( \Omega \right) \right) .
\end{equation*}%
Using the fact that the function $f\in L_{loc}^{2}\left( 
\mathbb{R}
_{+},L^{2}\left( \Omega \right) \right) ,$ we infer that the unique solution
of $\left( \ref{proof energy weighted system}\right) $ 
\begin{equation*}
w\in C\left( 
\mathbb{R}
_{+},H_{D}\left( \Omega \right) \right) \text{ and }\partial _{t}w\in
C\left( 
\mathbb{R}
_{+},L^{2}\left( \Omega \right) \right) ,
\end{equation*}%
and the following energy identity%
\begin{eqnarray}
&&E_{w}\left( t\right) +\int_{0}^{t}\dint_{\Omega }a\left( x\right) \left(
1+\varphi \right) ^{\frac{1}{2}}g_{n}\left( \left( 1+\varphi \right) ^{-%
\frac{1}{2}}\partial _{t}w+h\left( s,x\right) \right) \partial _{t}wdxds 
\notag \\
&=&E_{w}\left( 0\right) +\int_{0}^{t}\int_{\Omega }f\left( s,x\right)
\partial _{t}wdxds  \label{proof weighted energy estimate energy identity}
\end{eqnarray}%
holds for every $t\geq 0.$

Let $u_{n}$ be the solution of the following system%
\begin{equation}
\left\{ 
\begin{array}{ll}
\partial _{t}^{2}u_{n}-\Delta u_{n}+a\left( x\right) g_{n}\left( \partial
_{t}u_{n}\right) =0 & 
\mathbb{R}
_{+}\times \Omega \\ 
u_{n}=0 & 
\mathbb{R}
_{+}\times \partial \Omega \\ 
\left( u_{n}\left( 0\right) ,\partial _{t}u_{n}\left( 0\right) \right)
=\left( u_{0},u_{1}\right) & 
\end{array}%
\right.  \label{sys: u mu}
\end{equation}%
with $\left( u_{0},u_{1}\right) $ in $H_{0}^{1}\left( \Omega \right) \cap
H^{2}\left( \Omega \right) \times H_{0}^{1}\left( \Omega \right) $ such that 
\begin{equation}
\dint_{\Omega }\varphi \left( \eta q\left( x\right) \right) \left(
\left\vert \nabla u_{0}\right\vert ^{2}+\left\vert u_{1}\right\vert
^{2}\right) dx<\infty .  \label{initial data regularity}
\end{equation}%
The function $g_{n}$ is globally Lipschitz., hence%
\begin{equation*}
u_{n}\in L^{\infty }\left( \left[ 0,T\right] ,H_{0}^{1}\left( \Omega \right)
\cap H^{2}\left( \Omega \right) \right) \cap W^{1,\infty }\left( \left[ 0,T%
\right] ,H_{0}^{1}\left( \Omega \right) \right) ,
\end{equation*}%
moreover we have the following energy identity 
\begin{equation}
E_{u_{n}}\left( t\right) +\int_{0}^{t}\dint_{\Omega }a\left( x\right)
g_{n}\left( \partial _{t}u_{n}\right) \partial _{t}u_{n}dxds=E_{u_{n}}\left(
0\right) .  \label{proof weighted energy u energy identity}
\end{equation}%
In addition we have 
\begin{equation}
\begin{array}{l}
\left\Vert \Delta u_{n}\right\Vert _{L^{\infty }\left( 
\mathbb{R}
_{+},L^{2}\left( \Omega \right) \right) }^{2}+\left\Vert \nabla \partial
_{t}u_{n}\right\Vert _{L^{\infty }\left( 
\mathbb{R}
_{+},L^{2}\left( \Omega \right) \right) }^{2}+\left\Vert \partial
_{t}^{2}u_{n}\right\Vert _{L^{\infty }\left( 
\mathbb{R}
_{+},L^{2}\left( \Omega \right) \right) }^{2} \\ 
\leq 2\left( 1+\left\Vert a\right\Vert _{L^{\infty }}\right) \left(
\left\Vert u_{0}\right\Vert _{H^{2}}^{2}+\left\Vert u_{1}\right\Vert
_{H^{1}}^{2}\right) .%
\end{array}
\label{proof weighted energy dtu energy identity}
\end{equation}%
From $\left( \ref{proof weighted energy u energy identity}\right) $ and $%
\left( \ref{proof weighted energy dtu energy identity}\right) ,$ we infer
that there exists $u$ \ and $\psi $ in $L^{\frac{r+1}{r}}\left( \left(
0,T\right) \times \Omega ,a\right) $\ such that 
\begin{equation}
\begin{array}{l}
u_{n}\underset{n\rightarrow +\infty }{\longrightarrow }u\text{ in the weak
star topology of }L^{\infty }\left( \left[ 0,T\right] ,H_{0}^{1}\left(
\Omega \right) \cap H^{2}\left( \Omega \right) \right) \\ 
\partial _{t}u_{n}\underset{n\rightarrow +\infty }{\longrightarrow }\partial
_{t}u\text{ in the weak star topology of }L^{\infty }\left( \left[ 0,T\right]
,H_{0}^{1}\left( \Omega \right) \right) \\ 
\left( I+n^{-1}g\right) ^{-1}\left( \partial _{t}u_{n}\right) \underset{%
n\rightarrow +\infty }{\longrightarrow }\partial _{t}u\text{ in the weak
topology of }L^{r+1}\left( \left( 0,T\right) \times \Omega ,a\right) \\ 
g_{n}\left( \partial _{t}u_{n}\right) \underset{n\rightarrow +\infty }{%
\longrightarrow }\psi \text{ in the weak topology of }L^{\frac{r+1}{r}%
}\left( \left( 0,T\right) \times \Omega ,a\right) ,%
\end{array}
\label{weak convergence un}
\end{equation}%
where 
\begin{equation*}
L^{r+1}\left( \left( 0,T\right) \times \Omega ,a\right) =\left\{ \varkappa ~%
\text{such that }\int_{0}^{T}\dint_{\Omega }\left\vert \varkappa \left(
s,x\right) \right\vert ^{r+1}a\left( x\right) dxds<\infty \right\} .
\end{equation*}%
To show that, $\psi =g\left( \partial _{t}u\right) ,$ we proceed as in \cite[%
P55-56]{lions straus}$.$ By a classical compactness argument, we can show
that there exists a subsequence of $\left( u_{n}\right) $ still denoted by $%
\left( u_{n}\right) ,$ such that%
\begin{equation}
\partial _{t}u_{n}\underset{n\rightarrow +\infty }{\longrightarrow }\partial
_{t}u\text{ strongly in }L^{2}\left( K\right) ,
\label{Strong convergence u n}
\end{equation}%
for a given compact subset $K$\ of $\left( 0,T\right) \times \Omega .$
Therefore we can assume that%
\begin{equation}
\partial _{t}u_{n}\underset{n\rightarrow +\infty }{\longrightarrow }\partial
_{t}u,\text{ a.e. in }K.  \label{Strong convergence u n consequence}
\end{equation}%
Since the function $%
\begin{array}{l}
s\longmapsto \left( I+n^{-1}g\right) ^{-1}\left( s\right) ,%
\end{array}%
$ is non-expansive on $%
\mathbb{R}
,$ we obtain%
\begin{equation*}
\left( I+n^{-1}g\right) ^{-1}\left( \partial _{t}u_{n}\right) \underset{%
n\rightarrow +\infty }{\longrightarrow }\partial _{t}u,\text{ a.e. in }K.
\end{equation*}%
Hence%
\begin{equation*}
g_{n}\left( \partial _{t}u_{n}\right) \underset{n\rightarrow +\infty }{%
\longrightarrow }g\left( \partial _{t}u\right) ,\text{ a.e. in }K.
\end{equation*}%
This enough to gives $\psi =g\left( \partial _{t}u\right) .$ Therefore $u$
is a solution of $\left( \ref{system}\right) $ with initial data in $%
H_{0}^{1}\left( \Omega \right) \cap H^{2}\left( \Omega \right) \times
H_{0}^{1}\left( \Omega \right) $ such that%
\begin{equation*}
\dint_{\Omega }\varphi \left( \eta q\left( x\right) \right) \left(
\left\vert \nabla u_{0}\right\vert ^{2}+\left\vert u_{1}\right\vert
^{2}\right) dx<\infty .
\end{equation*}%
We set $v_{n}=\left( 1+\varphi \right) ^{\frac{1}{2}}u_{n}.$ Therefore $%
v_{n} $ satisfies%
\begin{equation}
\left\{ 
\begin{array}{ll}
\partial _{t}^{2}v_{n}-\Delta v_{n}+a\left( x\right) \left( 1+\varphi
\right) ^{\frac{1}{2}}g_{n}\left( \left( 1+\varphi \right) ^{-\frac{1}{2}%
}\partial _{t}v_{n}+h_{n}\left( t,x\right) \right) =f\left( t,x\right) & 
\mathbb{R}
_{+}\times \Omega \\ 
v_{n}=0 & 
\mathbb{R}
_{+}\times \partial \Omega \\ 
\left( v_{n}\left( 0\right) ,\partial _{t}v_{n}\left( 0\right) \right)
=\left( v_{0},v_{1}\right) & 
\end{array}%
\right.
\end{equation}%
with%
\begin{equation*}
\begin{array}{l}
\left( v_{0},v_{1}\right) =\left( \left( 1+\varphi \left( \eta q\left(
x\right) \right) \right) ^{\frac{1}{2}}u_{0},\frac{1}{2}\left( 1+\varphi
\left( \eta q\left( x\right) \right) \right) ^{-\frac{1}{2}}\varphi ^{\prime
}\left( \eta q\left( x\right) \right) \eta u_{0}+\left( 1+\varphi \left(
\eta q\left( x\right) \right) \right) ^{\frac{1}{2}}u_{1}\right) \\ 
h_{n}=-\frac{\alpha }{2}\left( 1+\varphi \right) ^{-1}\varphi ^{\prime }u_{n}%
\end{array}%
\end{equation*}%
and%
\begin{equation*}
\begin{array}{l}
f=\frac{1}{2}\left( 1+\varphi \right) ^{-\frac{1}{2}}\left[ \eta ^{2}\left(
\varphi ^{\prime \prime }-\frac{1}{2}\left( \varphi ^{\prime }\right)
^{2}\left( 1+\varphi \right) ^{-1}\right) \frac{\left\vert x\right\vert ^{2}%
}{q^{2}}+\eta \left( \frac{d}{q}-\frac{\left\vert x\right\vert ^{2}}{q^{3}}%
\right) \varphi ^{\prime }\right] u_{n} \\ 
+\frac{\alpha ^{2}}{2}\left( 1+\varphi \right) ^{-\frac{1}{2}}\left[ \varphi
^{\prime \prime }-\frac{1}{2}\left( \varphi ^{\prime }\right) ^{2}\left(
1+\varphi \right) ^{-1}\right] u_{n}+\varphi ^{\prime }\left( 1+\varphi
\right) ^{-\frac{1}{2}}\left( \alpha \partial _{t}u_{n}+\eta \frac{x\cdot
\nabla u_{n}}{q}\right) .%
\end{array}%
\end{equation*}%
Hence, recalling $\left( \ref{initial data regularity}\right) $, $\varphi
^{\prime }\in L^{\infty }\left( 
\mathbb{R}
_{+}\right) $ and $\varphi ^{\prime \prime }\in L^{\infty }\left( 
\mathbb{R}
_{+}\right) $%
\begin{equation*}
\begin{array}{l}
\left( v_{0},v_{1}\right) \in H_{D}\left( \Omega \right) \times L^{2}\left(
\Omega \right) \\ 
\left( 1+\varphi \right) ^{\frac{1}{2}}h_{n}\in L_{loc}^{2}\left( 
\mathbb{R}
_{+},L^{2}\left( \Omega \right) \right) \\ 
f\in L_{loc}^{2}\left( 
\mathbb{R}
_{+},L^{2}\left( \Omega \right) \right) .%
\end{array}%
\end{equation*}%
Therefore using $\left( \ref{proof weighted energy estimate energy identity}%
\right) $ along with%
\begin{equation*}
\left( 1+\varphi \right) ^{-\frac{1}{2}}\partial _{t}v_{n}-\frac{\alpha }{2}%
\left( 1+\varphi \right) ^{-1}\varphi ^{\prime }u_{n}=\partial _{t}u_{n},
\end{equation*}%
and making some arrangement, we deduce that%
\begin{equation}
\begin{array}{l}
\medskip E_{v_{n}}\left( t\right) +\dint_{0}^{t}\dint_{\Omega }a\left(
x\right) \left( 1+\varphi \right) g_{n}\left( \partial _{t}u_{n}\right)
\partial _{t}u_{n}dxds \\ 
\medskip =E_{v_{n}}\left( 0\right) +\dint_{0}^{t}\dint_{\Omega }\left(
1+\varphi \right) ^{\frac{1}{2}}f\left( s,x\right) \partial _{t}u_{n}dxds+%
\frac{\alpha }{2}\dint_{0}^{t}\dint_{\Omega }\left( 1+\varphi \right) ^{-%
\frac{1}{2}}f\left( s,x\right) \varphi ^{\prime }u_{n}dxds \\ 
-\frac{\alpha }{2}\dint_{0}^{t}\dint_{\Omega }a\left( x\right) \varphi
^{\prime }g_{n}\left( \partial _{t}u_{n}\right) u_{n}dxds.%
\end{array}
\label{proof energy weighted energy identity 1}
\end{equation}%
On the other hand, since $\varphi ^{\prime }\in L^{\infty }\left( 
\mathbb{R}
_{+}\right) $ and $\varphi ^{\prime \prime }\in L^{\infty }\left( 
\mathbb{R}
_{+}\right) $\ then there exists a positive constant $C=C\left( \varphi
\right) $ such that%
\begin{equation*}
\begin{array}{l}
\medskip \left\vert \dint_{0}^{t}\dint_{\Omega }\left( 1+\varphi \right) ^{%
\frac{1}{2}}f\left( s,x\right) \partial _{t}u_{n}dxds\right\vert \leq
C\dint_{0}^{t}\dint_{\Omega }\left\vert u_{n}\right\vert ^{2}+\left\vert
\partial _{t}u_{n}\right\vert ^{2}+\left\vert \nabla u_{n}\right\vert
^{2}dxds, \\ 
\medskip \left\vert \frac{\alpha }{2}\dint_{0}^{t}\dint_{\Omega }\left(
1+\varphi \right) ^{-\frac{1}{2}}f\left( s,x\right) \varphi ^{\prime
}u_{n}dxds\right\vert \leq C\dint_{0}^{t}\dint_{\Omega }\left\vert
u_{n}\right\vert ^{2}+\left\vert \partial _{t}u_{n}\right\vert
^{2}+\left\vert \nabla u_{n}\right\vert ^{2}dxds.%
\end{array}%
\end{equation*}%
To estimate the last term of the RHS of $\left( \ref{proof energy weighted
energy identity 1}\right) ,$ we use Young's inequality along with the fact
that $g(s)=\left\vert s\right\vert ^{r-1}s$\ 
\begin{equation*}
\left\vert \frac{\alpha }{2}\int_{0}^{t}\dint_{\Omega }a\left( x\right)
\varphi ^{\prime }g_{n}\left( \partial _{t}u_{n}\right) u_{n}dxds\right\vert
\leq C\int_{0}^{t}\dint_{\Omega }a\left( x\right) \left\vert
u_{n}\right\vert ^{r+1}+a\left( x\right) \left\vert \left( I+n^{-1}g\right)
^{-1}\left( \partial _{t}u_{n}\right) \right\vert ^{r+1}dxds.
\end{equation*}%
Now using $\left( \ref{proof weighted energy u energy identity}\right) $ and
the fact that%
\begin{equation}
g_{n}\left( \partial _{t}u_{n}\right) \partial _{t}u_{n}\geq \left\vert
\left( I+n^{-1}g\right) ^{-1}\left( \partial _{t}u_{n}\right) \right\vert
^{r+1},  \label{lipshitz approxmation dtu}
\end{equation}%
we infer that%
\begin{equation*}
\int_{0}^{t}\dint_{\Omega }a\left( x\right) \left\vert \left(
I+n^{-1}g\right) ^{-1}\left( \partial _{t}u_{n}\right) \right\vert
^{r+1}dxds\leq E_{u_{n}}\left( 0\right) ,
\end{equation*}%
and%
\begin{equation*}
\begin{array}{c}
\dint_{0}^{t}\dint_{\Omega }\left\vert \partial _{t}u_{n}\right\vert
^{2}+\left\vert \nabla u_{n}\right\vert ^{2}dxds\leq \left( 1+t\right)
E_{u_{n}}\left( 0\right) .%
\end{array}%
\end{equation*}%
We have%
\begin{equation}
\begin{array}{l}
\dint_{\Omega }\left\vert u_{n}\left( s\right) \right\vert ^{2}dx\leq
C\left( 1+s\right) \left( E_{u_{n}}\left( 0\right) +\left\Vert
u_{0}\right\Vert _{L^{2}}^{2}\right) \\ 
\text{and} \\ 
\dint_{\Omega }\left\vert u_{n}\left( s\right) \right\vert ^{r+1}dx\leq
C\left( 1+s\right) ^{\frac{r+1}{2}}\left( E_{u_{n}}\left( 0\right)
+\left\Vert u_{0}\right\Vert _{L^{2}}^{2}\right) ^{\frac{r+1}{2}}.%
\end{array}
\label{proof weighted energy L2 esimate}
\end{equation}%
Therefore%
\begin{equation*}
\begin{array}{l}
\dint_{0}^{t}\dint_{\Omega }\left\vert u_{n}\right\vert ^{2}dxds\leq C\left(
1+t\right) ^{2}\left( E_{u_{n}}\left( 0\right) +\left\Vert u_{0}\right\Vert
_{L^{2}}^{2}\right) , \\ 
\dint_{0}^{t}\dint_{\Omega }\left\vert u_{n}\right\vert ^{r+1}dxds\leq
C\left( 1+t\right) ^{\frac{r+3}{2}}\left( E_{u_{n}}\left( 0\right)
+\left\Vert u_{0}\right\Vert _{L^{2}}^{2}\right) ^{\frac{r+1}{2}}.%
\end{array}%
\end{equation*}%
Combining the estimates above with $\left( \ref{proof energy weighted energy
identity 1}\right) ,$ we obtain\ 
\begin{equation}
\begin{array}{l}
E_{v_{n}}\left( t\right) +\dint_{0}^{t}\dint_{\Omega }a\left( x\right)
\left( 1+\varphi \right) g_{n}\left( \partial _{t}u_{n}\right) \partial
_{t}u_{n}dxds \\ 
\leq C\left( 1+t\right) ^{3}\left( E_{v_{n}}\left( 0\right) +E_{u_{n}}\left(
0\right) +\left( E_{u_{n}}\left( 0\right) +\left\Vert u_{0}\right\Vert
_{L^{2}}^{2}\right) ^{\frac{r+1}{2}}+\left\Vert u_{0}\right\Vert
_{L^{2}}^{2}\right) .%
\end{array}
\label{weight energy proof 1}
\end{equation}%
It is easy to see that%
\begin{equation*}
E_{\varphi }\left( u_{n}\right) \left( t\right) \leq 2E_{v_{n}}\left(
t\right) +C\left\Vert u_{n}\left( t\right) \right\Vert _{L^{2}}^{2}.
\end{equation*}%
Therefore combining the estimate above with $\left( \ref{weight energy proof
1}\right) $ and $\left( \ref{proof weighted energy L2 esimate}\right) $ we
obtain\ 
\begin{equation}
\begin{array}{l}
E_{\varphi }\left( u_{n}\right) \left( t\right) +\dint_{0}^{t}\dint_{\Omega
}a\left( x\right) \left( 1+\varphi \right) g_{n}\left( \partial
_{t}u_{n}\right) \partial _{t}u_{n}dxds \\ 
\leq C\left( 1+t\right) ^{2}\left( E_{\varphi }\left( u_{n}\right) \left(
0\right) +\left\Vert u_{0}\right\Vert _{L^{2}}^{2}+E_{u_{n}}\left( 0\right)
+\left( E_{u_{n}}\left( 0\right) +\left\Vert u_{0}\right\Vert
_{L^{2}}^{2}\right) ^{\frac{r+1}{2}}\right) .%
\end{array}
\label{weighted estimate}
\end{equation}

Note that in the estimate above we have used the fact that%
\begin{equation*}
E_{v_{n}}\left( 0\right) \leq E_{\varphi }\left( u_{n}\right) \left(
0\right) +\left\Vert u_{0}\right\Vert _{L^{2}}^{2}.
\end{equation*}%
Now using $\left( \ref{weighted estimate}\right) $ and $\left( \ref{lipshitz
approxmation dtu}\right) ,$ we infer that 
\begin{equation*}
\begin{array}{l}
\sqrt{1+\varphi }\partial _{x_{i}}u_{n}\underset{n\rightarrow +\infty }{%
\longrightarrow }\psi _{i}\text{ in the weak star topology of }L^{\infty
}\left( \left[ 0,T\right] ,L^{2}\left( \Omega \right) \right) ,\text{ }i\in
\left\{ 1,..,d\right\} \\ 
\sqrt{1+\varphi }\partial _{t}u_{n}\underset{n\rightarrow +\infty }{%
\longrightarrow }\phi _{1}\text{ in the weak star topology of }L^{\infty
}\left( \left[ 0,T\right] ,L^{2}\left( \Omega \right) \right) \\ 
\left( a\left( 1+\varphi \right) \right) ^{\frac{1}{r+1}}\left(
I+n^{-1}g\right) ^{-1}\left( \partial _{t}u_{n}\right) \underset{%
n\rightarrow +\infty }{\longrightarrow }\phi _{2}\text{ in the weak topology
of }L^{r+1}\left( \left( 0,T\right) \times \Omega \right) .%
\end{array}%
\end{equation*}%
Now we show that%
\begin{equation*}
\begin{array}{c}
\psi _{i}=\sqrt{1+\varphi }\partial _{x_{i}}u,\text{ }\phi _{1}=\sqrt{%
1+\varphi }\partial _{t}u\text{ and }\phi _{2}=\left( a\left( 1+\varphi
\right) \right) ^{\frac{1}{r+1}}\partial _{t}u.%
\end{array}%
\end{equation*}%
Let $K$ be a compact set of $\left( 0,T\right) \times \Omega .$ Using $%
\left( \ref{Strong convergence u n consequence}\right) ,$ we get%
\begin{equation*}
\sqrt{1+\varphi }\partial _{t}u_{n}\underset{n\rightarrow +\infty }{%
\longrightarrow }\sqrt{1+\varphi }\partial _{t}u,\text{ a.e. in }K,
\end{equation*}%
and using the fact that the function $%
\begin{array}{l}
s\longmapsto \left( I+n^{-1}g\right) ^{-1}\left( s\right) ,%
\end{array}%
$ is non-expansive on $%
\mathbb{R}
,$ we obtain%
\begin{equation*}
\begin{array}{l}
\left( a\left( 1+\varphi \right) \right) ^{\frac{1}{r+1}}\left(
I+n^{-1}g\right) ^{-1}\left( \partial _{t}u_{n}\right) \underset{%
n\rightarrow +\infty }{\longrightarrow }\left( a\left( 1+\varphi \right)
\right) ^{\frac{1}{r+1}}\partial _{t}u,\text{ a.e. in }K,%
\end{array}%
\end{equation*}%
This is enough to imply%
\begin{equation*}
\begin{array}{c}
\phi _{1}=\sqrt{1+\varphi }\partial _{t}u\text{ and }\phi _{2}=\left(
a\left( 1+\varphi \right) \right) ^{\frac{1}{r+1}}\partial _{t}u.%
\end{array}%
\end{equation*}%
From $\left( \ref{weak convergence un}\right) $ and by a classical
compactness argument, we can show that there exists a subsequence of $\left(
u_{n}\right) $ still denoted by $\left( u_{n}\right) ,$ such that%
\begin{equation*}
\partial _{x_{i}}u_{n}\underset{n\rightarrow +\infty }{\longrightarrow }%
\partial _{x_{i}}u\text{ strongly in }L^{2}\left( K\right) .
\end{equation*}%
Therefore extracting a subsequence if necessary\ 
\begin{equation*}
\partial _{x_{i}}u_{n}\underset{n\rightarrow +\infty }{\longrightarrow }%
\partial _{x_{i}}u,\text{ a.e. in }K,
\end{equation*}%
which gives%
\begin{equation*}
\sqrt{1+\varphi }\partial _{x_{i}}u_{n}\underset{n\rightarrow +\infty }{%
\longrightarrow }\sqrt{1+\varphi }\partial _{x_{i}}u\text{ },\text{ a.e. in }%
K.
\end{equation*}%
We conclude that%
\begin{equation*}
\psi _{i}=\sqrt{1+\varphi }\partial _{x_{i}}u,\text{ }i\in \left\{
1,..,d\right\} .
\end{equation*}%
\ Therefore%
\begin{equation}
\sqrt{\varphi }\nabla u\in L_{loc}^{\infty }\left( 
\mathbb{R}
_{+},\left( L^{2}\left( \Omega \right) \right) ^{d}\right) \text{ and }\sqrt{%
\varphi }\partial _{t}u\in L_{loc}^{\infty }\left( 
\mathbb{R}
_{+},L^{2}\left( \Omega \right) \right) .
\label{proof weighted energy regularity weight}
\end{equation}%
Now we will prove the energy estimate $\left( \ref{weighted energy}\right) .$%
\ We remind that%
\begin{equation*}
u\in L_{loc}^{\infty }\left( 
\mathbb{R}
_{+},H_{0}^{1}\left( \Omega \right) \cap H^{2}\left( \Omega \right) \right)
\cap W^{1,\infty }\left( 
\mathbb{R}
_{+},H_{0}^{1}\left( \Omega \right) \right) \cap W^{2,\infty }\left( 
\mathbb{R}
_{+},L^{2}\left( \Omega \right) \right) .
\end{equation*}%
Let $R\gg 1$ and setting $S\left( R\right) =\partial B_{R}.$ It is easy to
see that%
\begin{equation*}
\begin{array}{l}
\frac{1}{2}\frac{d}{dt}\dint_{\Omega \cap B_{R}}\varphi \left( \left\vert
\nabla u\left( t\right) \right\vert ^{2}+\left\vert \partial _{t}u\left(
t\right) \right\vert ^{2}\right) dx+\dint_{\Omega \cap B_{R}}a\left(
x\right) \varphi \left\vert \partial _{t}u\left( t\right) \right\vert
^{r+1}dx \\ 
=\frac{\alpha }{2}\dint_{\Omega \cap B_{R}}\varphi ^{\prime }\left(
\left\vert \nabla u\left( t\right) \right\vert ^{2}+\left\vert \partial
_{t}u\left( t\right) \right\vert ^{2}\right) dx+\dint_{\Omega \cap
B_{R}}\varphi \nabla u\left( t\right) \cdot \nabla \partial _{t}u\left(
t\right) +\varphi \partial _{t}u\left( t\right) \partial _{t}^{2}u\left(
t\right) dx \\ 
+\dint_{\Omega \cap B_{R}}a\left( x\right) \varphi \left\vert \partial
_{t}u\left( t\right) \right\vert ^{r+1}dx \\ 
=\frac{\alpha }{2}\dint_{\Omega \cap B_{R}}\varphi ^{\prime }\left(
\left\vert \nabla u\left( t\right) \right\vert ^{2}+\left\vert \partial
_{t}u\left( t\right) \right\vert ^{2}\right) dx+\dint_{\Omega \cap
B_{R}}\nabla u\left( t\right) \cdot \nabla \left( \varphi \partial
_{t}u\left( t\right) \right) +\varphi \partial _{t}u\left( t\right) \partial
_{t}^{2}u\left( t\right) dx \\ 
+\dint_{\Omega \cap B_{R}}a\left( x\right) \varphi \left\vert \partial
_{t}u\left( t\right) \right\vert ^{r+1}dx-\eta \dint_{\Omega \cap
B_{R}}\varphi ^{\prime }\frac{x\cdot \nabla u\left( t\right) }{q\left(
x\right) }\partial _{t}u\left( t\right) dx.%
\end{array}%
\end{equation*}%
Green's formula along with the fact that $u$ is a solution of $\left( \ref%
{system}\right) ,$%
\begin{equation*}
\begin{array}{l}
\frac{1}{2}\frac{d}{dt}\dint_{\Omega \cap B_{R}}\varphi \left( \left\vert
\nabla u\left( t\right) \right\vert ^{2}+\left\vert \partial _{t}u\left(
t\right) \right\vert ^{2}\right) dx+\dint_{\Omega \cap B_{R}}a\left(
x\right) \varphi \left\vert \partial _{t}u\left( t\right) \right\vert
^{r+1}dx \\ 
=\frac{\alpha }{2}\dint_{\Omega \cap B_{R}}\varphi ^{\prime }\left(
\left\vert \nabla u\left( t\right) \right\vert ^{2}+\left\vert \partial
_{t}u\left( t\right) \right\vert ^{2}\right) dx-\eta \dint_{\Omega \cap
B_{R}}\varphi ^{\prime }\frac{x\cdot \nabla u\left( t\right) }{q\left(
x\right) }\partial _{t}u\left( t\right) dx+\dint_{S\left( R\right) }\varphi 
\frac{x\cdot \nabla u\left( t\right) }{R}\partial _{t}u\left( t\right) dS.%
\end{array}%
\end{equation*}%
\ Integrating the estimate above between $t$ and $t+T$, we obtain%
\begin{equation}
\begin{array}{l}
\dint_{\Omega \cap B_{R}}\varphi \left( \left\vert \nabla u\left( t+T\right)
\right\vert ^{2}+\left\vert \partial _{t}u\left( t+T\right) \right\vert
^{2}\right) dx+\dint_{t}^{t+T}\dint_{\Omega \cap B_{R}}a\left( x\right)
\varphi \left\vert \partial _{t}u\right\vert ^{r+1}dxds \\ 
\leq E_{\varphi }\left( u\right) \left( t\right) +\frac{\alpha }{2}%
\dint_{t}^{t+T}\dint_{\Omega }\left\vert \varphi ^{\prime }\right\vert
\left( \left\vert \nabla u\left( s\right) \right\vert ^{2}+\left\vert
\partial _{t}u\left( s\right) \right\vert ^{2}\right) dxds \\ 
+\eta \dint_{t}^{t+T}\dint_{\Omega }\left\vert \varphi ^{\prime }\frac{%
x\cdot \nabla u\left( s\right) }{q\left( x\right) }\partial _{t}u\left(
s\right) \right\vert dxds+\dint_{t}^{t+T}\dint_{S\left( R\right) }\varphi
\left\vert \frac{x\cdot \nabla u\left( s\right) }{R}\partial _{t}u\left(
s\right) \right\vert dSds.%
\end{array}
\label{proof weighted energy 1}
\end{equation}%
Using Young's inequality%
\begin{equation*}
\int_{t}^{t+T}\int_{S\left( R\right) }\varphi \left\vert \frac{x\cdot \nabla
u}{R}\partial _{t}u\right\vert dSd\tau \leq \frac{1}{2}\int_{t}^{t+T}\int_{S%
\left( R\right) }\left( \left\vert \partial _{r}u\right\vert ^{2}+\left\vert
\partial _{t}u\right\vert ^{2}\right) \varphi dSd\tau .
\end{equation*}%
From $\left( \ref{proof weighted energy regularity weight}\right) ,$ we
infer that 
\begin{equation*}
\underset{R\longrightarrow +\infty }{\lim \inf }\int_{t}^{t+T}\int_{S\left(
R\right) }\varphi \left\vert \frac{x\cdot \nabla u}{R}\partial
_{t}u\right\vert dSd\tau =0.
\end{equation*}%
Passing to the limit in $\left( \ref{proof weighted energy 1}\right) ,$ we
get%
\begin{equation*}
\begin{array}{l}
E_{\varphi }\left( u\right) \left( t+T\right) +\dint_{t}^{t+T}\dint_{\Omega
}a\left( x\right) \varphi \left\vert \partial _{t}u\right\vert
^{r+1}dxds\leq E_{\varphi }\left( u\right) \left( t\right) \\ 
+\frac{\alpha }{2}\dint_{t}^{t+T}\dint_{\Omega }\left\vert \varphi ^{\prime
}\right\vert \left( \left\vert \nabla u\left( s\right) \right\vert
^{2}+\left\vert \partial _{t}u\left( s\right) \right\vert ^{2}\right)
dxds+\eta \dint_{t}^{t+T}\dint_{\Omega }\left\vert \varphi ^{\prime }\frac{%
x\cdot \nabla u}{q\left( x\right) }\partial _{t}u\left( s\right) \right\vert
dxds%
\end{array}%
\end{equation*}%
\ Young's inequality, gives%
\begin{equation*}
\begin{array}{l}
E_{\varphi }\left( u\right) \left( t+T\right) +\dint_{t}^{t+T}\dint_{\Omega
}a\left( x\right) \varphi \left\vert \partial _{t}u\right\vert ^{r+1}dxds \\ 
\leq E_{\varphi }\left( u\right) \left( t\right) +\frac{\alpha +\eta }{2}%
\dint_{t}^{t+T}\dint_{\Omega }\left\vert \varphi ^{\prime }\right\vert
\left( \left\vert \nabla u\left( s\right) \right\vert ^{2}+\left\vert
\partial _{t}u\left( s\right) \right\vert ^{2}\right) dxds.%
\end{array}%
\end{equation*}
\end{proof}

The proof of our results need a weighted observability estimate for the
local energy and to show such result we need to prove a unique continuation
result for the wave equation.

\begin{lemma}
\label{lemma unique continuation}We assume that Hyp A holds and $\left(
\omega ,T\right) $ geometrically controls $\Omega $. Then the only solution
of the system%
\begin{equation}
\left\{ 
\begin{array}{ll}
\partial _{t}^{2}z-\Delta z=0 & \text{in }\left( 0,T\right) \times \Omega ,
\\ 
z=0 & \text{on }\left( 0,T\right) \times \Gamma , \\ 
a\left( x\right) \partial _{t}z=0 & \text{on }\left( 0,T\right) \times
\Omega ,%
\end{array}%
\right.  \label{wave system unique continuation}
\end{equation}%
in the class%
\begin{equation*}
C^{0}\left( \left[ 0,T\right] ;H_{D}\left( \Omega \right) \right) \cap
C^{1}\left( \left[ 0,T\right] ;L^{2}\left( \Omega \right) \right) ,
\end{equation*}%
is the null one, where $H_{D}\left( \Omega \right) $ is the completion of $%
C_{c}^{\infty }\left( \Omega \right) $\ with respect to the norm%
\begin{equation*}
\left\Vert \varphi \right\Vert _{H}^{2}=\dint_{\Omega }\left\vert \nabla
\varphi \left( x\right) \right\vert ^{2}dx.
\end{equation*}
\end{lemma}

\begin{proof}
Let $\chi \in C_{c}^{\infty }\left( 
\mathbb{R}
^{d}\right) $ such that $\chi =1$ on $\left\{ \left\vert x\right\vert \leq
L\right\} $ and the support of $\chi $ is contained in $\left\{ \left\vert
x\right\vert \leq 2L\right\} .$ First we note that $H_{D}\left( \Omega
\right) \subset H_{loc}^{1}\left( \Omega \right) .$\ Let $z$ be a solution
of the system $\left( \ref{wave system unique continuation}\right) .$\ We
set $w=\chi z,$ we observe that 
\begin{equation*}
\left\{ 
\begin{array}{ll}
\partial _{t}^{2}w-\Delta w=-2\nabla \chi \nabla z-z\Delta \chi & \text{in }%
\left( 0,T\right) \times \Omega \cap B_{2L}, \\ 
w=0 & \text{on }\left( 0,T\right) \times \Gamma \cup \left\{ \left\vert
x\right\vert =2L\right\} , \\ 
\left( w_{0},w_{1}\right) \in H_{0}^{1}\left( \Omega \cap B_{2L}\right)
\times L^{2}\left( \Omega \cap B_{2L}\right) &  \\ 
a\left( x\right) \partial _{t}w=0 & \text{on }\left( 0,T\right) \times
\Omega .%
\end{array}%
\right.
\end{equation*}%
From linear semi-group theory, we infer that 
\begin{equation*}
w\in C^{0}\left( \left[ 0,T\right] ;H_{0}^{1}\left( \Omega \cap
B_{2L}\right) \right) \cap C^{1}\left( \left[ 0,T\right] ;L^{2}\left( \Omega
\cap B_{2L}\right) \right) .
\end{equation*}%
We set 
\begin{equation*}
v_{n}\left( t,x\right) =n\left( w\left( t+\frac{1}{n},x\right) -w\left(
t,x\right) \right) .
\end{equation*}%
Since 
\begin{equation*}
a\left( x\right) \geq \epsilon _{0}>0\text{ for }\left\vert x\right\vert
\geq L,
\end{equation*}%
and $\chi =1$ on $\left\{ \left\vert x\right\vert \leq L\right\} ,$
therefore, $v_{n}$ is a solution of%
\begin{equation*}
\left\{ 
\begin{array}{ll}
\partial _{t}^{2}v_{n}-\Delta v_{n}=0 & \text{in }\left( 0,T\right) \times
\Omega \cap B_{2L} \\ 
v_{n}=0 & \text{on }\left( 0,T\right) \times \Gamma \cup \left\{ \left\vert
x\right\vert =2L\right\} \\ 
a\left( x\right) \partial _{t}v_{n}=0 & \text{on }\left( 0,T\right) \times
\Omega .%
\end{array}%
\right.
\end{equation*}%
We have $\left( \omega \cap B_{2L},T\right) $ geometrically controls $\Omega
\cap B_{2L}$ and 
\begin{equation*}
v_{n}\in C^{0}\left( \left[ 0,T\right] ;H_{0}^{1}\left( \Omega \cap
B_{2L}\right) \right) \cap C^{1}\left( \left[ 0,T\right] ;L^{2}\left( \Omega
\cap B_{2L}\right) \right) ,
\end{equation*}%
thus using the observability estimate for the wave equation in bounded
domain (see e.g. \cite{daou}), we end up with 
\begin{equation*}
E_{v_{n}}\left( s\right) =0,\text{ for all }s\in \left[ 0,T\right] .
\end{equation*}

On the other hand, 
\begin{equation*}
v_{n}\underset{n\rightarrow +\infty }{\longrightarrow }\partial _{t}w\text{
\ in }\mathcal{D}^{\prime }\left( \left( 0,T\right) \times \Omega \right) .
\end{equation*}%
We deduce that $\partial _{t}w=0.$ Recalling that $\chi =1$ on $\left\{
\left\vert x\right\vert \leq L\right\} ,$ hence%
\begin{equation*}
\partial _{t}z\left( t,x\right) =0,~\text{on }\left\{ \left\vert
x\right\vert \leq L\right\} .
\end{equation*}%
Using $a\left( x\right) \partial _{t}z=0$ on $\left( 0,T\right) \times
\Omega $\ along with $a\left( x\right) >\epsilon _{0}>0$ for $\left\vert
x\right\vert \geq L,$ we infer that $\partial _{t}z\equiv 0$ on $\left[ 0,T%
\right] \times \Omega $. This mean that $z\left( t,x\right) =z\left(
x\right) $ is independent of $t$. Therefore, we have 
\begin{equation*}
\begin{array}{l}
\Delta z=0\text{ and }z\in H_{D}\left( \Omega \right) ,%
\end{array}%
\end{equation*}%
we conclude from this that $z\equiv 0$ on $\left[ 0,T\right] \times \Omega $.
\end{proof}

In view of the fact that the energy doesn't control the $L^{2}$ norm of the
solution, we do not expect to prove an observability estimate for the global
energy and this is the essential difference with the equation in a bounded
domain or the Klein-Gordon equation.

We remind that under our assumptions we have the following Poincar\'{e}
inequality (see \cite{Dan-Shi} and \cite{Lax}) 
\begin{equation}
\left\Vert f\right\Vert _{L^{2}\left( \Omega \cap B_{R}\right) }\leq
C_{R}\left\Vert \nabla f\right\Vert _{L^{2}\left( \Omega \right) },\text{
for every }f\in H_{D}\left( \Omega \right) \text{ and }R\geq r_{0}.
\label{poincare inequality}
\end{equation}

Next we show a weighted observability estimate for the local energy of
solutions of the system $\left( \ref{system}\right) $.

\begin{proposition}
\label{proposition onservability global}We assume that Hyp A holds and $%
\omega $ satisfies GCC. Let $\delta >0$ and $R_{0}\geq L.$ Let $\varphi $ be
a positive function in $C^{2}\left( 
\mathbb{R}
_{+}\right) $ such that $\varphi ^{\prime }$ in $L^{\infty }\left( 
\mathbb{R}
_{+}\right) .$ We suppose that there exists a positive constant $K$ such that%
\begin{equation*}
\underset{%
\mathbb{R}
_{+}}{\sup }\left\vert \frac{\varphi ^{\prime \prime }\left( t\right) }{%
\varphi ^{\prime }\left( t\right) }\right\vert \leq K.
\end{equation*}%
Moreover we assume that the function $t\longmapsto \left\vert \frac{\varphi
^{\prime }\left( t\right) }{\varphi \left( t\right) }\right\vert $ is
monotone decreasing and $\underset{t\rightarrow +\infty }{\lim }\left\vert 
\frac{\varphi ^{\prime }\left( t\right) }{\varphi \left( t\right) }%
\right\vert =0$. There exist $T>0$ and $C_{T,\delta }=C\left( T,\delta
,R_{0}\right) >0$ , such that the following inequality%
\begin{equation}
\begin{array}{l}
\medskip \dint_{t}^{t+T}\dint_{\Omega \cap B_{R_{0}}}\varphi \left( q\left(
x\right) +s\right) \left( \left\vert u\right\vert ^{2}+\left\vert \nabla
u\right\vert ^{2}+\left\vert \partial _{t}u\right\vert ^{2}\right) dxds \\ 
\medskip \leq C_{T,\delta }\dint_{t}^{t+T}\dint_{\Omega }a\left( x\right)
\varphi \left( q\left( x\right) +s\right) \left\vert \partial
_{t}u\right\vert ^{2}dxds \\ 
\medskip +C_{T,\delta }\dint_{t}^{t+T}\dint_{\Omega }\varphi \left( q\left(
x\right) +s\right) \left\vert g\left( s,x\right) \right\vert ^{2}dxds \\ 
\medskip +C_{T,\delta }\dint_{t}^{t+T}\dint_{\Omega }\frac{\left( \varphi
^{\prime }\left( q\left( x\right) +s\right) \right) ^{2}}{\varphi \left(
q\left( x\right) +s\right) }a\left( x\right) \left\vert u\right\vert ^{2}dxds
\\ 
+\delta \dint_{t}^{t+T}\dint_{\Omega }\varphi \left( q\left( x\right)
+s\right) \left( \left\vert \nabla u\right\vert ^{2}+\left\vert \partial
_{t}u\right\vert ^{2}\right) dxds,%
\end{array}
\label{observability gradient1 obs}
\end{equation}%
holds for every%
\begin{equation*}
g\text{ such that }\sqrt{\varphi }g\in L_{loc}^{2}\left( 
\mathbb{R}
_{+},L^{2}\left( \Omega \right) \right) ,
\end{equation*}%
for all 
\begin{equation*}
u\in C^{0}\left( 
\mathbb{R}
_{+},H_{0}^{1}\left( \Omega \right) \right) \cap C^{1}\left( 
\mathbb{R}
_{+},L^{2}\left( \Omega \right) \right) ,
\end{equation*}%
solution of 
\begin{equation}
\left\{ 
\begin{array}{lc}
\partial _{t}^{2}u-\Delta u=g & \text{in }\mathbb{R}_{+}\times \Omega , \\ 
u=0 & \text{on }\mathbb{R}_{+}\times \Gamma , \\ 
u\left( 0,x\right) =u_{0}\quad \text{ and }\quad \partial _{t}u\left(
0,x\right) =u_{1}, & 
\end{array}%
\right.  \label{sys g obs}
\end{equation}%
such that $E_{\varphi }\left( u\right) \left( 0\right) <\infty $.
\end{proposition}

\begin{proof}
Let $T>0$ such that $\left( \omega ,T\right) $ geometrically controls $%
\Omega $.

To prove this result we argue by contradiction: If $\left( \ref%
{observability gradient1 obs}\right) $ was false, there would exist a
sequences $\left( t_{n}\right) ,$~$\left( g_{n}\right) $ such that $\sqrt{%
\varphi }g_{n}\in L_{loc}^{2}\left( 
\mathbb{R}
_{+},L^{2}\left( \Omega \right) \right) $ and a sequence of solutions $%
\left( u_{n}\right) $ in $C^{0}\left( 
\mathbb{R}
_{+},H_{0}^{1}\left( \Omega \right) \right) \cap C^{1}\left( 
\mathbb{R}
_{+},L^{2}\left( \Omega \right) \right) $ with $E_{\varphi }\left(
u_{n}\right) \left( 0\right) <\infty $ and such that%
\begin{equation}
\begin{array}{l}
\medskip \dint_{t_{n}}^{t_{n}+T}\dint_{\Omega \cap B_{R_{0}}}\varphi \left(
q\left( x\right) +s\right) \left( \left\vert u_{n}\right\vert
^{2}+\left\vert \nabla u_{n}\right\vert ^{2}+\left\vert \partial
_{t}u_{n}\right\vert ^{2}\right) dxds \\ 
\medskip \geq n\left( \dint_{t_{n}}^{t_{n}+T}\dint_{\Omega }a\left( x\right)
\varphi \left( q\left( x\right) +s\right) \left\vert \partial
_{t}u_{n}\right\vert ^{2}dxds\right) \\ 
\medskip +n\dint_{t_{n}}^{t_{n}+T}\dint_{\Omega }\varphi \left( q\left(
x\right) +s\right) \left\vert g_{n}\left( s,x\right) \right\vert ^{2}dxds \\ 
\medskip +n\left( \dint_{t_{n}}^{t_{n}+T}\dint_{\Omega }\frac{\left( \varphi
^{\prime }\left( q\left( x\right) +s\right) \right) ^{2}}{\varphi \left(
q\left( x\right) +s\right) }a\left( x\right) \left\vert u_{n}\right\vert
^{2}dxds\right) \\ 
+\delta \dint_{t_{n}}^{t_{n}+T}\dint_{\Omega }\varphi \left( q\left(
x\right) +s\right) \left( \left\vert \nabla u_{n}\right\vert ^{2}+\left\vert
\partial _{t}u_{n}\right\vert ^{2}\right) dxds.%
\end{array}
\label{proof proposition 1 contradiction estimate obs}
\end{equation}

$~$

\begin{description}
\item[First case] The sequence $\left( t_{n}\right) $ is bounded.

$\varphi $ is a continuous positive function on $%
\mathbb{R}
_{+},$ therefore for all $K>R_{0}$ there exist $M>N>0$ such that%
\begin{equation}
N\leq \varphi \left( q\left( x\right) +t_{n}+s\right) \leq M,\text{ for all }%
\left( s,x\right) \in \left[ 0,T\right] \times B_{K}.  \label{fi boundness}
\end{equation}%
We set%
\begin{equation*}
\begin{array}{l}
\sigma _{n}^{2}=\dint_{t_{n}}^{t_{n}+T}\dint_{\Omega \cap B_{R_{0}}}\left(
\left\vert u_{n}\right\vert ^{2}+\left\vert \nabla u_{n}\right\vert
^{2}+\left\vert \partial _{t}u_{n}\right\vert ^{2}\right) dxds \\ 
\text{ and }v_{n}\left( t,x\right) =\frac{u_{n}\left( t_{n}+t,x\right) }{%
\sigma _{n}}.%
\end{array}%
\end{equation*}%
From $\left( \ref{proof proposition 1 contradiction estimate obs}\right) $
and $\left( \ref{fi boundness}\right) $, we infer that%
\begin{equation}
\begin{array}{l}
\dint_{t_{n}}^{t_{n}+T}\dint_{\Omega \cap B_{K}}\left( \left\vert \nabla
v_{n}\left( t\right) \right\vert ^{2}+\left\vert \partial _{t}v_{n}\left(
t\right) \right\vert ^{2}\right) dxdt\leq C_{\delta } \\ 
\text{and }\dint_{t_{n}}^{t_{n}+T}\dint_{\Omega \cap B_{R_{0}}}\left\vert
v_{n}\left( t\right) \right\vert ^{2}dxdt\leq C,%
\end{array}
\label{contradiction boundness}
\end{equation}%
and%
\begin{equation}
\begin{array}{l}
\dint_{0}^{T}\dint_{\Omega \cap B_{K}}a\left( x\right) \left\vert \partial
_{t}v_{n}\left( s,x\right) \right\vert ^{2}dxds\underset{n\rightarrow
+\infty }{\longrightarrow }0 \\ 
\frac{1}{\sigma _{n}^{2}}\dint_{0}^{T}\dint_{\Omega \cap B_{K}}\left\vert
g_{n}\left( s+t_{n},x\right) \right\vert ^{2}dxds\underset{n\rightarrow
+\infty }{\longrightarrow }0,%
\end{array}
\label{proof of proposition 5 fn converge obs 1 case}
\end{equation}%
for all $K>R_{0}.$ We note that since the function $t\longmapsto \left\vert 
\frac{\varphi ^{\prime }\left( t\right) }{\varphi \left( t\right) }%
\right\vert $ is monotone decreasing and the the sequence $\left(
t_{n}\right) $\ is bounded, then for all $K\geq L,$ we have%
\begin{equation}
\begin{array}{l}
\dint_{0}^{T}\dint_{\Omega \cap B_{K}}a\left( x\right) \left\vert
v_{n}\left( s,x\right) \right\vert ^{2}dxds \\ 
\leq C\dint_{0}^{T}\dint_{\Omega }\frac{\left( \varphi ^{\prime }\left(
q\left( x\right) +t_{n}+t\right) \right) ^{2}}{\varphi \left( q\left(
x\right) +t_{n}+t\right) }a\left( x\right) \left\vert u_{n}\left( s,x\right)
\right\vert ^{2}dxds\underset{n\rightarrow +\infty }{\longrightarrow }0.%
\end{array}
\label{a vn convergence}
\end{equation}%
Then the result above combined with $\left( \ref{contradiction boundness}%
\right) ,$ gives 
\begin{equation}
\int_{0}^{T}\int_{\Omega \cap B_{K}}\left( \left\vert v_{n}\left( t\right)
\right\vert ^{2}+\left\vert \nabla v_{n}\left( t\right) \right\vert
^{2}+\left\vert \partial _{t}v_{n}\left( t\right) \right\vert ^{2}\right)
dxdt\leq C_{\delta },~\text{for n large enough.}
\label{energy bound first case}
\end{equation}

We take $R_{1}$ and $R_{2}$ such that, $R_{2}>R_{1}>\max \left(
R_{0},2L\right) $ and let $\psi \in C_{c}^{\infty }\left( 
\mathbb{R}
^{d}\right) $ such that $\psi =1$ on $\left\{ x\in 
\mathbb{R}
^{d},\frac{3L}{2}\leq \left\vert x\right\vert \leq R_{1}\right\} $ and the
support of $\psi $ is contained in $\left\{ x\in 
\mathbb{R}
^{d},~L\leq \left\vert x\right\vert \leq R_{2}\right\} .$ Let $0<\epsilon
<<1 $ and $\eta $ be a nonnegative function in $C_{c}^{\infty }\left(
0,T\right) $ such that 
\begin{equation*}
\eta \left( s\right) =1\text{ for }\epsilon \leq s\leq T-\epsilon .
\end{equation*}%
Now we show that%
\begin{equation}
\int_{\epsilon }^{T-\epsilon }\int_{\Omega \cap \left\{ \frac{3L}{2}\leq
\left\vert x\right\vert \leq R_{1}\right\} }\left\vert \nabla v_{n}\left(
s\right) \right\vert ^{2}+\left\vert v_{n}\left( s,x\right) \right\vert
^{2}dxds\underset{n\rightarrow +\infty }{\longrightarrow }0.
\label{exterior convergence to zero}
\end{equation}%
First we note that since the support of $\psi $ is contained in $\left\{
L\leq \left\vert x\right\vert \leq R_{2}\right\} $ and $a\left( x\right)
>\epsilon _{0}$ on $\left\{ L\leq \left\vert x\right\vert \right\} ,$ then
using $\left( \ref{a vn convergence}\right) $ we get%
\begin{equation*}
\epsilon _{0}\int_{0}^{T}\int_{\Omega \cap \left\{ \frac{3L}{2}\leq
\left\vert x\right\vert \leq R_{1}\right\} }\left\vert v_{n}\left(
s,x\right) \right\vert ^{2}dxds\leq \int_{0}^{T}\dint_{\Omega \cap \left\{ 
\frac{3L}{2}\leq \left\vert x\right\vert \leq R_{1}\right\} }a\left(
x\right) \left\vert v_{n}\left( s,x\right) \right\vert ^{2}dxds\underset{%
n\rightarrow +\infty }{\longrightarrow }0.
\end{equation*}%
We have, $v_{n}$ is a solution of the following system%
\begin{equation}
\left\{ 
\begin{array}{lc}
\partial _{t}^{2}v_{n}-\Delta v_{n}=\frac{1}{\sigma _{n}}g_{n}\left(
t,x\right) & \text{in }%
\mathbb{R}
_{+}\times \Omega , \\ 
v_{n}\left( t,x\right) =0 & \text{on }%
\mathbb{R}
_{+}\times \Gamma , \\ 
\left( v_{n}\left( 0\right) ,\partial _{t}v_{n}\left( 0\right) \right) =%
\frac{1}{\sigma _{n}}\left( u_{n}\left( t_{n}\right) ,\partial
_{t}u_{n}\left( t_{n}\right) \right) \in H_{0}^{1}\left( \Omega \right)
\times L^{2}\left( \Omega \right) . & 
\end{array}%
\right.  \label{proof proposition 4 system vn}
\end{equation}%
We multiply Eq$\left( \ref{proof proposition 4 system vn}\right) $ by $\eta
\psi ^{2}v_{n}$ and integrate over $\left( 0,T\right) \times \Omega ,$ we
obtain%
\begin{eqnarray*}
&&\int_{0}^{T}\dint_{\Omega }\eta \left( s\right) \psi ^{2}\left( x\right)
\left\vert \nabla v_{n}\left( s\right) \right\vert ^{2}dxds \\
&=&\int_{0}^{T}\dint_{\Omega }\eta ^{\prime }\left( s\right) \psi ^{2}\left(
x\right) v_{n}\left( s\right) \partial _{t}v_{n}\left( s\right) +\eta \left(
s\right) \psi ^{2}\left( x\right) \left\vert \partial _{t}v_{n}\left(
s\right) \right\vert ^{2}dxds \\
&&+\int_{0}^{T}\dint_{\Omega }\frac{1}{2}\eta \left( s\right) \Delta \psi
^{2}\left( x\right) \left\vert v_{n}\left( s\right) \right\vert ^{2}+\frac{1%
}{\sigma _{n}}\eta \left( s\right) \psi ^{2}\left( x\right) g_{n}\left(
s,x\right) v_{n}\left( s\right) dxds.
\end{eqnarray*}%
Using Young's inequality and the fact that $\eta $ is in $C_{c}^{\infty
}\left( 0,T\right) $, we infer that there exists a positive constant $c$
such that%
\begin{eqnarray*}
&&\int_{0}^{T}\dint_{\Omega }\eta \left( s\right) \psi ^{2}\left( x\right)
\left\vert \nabla v_{n}\left( s\right) \right\vert ^{2}dxds \\
&\leq &c\int_{0}^{T}\dint_{\Omega }\psi ^{2}\left( x\right) \left(
\left\vert \partial _{t}v_{n}\left( s\right) \right\vert ^{2}+\left\vert
v_{n}\left( s\right) \right\vert ^{2}\right) +\left\vert \Delta \psi
^{2}\left( x\right) \right\vert \left\vert v_{n}\left( s\right) \right\vert
^{2}dxds \\
&&+\frac{c}{\sigma _{n}^{2}}\int_{0}^{T}\int_{\Omega \cap
B_{R_{2}}}\left\vert g_{n}\left( s+t_{n},x\right) \right\vert ^{2}dxds,
\end{eqnarray*}%
therefore%
\begin{eqnarray*}
&&\int_{0}^{T}\dint_{\Omega }\eta \left( s\right) \psi ^{2}\left( x\right)
\left\vert \nabla v_{n}\left( s\right) \right\vert ^{2}dxds \\
&\leq &c\int_{0}^{T}\int_{\Omega \cap B_{R_{2}}}a\left( x\right) \left\vert
\partial _{t}v_{n}\left( s\right) \right\vert ^{2}+a\left( x\right)
\left\vert v_{n}\left( s\right) \right\vert ^{2}dxds \\
&&+c\int_{0}^{T}\int_{\Omega \cap B_{R_{2}}}\left\vert \frac{1}{\sigma _{n}}%
g_{n}\left( t_{n}+s,x\right) \right\vert ^{2}dxds.
\end{eqnarray*}%
Combining the estimate above with $\left( \ref{proof of proposition 5 fn
converge obs 1 case}\right) $ and $\left( \ref{a vn convergence}\right) ,$
we get%
\begin{equation*}
\begin{array}{l}
\dint_{\epsilon }^{T-\epsilon }\dint_{\Omega \cap \left\{ \frac{3L}{2}\leq
\left\vert x\right\vert \leq R_{1}\right\} }\left\vert \nabla v_{n}\left(
s\right) \right\vert ^{2}dxds \\ 
\leq \dint_{0}^{T}\dint_{\Omega }\eta \left( s\right) \psi ^{2}\left(
x\right) \left\vert \nabla v_{n}\left( s\right) \right\vert ^{2}dxds\underset%
{n\rightarrow +\infty }{\longrightarrow }0,%
\end{array}%
\end{equation*}%
we note that in the inequality above we have used the fact that $\psi =1$ on 
$\left\{ x\in 
\mathbb{R}
^{d},\frac{3L}{2}\leq \left\vert x\right\vert \leq R_{1}\right\} $ and $\eta
=1$ on $\left[ \epsilon ,T-\epsilon \right] .~$

Let $\chi \in C_{c}^{\infty }\left( 
\mathbb{R}
^{d}\right) $ such that $\chi =1$ on $\left\{ \left\vert x\right\vert \leq
R\right\} $ and the support of $\chi $ is contained in $\left\{ \left\vert
x\right\vert \leq R_{1}\right\} $ with $R_{1}>R>\max \left( R_{0},2L\right)
. $ We set $W_{n}=\chi v_{n},$ then $W_{n}$ is a solution of the following
system%
\begin{equation*}
\left\{ 
\begin{array}{ll}
\partial _{t}^{2}W_{n}-\Delta W_{n}=-2\nabla \chi \nabla v_{n}-v_{n}\Delta
\chi +\frac{1}{\sigma _{n}}\chi g_{n}\left( t,x\right) & 
\mathbb{R}
_{+}\times \Omega \cap B_{R_{1}}, \\ 
W_{n}=0 & 
\mathbb{R}
_{+}\times \Gamma \cup \left\{ \left\vert x\right\vert =R_{1}\right\} , \\ 
\left( W_{n}\left( 0\right) ,\partial _{t}W_{n}\left( 0\right) \right) =\chi
\left( v_{n}\left( 0\right) ,\partial _{t}v_{n}\left( 0\right) \right) . & 
\end{array}%
\right.
\end{equation*}%
In addition we have 
\begin{equation*}
W_{n}\in C\left( \left( 0,T\right) ,\text{ }H_{0}^{1}\left( \Omega \cap
B_{R_{1}}\right) \right) \cap C^{1}\left( \left( 0,T\right) ,\text{ }%
L^{2}\left( \Omega \cap B_{R_{1}}\right) \right) .
\end{equation*}%
Now we show that%
\begin{equation}
\underset{\lbrack 0,T]}{\sup }E_{W_{n}}\left( s\right) \leq C_{T,\delta },%
\text{ for n large enough.}  \label{energy Zn bounded1}
\end{equation}%
First we note that we have the following energy identity 
\begin{equation*}
\begin{array}{c}
tE_{W_{n}}\left( t\right) =\dint_{0}^{t}E_{W_{n}}\left( s\right)
ds+\dint_{0}^{t}\dint_{\Omega }s\left( -2\nabla \chi \nabla
v_{n}-v_{n}\Delta \chi +\frac{1}{\sigma _{n}}\chi g_{n}\right) \partial
_{t}W_{n}dxds%
\end{array}%
\end{equation*}%
for all $0\leq t\leq T.$ Then using Young's inequality and the fact that the
support of $W_{n}$ is contained in $\left\{ \left\vert x\right\vert \leq
R_{1}\right\} $, we deduce that 
\begin{equation*}
\begin{array}{l}
E_{W_{n}}\left( T\right) \\ 
\leq \frac{c}{T}\left( \dint_{0}^{T}\left( E_{W_{n}}\left( s\right)
+T\dint_{\Omega \cap B_{R_{1}}}\left\vert -2\nabla \chi \nabla
v_{n}-v_{n}\Delta \chi +\frac{1}{\sigma _{n}}\chi g_{n}\right\vert
^{2}+\left\vert \partial _{t}W_{n}\right\vert ^{2}dx\right) ds\right) \\ 
\leq \frac{c}{T}\dint_{0}^{T}\dint_{\Omega \cap B_{R_{1}}}\left\vert \nabla
v_{n}\right\vert ^{2}+\left\vert \partial _{t}v_{n}\right\vert
^{2}+\left\vert v_{n}\right\vert ^{2}+\left\vert \frac{1}{\sigma _{n}}\chi
g_{n}\right\vert ^{2}dxds.%
\end{array}%
\end{equation*}%
Combining the estimate above with $\left( \ref{proof of proposition 5 fn
converge obs 1 case}\right) $ and $\left( \ref{energy bound first case}%
\right) $, we obtain%
\begin{equation}
E_{W_{n}}\left( T\right) \leq C_{T,\delta },\text{ for n large enough.}
\label{bound1}
\end{equation}%
On the other hand, we have the following energy identity%
\begin{equation*}
\begin{array}{c}
E_{W_{n}}\left( t\right) =E_{W_{n}}\left( T\right)
+\dint_{t}^{T}\dint_{\Omega \cap B_{R_{1}}}\left( -2\nabla \chi \nabla
v_{n}-v_{n}\Delta \chi +\frac{1}{\sigma _{n}}\chi g_{n}\right) \partial
_{t}W_{n}dxds%
\end{array}%
\end{equation*}%
for all $0\leq t\leq T.$ Using Young's inequality and making some
arrangement, we deduce that%
\begin{equation*}
\begin{array}{l}
E_{W_{n}}\left( t\right) \\ 
\leq E_{W_{n}}\left( T\right) +c\dint_{0}^{T}\dint_{\Omega \cap
B_{R_{1}}}\left\vert -2\nabla \chi \nabla v_{n}-v_{n}\Delta \chi +\frac{1}{%
\sigma _{n}}\chi g_{n}\right\vert ^{2}+\left\vert \partial
_{t}v_{n}\right\vert ^{2}dxds,%
\end{array}%
\end{equation*}%
for all $0\leq t\leq T.~$The estimate above combined with $\left( \ref{proof
of proposition 5 fn converge obs 1 case}\right) ,$ $\left( \ref{energy bound
first case}\right) $ and $\left( \ref{bound1}\right) $ gives $\left( \ref%
{energy Zn bounded1}\right) .$

The next step is to show that 
\begin{equation}
\int_{0}^{T}E_{W_{n}}\left( s\right) ds\underset{n\rightarrow +\infty }{%
\longrightarrow }0.  \label{integral convergence to zero}
\end{equation}

For $\epsilon $ small enough, we have $\left( \omega \cap
B_{R_{1}},T-2\epsilon \right) $ geometrically controls $\Omega \cap
B_{R_{1}}.$ Therefore, using the control theory of the wave equation in
bounded domain, we deduce that the following observability estimate\ holds 
\begin{equation}
\begin{array}{l}
E_{W_{n}}\left( \epsilon \right) \\ 
\leq C_{\epsilon ,T}\left( \dint_{\epsilon }^{T-\epsilon }\dint_{\Omega \cap
B_{R_{1}}}a\left( x\right) \left\vert \partial _{t}v_{n}\right\vert
^{2}+\left\vert -2\nabla \chi \nabla v_{n}-v_{n}\Delta \chi +\frac{1}{\sigma
_{n}}\chi g_{n}\right\vert ^{2}dxds\right) ,%
\end{array}
\label{E Wn estimate}
\end{equation}%
(we can show this result using \cite{daou}.) Recalling%
\begin{equation*}
\nabla \chi =0\text{ on }\left\{ \left\vert x\right\vert \leq 2L\right\} 
\text{ and }Supp\chi \subset \left\{ \left\vert x\right\vert \leq
R_{1}\right\} .
\end{equation*}%
Hence $\left( \ref{exterior convergence to zero}\right) $ and $\left( \ref%
{proof of proposition 5 fn converge obs 1 case}\right) $\ give 
\begin{equation}
\int_{\epsilon }^{T-\epsilon }\int_{\Omega \cap B_{R_{1}}}\left\vert
-2\nabla \chi \nabla v_{n}-v_{n}\Delta \chi +\frac{1}{\sigma _{n}}\chi
g_{n}\right\vert ^{2}dxds\underset{n\rightarrow +\infty }{\longrightarrow }0.
\label{grad exterior conv}
\end{equation}%
Combining the estimate above with $\left( \ref{proof of proposition 5 fn
converge obs 1 case}\right) ,$ we get%
\begin{equation}
\begin{array}{c}
E_{W_{n}}\left( \epsilon \right) \underset{n\rightarrow +\infty }{%
\longrightarrow }0,%
\end{array}
\label{epsilon energy conv}
\end{equation}%
for all $\epsilon >0$ small enough, such that $\left( \omega \cap
B_{R_{1}},T-2\epsilon \right) $ geometrically controls $\Omega \cap
B_{R_{1}}.~$On the other hand the energy estimate for the nonhomogeneous
wave equation, gives 
\begin{equation}
\begin{array}{l}
E_{W_{n}}\left( s\right) \\ 
\leq 2e^{T}\left( E_{W_{n}}\left( \epsilon \right) +\dint_{\epsilon
}^{T-\epsilon }\dint_{\Omega \cap B_{R_{1}}}\left\vert -2\nabla \chi \nabla
v_{n}-v_{n}\Delta \chi +\frac{1}{\sigma _{n}}\chi g_{n}\right\vert
^{2}dxdt\right) ,%
\end{array}
\label{e Wn s convergence}
\end{equation}%
for $\epsilon \leq s\leq T-\epsilon .~$Using $\left( \ref{grad exterior conv}%
\right) $ and $\left( \ref{epsilon energy conv}\right) ,$ we see that 
\begin{equation*}
\begin{array}{l}
E_{W_{n}}\left( s\right) \\ 
\leq C\left( E_{W_{n}}\left( \epsilon \right) +\dint_{\epsilon }^{T-\epsilon
}\dint_{\Omega \cap B_{R_{1}}}\left\vert -2\nabla \chi \nabla
v_{n}-v_{n}\Delta \chi +\frac{1}{\sigma _{n}}\chi g_{n}\right\vert
^{2}dxdt\right) \underset{n\rightarrow +\infty }{\longrightarrow }0,%
\end{array}%
\end{equation*}%
for all $s\in \left[ \epsilon ,T-\epsilon \right] .$ We conclude that 
\begin{equation}
\begin{array}{c}
E_{W_{n}}\left( s\right) \underset{n\rightarrow +\infty }{\longrightarrow }0,%
\text{ for all }0<s<T.%
\end{array}
\label{1}
\end{equation}%
Using $\left( \ref{energy Zn bounded1}\right) $ and applying the dominated
convergence theorem, we obtain $\left( \ref{integral convergence to zero}%
\right) .$

Now $\left( \ref{integral convergence to zero}\right) $\ and the fact that $%
\chi =1$ on $\left\{ \left\vert x\right\vert \leq R\right\} $ along with $%
\left( \ref{a vn convergence}\right) ,$ give 
\begin{equation}
\int_{0}^{T}\int_{\Omega \cap B_{R}}\left\vert \nabla v_{n}\left( s\right)
\right\vert ^{2}+\left\vert \partial _{t}v_{n}\left( s\right) \right\vert
^{2}dxds\underset{n\rightarrow +\infty }{\longrightarrow }0.
\label{vn convergence Br}
\end{equation}%
On the other hand let $\theta \in C_{c}^{\infty }\left( 
\mathbb{R}
^{d}\right) $ such that $\theta =1$ on $\left\{ \left\vert x\right\vert \leq
R_{0}\right\} $ and the support of $\theta $ is contained in $\left\{
\left\vert x\right\vert \leq R\right\} .$ Using Poincar\'{e}'s inequality,
we obtain 
\begin{eqnarray*}
&&\int_{0}^{T}\int_{\Omega \cap B_{R_{0}}}\left\vert v_{n}\left( s,x\right)
\right\vert ^{2}dxds \\
&\leq &C\int_{0}^{T}\int_{\Omega \cap B_{R}}\left\vert v_{n}\left(
s,x\right) \nabla \theta \left( x\right) \right\vert ^{2}+\left\vert \theta
\left( x\right) \nabla v_{n}\left( s,x\right) \right\vert ^{2}dxds.
\end{eqnarray*}%
The estimate above combined with $\left( \ref{vn convergence Br}\right) $
and $\left( \ref{a vn convergence}\right) $, give 
\begin{equation*}
\int_{0}^{T}\int_{\Omega \cap B_{R_{0}}}\left\vert v_{n}\left( s\right)
\right\vert ^{2}+\left\vert \nabla v_{n}\left( s\right) \right\vert
^{2}+\left\vert \partial _{t}v_{n}\left( s\right) \right\vert ^{2}dxds%
\underset{n\rightarrow +\infty }{\longrightarrow }0.
\end{equation*}%
The contradiction follows from the fact that%
\begin{eqnarray*}
1 &=&\frac{1}{\sigma _{n}^{2}}\int_{t_{n}}^{t_{n}+T}\int_{\Omega \cap
B_{R_{0}}}\varphi \left( s+t_{n}+q\left( x\right) \right) \left( \left\vert
u_{n}\right\vert ^{2}+\left\vert \nabla u_{n}\right\vert ^{2}+\left\vert
\partial _{t}u_{n}\right\vert ^{2}\right) dxds \\
&\leq &C\int_{0}^{T}\int_{\Omega \cap B_{R_{0}}}\left( \left\vert
v_{n}\right\vert ^{2}+\left\vert \nabla v_{n}\right\vert ^{2}+\left\vert
\partial _{t}v_{n}\right\vert ^{2}\right) dxds\underset{n\rightarrow +\infty 
}{\longrightarrow }0.
\end{eqnarray*}

\item[Second case ] The sequence $t_{n}\underset{n\rightarrow +\infty }{%
\longrightarrow }+\infty .~$We set%
\begin{equation*}
\begin{array}{l}
\sigma _{n}^{2}=\dint_{t_{n}}^{t_{n}+T}\dint_{\Omega \cap B_{R_{0}}}\varphi
\left( q\left( x\right) +s\right) \left( \left\vert u_{n}\right\vert
^{2}+\left\vert \nabla u_{n}\right\vert ^{2}+\left\vert \partial
_{t}u_{n}\right\vert ^{2}\right) dxds \\ 
\text{ and }v_{n}\left( t,x\right) =\frac{\left( \varphi \left( q\left(
x\right) +t_{n}+t\right) \right) ^{\frac{1}{2}}u_{n}\left( t_{n}+t,x\right) 
}{\sigma _{n}}.%
\end{array}%
\end{equation*}%
From $\left( \ref{proof proposition 1 contradiction estimate obs}\right) $,
we infer that%
\begin{equation}
\begin{array}{l}
\frac{1}{\sigma _{n}^{2}}\dint_{t_{n}}^{t_{n}+T}\dint_{\Omega }\varphi
\left( q\left( x\right) +t\right) \left( \left\vert \nabla u_{n}\left(
t\right) \right\vert ^{2}+\left\vert \partial _{t}u_{n}\left( t\right)
\right\vert ^{2}\right) dxdt\leq \frac{1}{\delta } \\ 
\text{and }\dint_{t_{n}}^{t_{n}+T}\dint_{\Omega \cap B_{R_{0}}}\left\vert
v_{n}\left( t\right) \right\vert ^{2}dxdt\leq 1,%
\end{array}
\label{proof of proposition  contradiction obs}
\end{equation}%
and%
\begin{equation}
\begin{array}{l}
\frac{1}{\sigma _{n}^{2}}\dint_{t_{n}}^{t_{n}+T}\dint_{\Omega }a\left(
x\right) \varphi \left( q\left( x\right) +s\right) \left\vert \partial
_{t}u_{n}\right\vert ^{2}dxds\underset{n\rightarrow +\infty }{%
\longrightarrow }0 \\ 
\frac{1}{\sigma _{n}^{2}}\dint_{t_{n}}^{t_{n}+T}\dint_{\Omega }\varphi
\left( q\left( x\right) +s\right) \left\vert g_{n}\left( s,x\right)
\right\vert ^{2}dxds\underset{n\rightarrow +\infty }{\longrightarrow }0 \\ 
\frac{1}{\sigma _{n}^{2}}\dint_{t_{n}}^{t_{n}+T}\dint_{\Omega }\frac{\left(
\varphi ^{\prime }\left( q\left( x\right) +s\right) \right) ^{2}}{\varphi
\left( q\left( x\right) +s\right) }a\left( x\right) \left\vert
u_{n}\right\vert ^{2}dxds\underset{n\rightarrow +\infty }{\longrightarrow }0.%
\end{array}
\label{vn contradiction}
\end{equation}%
It is clear that $v_{n}$ is a solution of the following system%
\begin{equation*}
\left\{ 
\begin{array}{lc}
\partial _{t}^{2}v_{n}-\Delta v_{n}=f_{n}\left( t,x\right) & \text{in }%
\mathbb{R}
_{+}\times \Omega , \\ 
v_{n}\left( t,x\right) =0 & \text{on }%
\mathbb{R}
_{+}\times \Gamma , \\ 
\left( v_{n}\left( 0\right) ,\partial _{t}v_{n}\left( 0\right) \right) \in
H_{0}^{1}\left( \Omega \right) \times L^{2}\left( \Omega \right) , & 
\end{array}%
\right.
\end{equation*}%
where%
\begin{equation*}
\begin{array}{l}
\medskip f_{n}\left( t,x\right) =\frac{1}{2\sigma _{n}}\left[ \left( \varphi
^{\prime \prime }\left( \varphi \right) ^{-\frac{1}{2}}-\frac{1}{2}\left(
\varphi ^{\prime }\right) ^{2}\varphi ^{-3/2}\right) \frac{\left\vert
x\right\vert ^{2}}{q^{2}}\right] u_{n}\left( t_{n}+t\right) \\ 
+\frac{1}{2\sigma _{n}}\left[ \left( \frac{d}{q}-\frac{\left\vert
x\right\vert ^{2}}{q^{3}}\right) \varphi ^{\prime }\left( \varphi \right) ^{-%
\frac{1}{2}}\right] u_{n}\left( t_{n}+t\right) \\ 
\medskip +\frac{1}{2\sigma _{n}}\left[ \varphi ^{\prime \prime }\left(
\varphi \right) ^{-\frac{1}{2}}-\frac{1}{2}\left( \varphi ^{\prime }\right)
^{2}\varphi ^{-3/2}\right] u_{n}\left( t_{n}+t\right) -\frac{1}{\sigma _{n}}%
\varphi ^{\frac{1}{2}}g_{n}\left( t_{n}+t,x\right) \\ 
+\frac{\varphi ^{\prime }\left( \varphi \right) ^{-\frac{1}{2}}}{\sigma _{n}}%
\left( \partial _{t}u_{n}\left( t_{n}+t\right) +\frac{x\cdot \nabla
u_{n}\left( t_{n}+t\right) }{q}\right) ,%
\end{array}%
\end{equation*}%
where $\varphi ^{\left( j\right) }\left( t,x\right) =\varphi ^{\left(
j\right) }\left( q\left( x\right) +t+t_{n}\right) ,$ for $j=0,1,2.$

Now we will show that%
\begin{equation}
\int_{0}^{T}\dint_{\Omega }\left\vert f_{n}\left( s,x\right) \right\vert
^{2}dxds\underset{n\rightarrow +\infty }{\longrightarrow }0.
\label{proof of proposition 5 fn converge obs}
\end{equation}%
Using $\left( \ref{vn contradiction}\right) $ and the fact that $\underset{%
t\rightarrow +\infty }{\lim }\left\vert \frac{\varphi ^{\prime }\left(
t\right) }{\varphi \left( t\right) }\right\vert =0,$ we obtain%
\begin{equation*}
\begin{array}{l}
\medskip \dint_{0}^{T}\dint_{\Omega }\left\vert \frac{1}{2\sigma _{n}}\left[
\left( \varphi ^{\prime \prime }\left( \varphi \right) ^{-\frac{1}{2}}-\frac{%
1}{2}\left( \varphi ^{\prime }\right) ^{2}\varphi ^{-3/2}\right) \frac{%
\left\vert x\right\vert ^{2}}{q^{2}}\right] u_{n}\left( t_{n}+t\right)
\right\vert ^{2}dxdt \\ 
+\dint_{0}^{T}\dint_{\Omega }\left\vert \frac{1}{2\sigma _{n}}\left[ \left( 
\frac{d}{q}-\frac{\left\vert x\right\vert ^{2}}{q^{3}}\right) \varphi
^{\prime }\left( \varphi \right) ^{-\frac{1}{2}}\right] u_{n}\left(
t_{n}+t\right) \right\vert ^{2}dxdt \\ 
+\dint_{0}^{T}\dint_{\Omega }\left\vert \frac{1}{2\sigma _{n}}\left[ \varphi
^{\prime \prime }\left( \varphi \right) ^{-\frac{1}{2}}-\frac{1}{2}\left(
\varphi ^{\prime }\right) ^{2}\varphi ^{-3/2}\right] u_{n}\left(
t_{n}+t\right) \right\vert ^{2}dxdt \\ 
\leq \frac{C}{\sigma _{n}^{2}}\dint_{t_{n}}^{t_{n}+T}\dint_{\Omega }\frac{%
\left( \varphi ^{\prime }\left( q\left( x\right) +s\right) \right) ^{2}}{%
\varphi \left( q\left( x\right) +s\right) }\left( 1+\left( \frac{\varphi
^{\prime }\left( t_{n}\right) }{\varphi \left( t_{n}\right) }\right)
^{2}\right) \left\vert u_{n}\right\vert ^{2}dxds \\ 
\leq \frac{C}{\sigma _{n}^{2}}\left( \frac{\varphi ^{\prime }\left(
t_{n}\right) }{\varphi \left( t_{n}\right) }\right)
^{2}\dint_{t_{n}}^{t_{n}+T}\int_{\Omega \cap B_{L}}\varphi \left( q\left(
x\right) +s\right) \left\vert u_{n}\right\vert ^{2}dxds \\ 
+\frac{C}{\epsilon _{0}\sigma _{n}^{2}}\dint_{t_{n}}^{t_{n}+T}\dint_{\Omega }%
\frac{\left( \varphi ^{\prime }\left( q\left( x\right) +s\right) \right) ^{2}%
}{\varphi \left( q\left( x\right) +s\right) }a\left( x\right) \left\vert
u_{n}\right\vert ^{2}dxds \\ 
\leq C\left( \frac{\varphi ^{\prime }\left( t_{n}\right) }{\varphi \left(
t_{n}\right) }\right) ^{2}+\frac{C}{\epsilon _{0}\sigma _{n}^{2}}%
\dint_{t_{n}}^{t_{n}+T}\dint_{\Omega }\frac{\left( \varphi ^{\prime }\left(
q\left( x\right) +s\right) \right) ^{2}}{\varphi \left( q\left( x\right)
+s\right) }a\left( x\right) \left\vert u_{n}\right\vert ^{2}dxds\underset{%
n\rightarrow +\infty }{\longrightarrow }0.%
\end{array}%
\end{equation*}%
Now we estimate the remaining term of $f_{n}.~$\ Turn into account of $%
\left( \ref{proof of proposition contradiction obs}\right) ,$ we get,%
\begin{equation*}
\begin{array}{l}
\medskip \dint_{0}^{T}\dint_{\Omega }\left\vert \frac{\varphi ^{\prime
}\left( \varphi \right) ^{-\frac{1}{2}}}{\sigma _{n}}\left( \partial
_{t}u_{n}\left( t_{n}+t\right) +\frac{x\cdot \nabla u_{n}\left(
t_{n}+t\right) }{q}\right) \right\vert ^{2}dxdt \\ 
\medskip \leq \frac{C}{\sigma _{n}^{2}}\left( \frac{\varphi ^{\prime }\left(
t_{n}\right) }{\varphi \left( t_{n}\right) }\right)
^{2}\dint_{0}^{T}\dint_{\Omega }\varphi \left( q\left( x\right) +\left(
t_{n}+t\right) \right) \left( \left\vert \partial _{t}u_{n}\left(
t_{n}+t\right) \right\vert ^{2}+\left\vert \nabla u_{n}\left( t_{n}+t\right)
\right\vert ^{2}\right) dxdt \\ 
\leq \frac{C}{\delta }\left( \frac{\varphi ^{\prime }\left( t_{n}\right) }{%
\varphi \left( t_{n}\right) }\right) ^{2}\underset{n\rightarrow +\infty }{%
\longrightarrow }0.%
\end{array}%
\end{equation*}%
The results above combined with $\left( \ref{vn contradiction}\right) $,
gives $\left( \ref{proof of proposition 5 fn converge obs}\right) .$\ 

The next step is to show the boundeness of the energy of $v_{n}.$ It is easy
to see that%
\begin{equation*}
\begin{array}{l}
\medskip \dint_{0}^{T}E_{v_{n}}\left( t\right) dt\leq \frac{c}{\sigma
_{n}^{2}}\dint_{t_{n}}^{t_{n}+T}\dint_{\Omega }\varphi \left( q\left(
x\right) +t\right) \left( \left\vert \nabla u_{n}\left( t\right) \right\vert
^{2}+\left\vert \partial _{t}u_{n}\left( t\right) \right\vert ^{2}\right)
dxdt \\ 
\medskip +\frac{c}{\sigma _{n}^{2}}\dint_{t_{n}}^{t_{n}+T}\dint_{\Omega }%
\frac{\left( \varphi ^{\prime }\left( q\left( x\right) +t\right) \right) ^{2}%
}{\varphi \left( q\left( x\right) +t\right) }\left\vert u_{n}\left( t\right)
\right\vert ^{2}dxdt.%
\end{array}%
\end{equation*}%
Now using $\left( \ref{proof of proposition contradiction obs}\right) $\ and 
$\left( \ref{vn contradiction}\right) $ we infer that there exists a
positive constant $C_{\delta }$ such that 
\begin{equation}
\int_{0}^{T}E_{v_{n}}\left( t\right) dt\leq C_{\delta },\text{ for n large
enough.}  \label{integral energy vn bound}
\end{equation}%
On the other hand, we have%
\begin{equation*}
E_{v_{n}}\left( t\right) \leq \frac{c}{t}\left( \int_{0}^{T}\left(
E_{v_{n}}\left( s\right) +s\int_{\Omega }\left\vert f_{n}\left( s,x\right)
\right\vert ^{2}dx\right) ds\right) ,
\end{equation*}%
for all $0<t\leq T.$ Turn into account of the estimate above along with $%
\left( \ref{integral energy vn bound}\right) $ and $\left( \ref{proof of
proposition 5 fn converge obs}\right) $, we obtain%
\begin{equation}
E_{v_{n}}\left( T\right) \leq C_{T,\delta },\text{ for n large enough.}
\label{bound 1 second case}
\end{equation}%
On the other hand, from the energy identity, we see that%
\begin{equation*}
E_{v_{n}}\left( t\right) \leq E_{v_{n}}\left( T\right) +\int_{0}^{T}\left(
E_{v_{n}}\left( s\right) +\int_{\Omega }\left\vert f_{n}\left( s,x\right)
\right\vert ^{2}dx\right) ds,
\end{equation*}%
for all $0\leq t\leq T.~$The estimate above combined with $\left( \ref%
{integral energy vn bound}\right) $ and $\left( \ref{bound 1 second case}%
\right) $ gives%
\begin{equation}
\underset{\lbrack 0,T]}{\sup }E_{v_{n}}\left( s\right) \leq C_{T,\delta },%
\text{ for n large enough.}  \label{energy vn bound}
\end{equation}%
The last step is to show that 
\begin{equation}
\int_{0}^{T}\dint_{\Omega }a\left( x\right) \left\vert \partial
_{t}v_{n}\right\vert ^{2}dxdt\underset{n\rightarrow +\infty }{%
\longrightarrow }0.  \label{proof of proposition 5 adtun converge obs}
\end{equation}%
We have%
\begin{equation*}
\begin{array}{l}
\dint_{0}^{T}\dint_{\Omega }a\left( x\right) \left\vert \partial
_{t}v_{n}\right\vert ^{2}dxdt\leq \frac{2}{\sigma _{n}^{2}}%
\dint_{t_{n}}^{t_{n}+T}\dint_{\Omega }\frac{\left( \varphi ^{\prime }\left(
q\left( x\right) +s\right) \right) ^{2}}{\varphi \left( q\left( x\right)
+s\right) }a\left( x\right) \left\vert u_{n}\left( s\right) \right\vert
^{2}dxds \\ 
+\frac{2}{\sigma _{n}^{2}}\dint_{t_{n}}^{t_{n}+T}\dint_{\Omega }\varphi
\left( q\left( x\right) +s\right) a\left( x\right) \left\vert \partial
_{t}u_{n}\right\vert ^{2}dxds.%
\end{array}%
\end{equation*}%
Using $\left( \ref{vn contradiction}\right) ,$ we get $\left( \ref{proof of
proposition 5 adtun converge obs}\right) .$

For the rest of the proof we have only to argue as in \cite[Proof of
proposition 2]{Daou linear}.
\end{description}
\end{proof}

\section{Proof of Theorem 1}

\subsection{Preliminary results}

Throughout this section we use the following notations:

Let $\beta $ be a real number such that 
\begin{equation*}
\begin{array}{ll}
\beta >-1 & \text{if }1<r<1+\frac{2}{d} \\ 
-1<\beta <\frac{3-r}{r-1} & \text{if }r=1+\frac{2}{d}.%
\end{array}%
\end{equation*}

Let $\psi \in C_{0}^{\infty }\left( 
\mathbb{R}
^{d}\right) $ such that $0\leq \psi \leq 1$ and 
\begin{equation*}
\psi \left( x\right) =\left\{ 
\begin{array}{ll}
1 & \text{for }\left\vert x\right\vert \leq L \\ 
0 & \text{for }\left\vert x\right\vert \geq 2L%
\end{array}%
\right. .
\end{equation*}

Finally we set%
\begin{equation*}
\begin{array}{l}
\varphi \left( s\right) =\ln ^{\beta +1}\left( b+s\right) ,\text{ }f\left(
s\right) =\frac{\ln ^{\beta }\left( b+s\right) }{b+s},\text{ }f_{1}\left(
s\right) =\frac{\ln ^{\beta }\left( b+s\right) }{\left( b+s\right) ^{2}} \\ 
\text{and }f_{2}\left( s\right) =\frac{\ln ^{\beta -r+1}\left( b+s\right) }{%
\left( b+s\right) ^{r}},%
\end{array}%
\end{equation*}%
with%
\begin{equation*}
\begin{array}{c}
\ln b=\max \left( \left( 2\left( r+1\right) \right) ^{r+1},\frac{\beta +1-r}{%
r-1},\left( 8\left( r+1\right) \left( \beta +1\right) \right) ^{r+1}\right) .%
\end{array}%
\end{equation*}

\begin{proposition}
\label{proposition onservability global log}We assume that Hyp A holds and $%
(\omega $,$T)$ geometrically controls $\Omega $. Let $\beta >-1$. Let $%
\delta >0$ and $R_{0}>L$. There exists $C_{T,\delta }=C\left( T,\delta
,R_{0}\right) >0$, such that the following inequality%
\begin{equation}
\begin{array}{l}
\medskip \dint_{t}^{t+T}\dint_{\Omega \cap B_{R_{0}}}f\left( q\left(
x\right) +s\right) \left( \left\vert u\right\vert ^{2}+\left\vert \nabla
u\right\vert ^{2}+\left\vert \partial _{t}u\right\vert ^{2}\right) dxds \\ 
\medskip \leq C_{T,\delta }\dint_{t}^{t+T}\dint_{\Omega }a\left( x\right)
f\left( q\left( x\right) +s\right) \left( \left\vert \partial
_{t}u\right\vert ^{2}+\left\vert \partial _{t}u\right\vert ^{2r}\right) dxds
\\ 
\medskip +C_{T,\delta }\dint_{t}^{t+T}\dint_{\Omega }a\left( x\right)
f_{1}^{\prime }\left( q\left( x\right) +s\right) \left\vert u\right\vert
^{2}dxds \\ 
+\delta \dint_{t}^{t+T}\dint_{\Omega }f\left( q\left( x\right) +s\right)
\left( \left\vert \nabla u\left( s\right) \right\vert ^{2}+\left\vert
\partial _{t}u\left( s\right) \right\vert ^{2}\right) dxds,%
\end{array}
\label{observability gradient log}
\end{equation}%
holds for every $t\geq 0$ and for all $u$ solution of $\left( \ref{system}%
\right) $ with initial data $\left( u_{0},u_{1}\right) $ in $H_{0}^{1}\left(
\Omega \right) \cap H^{2}\left( \Omega \right) \times H_{0}^{1}\left( \Omega
\right) .$
\end{proposition}

\begin{proof}
In view of $f\in L^{\infty }\left( 
\mathbb{R}
_{+}\right) ,$ we have $E_{f}\left( u\right) \left( 0\right) <\infty .$ On
the other hand, it is clear that $f^{\prime }\in L^{\infty }\left( 
\mathbb{R}
_{+}\right) $ and there exists a positive constant $K,$ such that 
\begin{equation*}
\underset{%
\mathbb{R}
_{+}}{\sup }\left\vert \frac{f^{\prime \prime }\left( t\right) }{f^{\prime
}\left( t\right) }\right\vert \leq K.
\end{equation*}%
In addition the function $t\longmapsto \left\vert \frac{f^{\prime }\left(
t\right) }{f\left( t\right) }\right\vert $ is decreasing and $\underset{%
t\rightarrow +\infty }{\lim }\left\vert \frac{f^{\prime }\left( t\right) }{%
f\left( t\right) }\right\vert =0$. Moreover there exists $C>0,$ such that%
\begin{equation*}
\begin{array}{l}
\frac{\left( f^{\prime }\left( t\right) \right) ^{2}}{f\left( t\right) }\leq
C\text{ }\left( -f_{1}^{\prime }\left( t\right) \right) ,\text{ for all }%
t\geq 0.%
\end{array}%
\end{equation*}%
Since%
\begin{equation*}
\partial _{t}u\in L^{\infty }\left( 
\mathbb{R}
_{+},H_{0}^{1}\left( \Omega \right) \right) ,
\end{equation*}%
therefore, from Sobolev imbedding, we deduce that%
\begin{equation*}
\sqrt{a\left( x\right) f\left( q\left( x\right) +s\right) }\left\vert
\partial _{t}u\right\vert ^{r}\in L_{loc}^{2}\left( 
\mathbb{R}
_{+},L^{2}\left( \Omega \right) \right) .
\end{equation*}%
By taking into account of the results above, we can use proposition \ref%
{proposition onservability global} and we obtain $\left( \ref{observability
gradient log}\right) .$This finishes the proof of the proposition.
\end{proof}

In order to prove theorem 1 we need the following result.

\begin{lemma}
\label{lemma xt log} Let $T>0$ and $u$ be the solution of $\left( \ref%
{system}\right) $ with initial data in $H_{0}^{1}\left( \Omega \right) \cap
H^{2}\left( \Omega \right) \times H_{0}^{1}\left( \Omega \right) $ such that 
\begin{equation}
E_{\varphi }\left( u\right) \left( 0\right) =\dint_{\Omega }\varphi \left(
q\left( x\right) \right) \left( \left\vert \nabla u_{0}\right\vert
^{2}+\left\vert u_{1}\right\vert ^{2}\right) dx<\infty .
\label{lemma xt efi bound log}
\end{equation}%
We set $\chi =1-\psi ~$and%
\begin{equation}
\begin{array}{l}
\medskip X\left( t\right) =\dint_{\Omega }f\left( q\left( x\right) +t\right)
\chi ^{2}\left( x\right) u\left( t\right) \partial _{t}u\left( t\right) dx+%
\frac{k_{1}}{2}\dint_{\Omega }a\left( x\right) f_{1}\left( q\left( x\right)
+t\right) \left\vert u\left( t\right) \right\vert ^{2}dx \\ 
+\dint_{\Omega }a\left( x\right) f_{2}\left( q\left( x\right) +t\right)
\left\vert u\left( t\right) \right\vert ^{r+1}dx+\frac{k}{2}\dint_{\Omega
}\ln ^{\beta +1}\left( b+q\left( x\right) +t\right) \left( \left\vert \nabla
u\right\vert ^{2}+\left\vert \partial _{t}u\right\vert ^{2}\right) dx,%
\end{array}
\label{Xt definition}
\end{equation}%
where%
\begin{equation*}
k=\frac{1}{4\left( \beta +1\right) },\text{ }k_{1}>0.
\end{equation*}%
We have%
\begin{equation}
\begin{array}{l}
\medskip X\left( t+T\right) -X\left( t\right) +\frac{1}{4}%
\dint_{t}^{t+T}\dint_{\Omega }f\left( q\left( x\right) +s\right) \left(
\left\vert \nabla u\right\vert ^{2}+\left\vert \partial _{t}u\right\vert
^{2}\right) dxds \\ 
\medskip -\left( \frac{k_{1}}{4}-\frac{2\left( 1+\left\vert \beta
\right\vert \right) }{\epsilon _{0}}\right) \dint_{t}^{t+T}\dint_{\Omega
}a\left( x\right) f_{1}^{\prime }\left( q\left( x\right) +s\right)
\left\vert u\right\vert ^{2}dxds \\ 
\medskip -\frac{1}{2}\dint_{t}^{t+T}\dint_{\Omega }a\left( x\right)
f_{2}^{\prime }\left( q\left( x\right) +s\right) \left\vert u\right\vert
^{r+1}dxds \\ 
\medskip +\frac{1}{8\left( \beta +1\right) }\dint_{t}^{t+T}\dint_{\Omega
}a\left( x\right) \ln ^{\beta +1}\left( b+q\left( x\right) +s\right)
\left\vert \partial _{t}u\right\vert ^{r+1}dxds \\ 
\medskip \leq \left( 3+\frac{1}{2}\left\Vert \nabla \chi ^{2}\right\Vert
_{\infty }\right) \dint_{t}^{t+T}\dint_{\Omega \cap B_{2L}}f\left( q\left(
x\right) +s\right) \left( \left\vert u\right\vert ^{2}+\left\vert \nabla
u\right\vert ^{2}+\left\vert \partial _{t}u\right\vert ^{2}\right) dxds \\ 
+2\left( \frac{1}{\epsilon _{0}}+\frac{4\left( 1+\left\vert \beta
\right\vert \right) }{\epsilon _{0}^{2}k_{1}}+4k_{1}\right)
\dint_{t}^{t+T}\dint_{\Omega }a\left( x\right) f\left( q\left( x\right)
+t\right) \left\vert \partial _{t}u\right\vert ^{2}dxds.%
\end{array}
\label{X t estimate log}
\end{equation}
\end{lemma}

\begin{proof}
First $\left( \ref{lemma xt efi bound log}\right) $ allows us to apply $%
\left( \ref{weighted energy}\right) \ $and to obtain 
\begin{equation*}
\begin{array}{l}
\medskip E_{\varphi }\left( u\right) \left( t+T\right)
+\dint_{t}^{t+T}\dint_{\Omega }a\left( x\right) \varphi \left( q\left(
x\right) +s\right) \left\vert \partial _{t}u\right\vert ^{r+1}dxds \\ 
\leq E_{\varphi }\left( u\right) \left( t\right) +\left( \beta +1\right)
\dint_{t}^{t+T}\dint_{\Omega }f\left( q\left( x\right) +s\right) \left(
\left\vert \nabla u\right\vert ^{2}+\left\vert \partial _{t}u\right\vert
^{2}\right) dxds.%
\end{array}%
\end{equation*}%
We set%
\begin{equation*}
\begin{array}{l}
\medskip X_{0}\left( t\right) =\dint_{\Omega }f\left( q\left( x\right)
+t\right) \chi ^{2}\left( x\right) u\left( t\right) \partial _{t}u\left(
t\right) dx+\frac{k_{1}}{2}\dint_{\Omega }a\left( x\right) f_{1}\left(
q\left( x\right) +t\right) \left\vert u\left( t\right) \right\vert ^{2}dx \\ 
+\dint_{\Omega }a\left( x\right) f_{2}\left( q\left( x\right) +t\right)
\left\vert u\left( t\right) \right\vert ^{r+1}dx.%
\end{array}%
\end{equation*}%
Therefore, we have%
\begin{equation}
\begin{array}{l}
\medskip \frac{d}{dt}X_{0}\left( t\right) =\dint_{\Omega }\left( \left\vert
\partial _{t}u\left( t\right) \right\vert ^{2}-\left\vert \nabla u\left(
t\right) \right\vert ^{2}-a\left( x\right) \left\vert \partial _{t}u\left(
t\right) \right\vert ^{r-1}u\partial _{t}u\left( t\right) \right) \chi
^{2}\left( x\right) f\left( q\left( x\right) +t\right) dx \\ 
-\dint_{\Omega }\chi ^{2}\left( x\right) f^{\prime }\left( q\left( x\right)
+t\right) u\left( t\right) \frac{x\cdot \nabla u\left( t\right) }{q\left(
x\right) }+f\left( q\left( x\right) +t\right) \nabla \chi ^{2}\left(
x\right) \nabla u\left( t\right) dx \\ 
\medskip +\dint_{\Omega }f^{\prime }\left( q\left( x\right) +t\right) \chi
^{2}\left( x\right) u\left( t\right) \partial _{t}u\left( t\right) dx \\ 
\medskip +k_{1}\left( \dint_{\Omega }a\left( x\right) f_{1}\left( q\left(
x\right) +t\right) u\left( t\right) \partial _{t}u\left( t\right) dx+\frac{1%
}{2}\dint_{\Omega }a\left( x\right) f_{1}^{\prime }\left( q\left( x\right)
+t\right) \left\vert u\left( t\right) \right\vert ^{2}dx\right) \\ 
+\dint_{\Omega }a\left( x\right) f_{2}^{\prime }\left( q\left( x\right)
+t\right) \left\vert u\right\vert ^{r+1}dx+\left( r+1\right) \dint_{\Omega
}a\left( x\right) f_{2}\left( q\left( x\right) +t\right) \left\vert
u\right\vert ^{r-1}u\partial _{t}udx.%
\end{array}
\label{X0 log}
\end{equation}%
A direct computation gives%
\begin{equation*}
\begin{array}{l}
\frac{\left( f^{\prime }\left( s\right) \right) ^{2}}{f\left( s\right) }\leq
\left( 1+\left\vert \beta \right\vert \right) \frac{\ln ^{\beta }\left(
b+s\right) }{\left( b+s\right) ^{3}}\leq -\left( 1+\left\vert \beta
\right\vert \right) f_{1}^{\prime }\left( s\right) \\ 
\text{and} \\ 
\frac{\left( f_{1}\left( s\right) \right) ^{2}}{f\left( s\right) }=\frac{\ln
^{\beta }\left( b+s\right) }{\left( b+s\right) ^{3}}\leq -f_{1}^{\prime
}\left( s\right) .%
\end{array}%
\end{equation*}%
We note that $\left\Vert \chi \right\Vert _{\infty }\leq 1.$ Using Young's
inequality and the fact that the support of $\chi $ is contained in $\left\{
\left\vert x\right\vert \geq L\right\} $ and 
\begin{equation*}
a\left( x\right) >\epsilon _{0}>0\text{ for }\left\vert x\right\vert \geq L,
\end{equation*}%
we deduce that%
\begin{equation*}
\begin{array}{l}
\medskip \left\vert \dint_{\Omega }f^{\prime }\left( q\left( x\right)
+t\right) \chi ^{2}\left( x\right) u\left( t\right) \partial _{t}u\left(
t\right) dx\right\vert \\ 
\leq -\frac{k_{1}}{8}\dint_{\Omega }a\left( x\right) f_{1}^{\prime }\left(
q\left( x\right) +t\right) \left\vert u\left( t\right) \right\vert ^{2}dx+%
\frac{8\left( 1+\left\vert \beta \right\vert \right) }{\epsilon _{0}^{2}k_{1}%
}\dint_{\Omega }a\left( x\right) f\left( q\left( x\right) +t\right)
\left\vert \partial _{t}u\left( t\right) \right\vert ^{2}dx,%
\end{array}%
\end{equation*}%
and%
\begin{equation*}
\begin{array}{l}
\medskip \left\vert k_{1}\dint_{\Omega }a\left( x\right) f_{1}\left( q\left(
x\right) +t\right) u\left( t\right) \partial _{t}u\left( t\right)
dx\right\vert \\ 
\leq -\frac{k_{1}}{8}\dint_{\Omega }a\left( x\right) f_{1}^{\prime }\left(
q\left( x\right) +t\right) \left\vert u\left( t\right) \right\vert
^{2}dx+8k_{1}\dint_{\Omega }a\left( x\right) f\left( q\left( x\right)
+t\right) \left\vert \partial _{t}u\left( t\right) \right\vert ^{2}dx.%
\end{array}%
\end{equation*}%
Using the same arguments we also deduce that%
\begin{equation*}
\begin{array}{l}
\dint_{\Omega }\chi ^{2}\left( x\right) f^{\prime }\left( q\left( x\right)
+t\right) u\left( t\right) \frac{x\cdot \nabla u\left( t\right) }{q\left(
x\right) }dx \\ 
\leq \frac{1}{2}\dint_{\Omega }f\left( q\left( x\right) +t\right) \left\vert
\nabla u\left( t\right) \right\vert ^{2}dx-\frac{2\left( 1+\left\vert \beta
\right\vert \right) }{\epsilon _{0}}\dint_{\Omega }a\left( x\right)
f_{1}^{\prime }\left( q\left( x\right) +t\right) \left\vert u\left( t\right)
\right\vert ^{2}dx.%
\end{array}%
\end{equation*}

Since the support of $\psi $ is contained in $\left\{ \left\vert
x\right\vert \leq 2L\right\} $ and 
\begin{equation*}
a\left( x\right) >\epsilon _{0}\text{ for }\left\vert x\right\vert \geq L,
\end{equation*}%
therefore we see that 
\begin{equation*}
\begin{array}{l}
\medskip \dint_{\Omega }\left( \left\vert \partial _{t}u\left( t\right)
\right\vert ^{2}-\left\vert \nabla u\left( t\right) \right\vert ^{2}\right)
\chi ^{2}\left( x\right) f\left( q\left( x\right) +t\right) dx \\ 
=\dint_{\Omega }f\left( q\left( x\right) +t\right) \left( 1-2\psi \left(
x\right) +\psi ^{2}\left( x\right) \right) \left( \left\vert \partial
_{t}u\left( t\right) \right\vert ^{2}-\left\vert \nabla u\left( t\right)
\right\vert ^{2}\right) dx \\ 
\leq \frac{2}{\epsilon _{0}}\dint_{\Omega }a\left( x\right) f\left( q\left(
x\right) +t\right) \left\vert \partial _{t}u\left( t\right) \right\vert
^{2}dx \\ 
-\dint_{\Omega }f\left( q\left( x\right) +t\right) \left( \left\vert
\partial _{t}u\left( t\right) \right\vert ^{2}+\left\vert \nabla u\left(
t\right) \right\vert ^{2}\right) dx \\ 
+3\dint_{\Omega \cap B_{2L}}f\left( q\left( x\right) +t\right) \left(
\left\vert \partial _{t}u\left( t\right) \right\vert ^{2}+\left\vert \nabla
u\left( t\right) \right\vert ^{2}\right) dx.%
\end{array}%
\end{equation*}%
We note that the support of $\nabla \chi ^{2}$ is contained in $\left\{
\left\vert x\right\vert \leq 2L\right\} $, using Young's inequality, we
deduce that%
\begin{equation*}
\begin{array}{l}
\left\vert -\dint_{\Omega }f\left( q\left( x\right) +t\right) u\left(
t\right) \nabla \chi ^{2}\left( x\right) \nabla u\left( t\right)
dx\right\vert \\ 
\leq \frac{1}{2}\left\Vert \nabla \chi ^{2}\right\Vert _{\infty
}\int_{\Omega \cap B_{2L}}f\left( q\left( x\right) +t\right) \left(
\left\vert u\left( t\right) \right\vert ^{2}+\left\vert \nabla u\left(
t\right) \right\vert ^{2}\right) dx.%
\end{array}%
\end{equation*}%
Since 
\begin{equation*}
\ln b\geq \frac{\beta +1-r}{r-1},
\end{equation*}%
therefore a direct computation gives%
\begin{equation*}
\begin{array}{l}
-f_{2}\left( s\right) \geq \frac{\ln ^{\beta -r+1}\left( b+s\right) }{\left(
b+s\right) ^{r+1}} \\ 
\left( f\left( s\right) \right) ^{r+1}\ln ^{-r\left( \beta +1\right) }\left(
b+s\right) \leq \frac{-f_{2}^{\prime }\left( s\right) }{\ln \left(
b+s\right) } \\ 
\left( f_{2}\left( s\right) \right) ^{\frac{r+1}{r}}\ln ^{-\frac{\beta +1}{r}%
}\left( b+s\right) \leq \frac{-f_{2}^{\prime }\left( s\right) }{\ln \left(
b+s\right) }.%
\end{array}%
\end{equation*}%
Now we can estimate the last term of the RHS of $\left( \ref{X0 log}\right)
.~$H\"{o}lder's inequality along with Young's inequality, leads to%
\begin{equation*}
\begin{array}{l}
\medskip \dint_{\Omega }a\left( x\right) f\left( q\left( x\right) +s\right)
\left\vert \partial _{t}u\left( t\right) \right\vert ^{r-1}u\partial _{t}udx
\\ 
\medskip \leq \left( \ln b\right) ^{-\frac{1}{r+1}}\left( \dint_{\Omega
}a\left( x\right) \ln ^{\beta +1}\left( b+q\left( x\right) +s\right)
\left\vert \partial _{t}u\right\vert ^{r+1}dx\right) ^{\frac{r}{r+1}}\left(
-\dint_{\Omega }a\left( x\right) f_{2}^{\prime }\left( q\left( x\right)
+s\right) \left\vert u\right\vert ^{r+1}dx\right) ^{\frac{1}{r+1}} \\ 
\leq \left( \ln b\right) ^{-\frac{1}{r+1}}\dint_{\Omega }a\left( x\right)
\ln ^{\beta +1}\left( b+q\left( x\right) +s\right) \left\vert \partial
_{t}u\right\vert ^{r+1}dx-\left( \ln b\right) ^{-\frac{1}{r+1}}\dint_{\Omega
}a\left( x\right) f_{2}^{\prime }\left( q\left( x\right) +s\right)
\left\vert u\right\vert ^{r+1}dx,%
\end{array}%
\end{equation*}%
and%
\begin{equation*}
\begin{array}{l}
\medskip \left( r+1\right) \dint_{\Omega }a\left( x\right) f_{2}\left(
q\left( x\right) +s\right) \left\vert u\right\vert ^{r-1}u\partial _{t}udx
\\ 
\medskip \leq \left( r+1\right) \left( \ln b\right) ^{-\frac{r}{r+1}}\left(
\dint_{\Omega }a\left( x\right) \ln ^{\beta +1}\left( b+q\left( x\right)
+s\right) \left\vert \partial _{t}u\right\vert ^{r+1}dx\right) ^{\frac{1}{r+1%
}} \\ 
\medskip \times \left( -\dint_{\Omega }a\left( x\right) f_{2}^{\prime
}\left( q\left( x\right) +s\right) \left\vert u\right\vert ^{r+1}dx\right) ^{%
\frac{r}{r+1}} \\ 
\leq \left( \ln b\right) ^{-\frac{1}{r+1}}\dint_{\Omega }a\left( x\right)
\ln ^{\beta +1}\left( b+q\left( x\right) +s\right) \left\vert \partial
_{t}u\right\vert ^{r+1}dx-r\left( \ln b\right) ^{-\frac{1}{r+1}%
}\dint_{\Omega }a\left( x\right) f_{2}^{\prime }\left( q\left( x\right)
+s\right) \left\vert u\right\vert ^{r+1}dx.%
\end{array}%
\end{equation*}%
Thus%
\begin{equation*}
\begin{array}{l}
\medskip \dint_{t}^{t+T}\dint_{\Omega }a\left( x\right) f\left( q\left(
x\right) +s\right) \left\vert \partial _{t}u\right\vert ^{r-1}u\partial
_{t}udxds+\left( r+1\right) \dint_{t}^{t+T}\dint_{\Omega }a\left( x\right)
f_{2}\left( q\left( x\right) +s\right) \left\vert u\right\vert
^{r-1}u\partial _{t}udxds \\ 
\medskip \leq \left( r+1\right) \left( \ln b\right) ^{-\frac{1}{r+1}%
}\dint_{t}^{t+T}\dint_{\Omega }a\left( x\right) \ln ^{\beta +1}\left(
b+q\left( x\right) +s\right) \left\vert \partial _{t}u\right\vert ^{r+1}dxds
\\ 
-\left( r+1\right) \left( \ln b\right) ^{-\frac{1}{r+1}}\dint_{t}^{t+T}%
\dint_{\Omega }a\left( x\right) f_{2}^{\prime }\left( q\left( x\right)
+s\right) \left\vert u\right\vert ^{r+1}dxdds.%
\end{array}%
\end{equation*}%
Collecting the inequalities above, making some arrangement in $\left( \ref%
{X0 log}\right) $ and integrating the result between $t$ and $t+T$, we end
up with%
\begin{equation*}
\begin{array}{l}
\medskip X\left( t+T\right) -X\left( t\right) +\left( \frac{1}{2}-\left(
1+\beta \right) k\right) \dint_{t}^{t+T}\dint_{\Omega }f\left( q\left(
x\right) +s\right) \left( \left\vert \nabla u\right\vert ^{2}+\left\vert
\partial _{t}u\right\vert ^{2}\right) dxds \\ 
\medskip -\left( \frac{k_{1}}{4}-\frac{2\left( 1+\left\vert \beta
\right\vert \right) }{\epsilon _{0}}\right) \dint_{t}^{t+T}\dint_{\Omega
}a\left( x\right) f_{1}^{\prime }\left( q\left( x\right) +s\right)
\left\vert u\left( s\right) \right\vert ^{2}dxds \\ 
\medskip -\left( 1-\left( r+1\right) \left( \ln b\right) ^{-\frac{1}{r+1}%
}\right) \dint_{t}^{t+T}\dint_{\Omega }a\left( x\right) f_{2}^{\prime
}\left( q\left( x\right) +s\right) \left\vert u\right\vert ^{r+1}dxds \\ 
\medskip +\left( k-\left( r+1\right) \left( \ln b\right) ^{-\frac{1}{r+1}%
}\right) \dint_{t}^{t+T}\dint_{\Omega }a\left( x\right) \ln ^{\beta
+1}\left( b+q\left( x\right) +s\right) \left\vert \partial _{t}u\right\vert
^{r+1}dxds \\ 
\medskip \leq \left( 3+\frac{1}{2}\left\Vert \nabla \chi ^{2}\right\Vert
_{\infty }\right) \left( \dint_{t}^{t+T}\dint_{\Omega \cap B_{2L}}f\left(
q\left( x\right) +s\right) \left( \left\vert u\right\vert ^{2}+\left\vert
\nabla u\right\vert ^{2}+\left\vert \partial _{t}u\right\vert ^{2}\right)
dxds\right) \\ 
+\left( \frac{2}{\epsilon _{0}}+\frac{8\left( 1+\left\vert \beta \right\vert
\right) }{\epsilon _{0}^{2}k_{1}}+8k_{1}\right) \dint_{t}^{t+T}\dint_{\Omega
}a\left( x\right) f\left( q\left( x\right) +s\right) \left\vert \partial
_{t}u\right\vert ^{2}dxds.%
\end{array}%
\end{equation*}%
Using the fact that $k=\frac{1}{4\left( \beta +1\right) }$ and 
\begin{equation*}
\begin{array}{l}
\ln b\geq \max \left( \left( 2\left( r+1\right) \right) ^{r+1},\left(
8\left( r+1\right) \left( \beta +1\right) \right) ^{r+1}\right) ,%
\end{array}%
\end{equation*}%
we obtain $\left( \ref{X t estimate log}\right) .$
\end{proof}

\subsection{Proof of Theorem 1}

We assume that Hyp A holds and $\omega $ satisfies the GCC. We set $\gamma
=\beta +1.$ Let $u$ be a solution of $\left( \ref{system}\right) $ with
initial data in $H_{0}^{1}\left( \Omega \right) \cap H^{2}\left( \Omega
\right) \times H_{0}^{1}\left( \Omega \right) $ such that 
\begin{equation*}
E_{\varphi }\left( u\right) \left( 0\right) =\dint_{\Omega }\ln ^{\beta
+1}\left( 1+q\left( x\right) \right) \left( \left\vert \nabla
u_{0}\right\vert ^{2}+\left\vert u_{1}\right\vert ^{2}\right) dx<\infty .
\end{equation*}%
Let $T>0$ such that the observability estimate $\left( \ref{observability
gradient log}\right) $ holds. First we estimate the first term of the RHS of 
$\left( \ref{X t estimate log}\right) .$\ Using the observability estimate $%
\left( \ref{observability gradient log}\right) $, we see that%
\begin{equation}
\begin{array}{l}
\medskip X\left( t+T\right) -X\left( t\right) +\left( \frac{1}{4}-\left(
3+\left\Vert \nabla \chi ^{2}\right\Vert _{\infty }\right) \delta \right)
\dint_{t}^{t+T}\dint_{\Omega }f\left( q\left( x\right) +s\right) \left(
\left\vert \nabla u\right\vert ^{2}+\left\vert \partial _{t}u\right\vert
^{2}\right) dxds \\ 
\medskip -\left( \frac{k_{1}}{4}-\frac{2\left( 1+\left\vert \beta
\right\vert \right) }{\epsilon _{0}}-\left( 3+\left\Vert \nabla \chi
^{2}\right\Vert _{\infty }\right) C_{T,\delta }\right)
\dint_{t}^{t+T}\dint_{\Omega }a\left( x\right) f_{1}^{\prime }\left( q\left(
x\right) +s\right) \left\vert u\right\vert ^{2}dxds \\ 
\medskip -\frac{1}{2}\dint_{t}^{t+T}\dint_{\Omega }a\left( x\right)
f_{2}^{\prime }\left( q\left( x\right) +s\right) \left\vert u\right\vert
^{r+1}dxds \\ 
\medskip +\frac{1}{8\left( \beta +1\right) }\dint_{t}^{t+T}\dint_{\Omega
}a\left( x\right) \ln ^{\beta +1}\left( b+q\left( x\right) +s\right)
\left\vert \partial _{t}u\right\vert ^{r+1}dxds \\ 
\medskip \leq k_{3}\dint_{t}^{t+T}\dint_{\Omega }a\left( x\right) f\left(
q\left( x\right) +s\right) \left( \left\vert \partial _{t}u\right\vert
^{2}+\left\vert \partial _{t}u\right\vert ^{2r}\right) dxds,%
\end{array}
\label{proof thm 1 estimate prefinal log}
\end{equation}%
for every $t\geq 0,$ where $k_{3}=2\left( \frac{1}{\epsilon _{0}}+\frac{%
4\left( 1+\left\vert \beta \right\vert \right) }{\epsilon _{0}^{2}k_{1}}%
+4k_{1}+2\left( 3+\left\Vert \nabla \chi ^{2}\right\Vert _{\infty }\right)
C_{T,\delta }\right) .$

On the other hand, using Young's inequality we get%
\begin{equation}
\begin{array}{l}
\medskip X\left( t\right) \leq \left( \frac{k_{1}}{2}+\frac{1}{\epsilon
_{0}\epsilon }\right) \dint_{\Omega }a\left( x\right) f_{1}\left( q\left(
x\right) +t\right) \left\vert u\left( t\right) \right\vert ^{2}dx \\ 
\medskip +\left( k+\epsilon \right) \dint_{\Omega }\ln ^{\beta +1}\left(
b+q\left( x\right) +t\right) \left( \left\vert \nabla u\left( t\right)
\right\vert ^{2}+\left\vert \partial _{t}u\left( t\right) \right\vert
^{2}\right) dx \\ 
+\dint_{\Omega }a\left( x\right) f_{2}\left( q\left( x\right) +t\right)
\left\vert u\left( t\right) \right\vert ^{r+1}dx%
\end{array}
\label{k estimate sup log}
\end{equation}%
and%
\begin{equation}
\begin{array}{l}
\medskip X\left( t\right) \geq \left( \frac{k_{1}}{2}-\frac{1}{\epsilon
_{0}\epsilon }\right) \dint_{\Omega }a\left( x\right) f_{1}\left( q\left(
x\right) +t\right) \left\vert u\left( t\right) \right\vert ^{2}dx \\ 
\medskip +\left( k-\epsilon \right) \dint_{\Omega }\ln ^{\beta +1}\left(
b+q\left( x\right) +t\right) \left( \left\vert \nabla u\left( t\right)
\right\vert ^{2}+\left\vert \partial _{t}u\left( t\right) \right\vert
^{2}\right) dx \\ 
+\dint_{\Omega }a\left( x\right) f_{2}\left( q\left( x\right) +t\right)
\left\vert u\left( t\right) \right\vert ^{r+1}dx,%
\end{array}
\label{k estimate inf log}
\end{equation}%
for all $\epsilon >0.$ We choose (by taking into account of the order below)%
\begin{equation*}
\begin{array}{l}
\medskip \delta \text{ such that }\frac{1}{4}-\left( 3+\left\Vert \nabla
\chi ^{2}\right\Vert _{\infty }\right) \delta =\frac{1}{8}, \\ 
\medskip \epsilon \text{ such that }k-\epsilon \geq \frac{1}{16\left( \beta
+1\right) }, \\ 
k_{1}\text{ such that }\frac{k_{1}}{2}-\frac{1}{\epsilon _{0}\epsilon }\geq 1%
\text{ and }\frac{k_{1}}{4}-\frac{2\left( 1+\left\vert \beta \right\vert
\right) }{\epsilon _{0}}-\left( 3+\left\Vert \nabla \chi ^{2}\right\Vert
_{\infty }\right) C_{T,\delta }\geq 1.%
\end{array}%
\end{equation*}%
Therefore%
\begin{equation}
\begin{array}{l}
\medskip X\left( t\right) \geq \dint_{\Omega }a\left( x\right) f_{1}\left(
q\left( x\right) +t\right) \left\vert u\left( t\right) \right\vert ^{2}dx \\ 
\medskip +\frac{1}{16\left( \beta +1\right) }\dint_{\Omega }\ln ^{\beta
+1}\left( b+q\left( x\right) +t\right) \left( \left\vert \nabla u\left(
t\right) \right\vert ^{2}+\left\vert \partial _{t}u\left( t\right)
\right\vert ^{2}\right) dx \\ 
+\dint_{\Omega }a\left( x\right) f_{2}\left( q\left( x\right) +t\right)
\left\vert u\left( t\right) \right\vert ^{r+1}dx.%
\end{array}
\label{k estimate inf log final}
\end{equation}%
and%
\begin{equation}
\begin{array}{l}
\medskip X\left( t+T\right) -X\left( t\right) +\frac{1}{8}%
\dint_{t}^{t+T}\dint_{\Omega }f\left( q\left( x\right) +s\right) \left(
\left\vert \nabla u\right\vert ^{2}+\left\vert \partial _{t}u\right\vert
^{2}\right) dxds \\ 
\medskip -\dint_{t}^{t+T}\dint_{\Omega }a\left( x\right) f_{1}^{\prime
}\left( q\left( x\right) +s\right) \left\vert u\right\vert ^{2}dxds-\frac{1}{%
2}\dint_{t}^{t+T}\dint_{\Omega }a\left( x\right) f_{2}^{\prime }\left(
q\left( x\right) +s\right) \left\vert u\right\vert ^{r+1}dxds \\ 
\medskip +\frac{1}{8\left( \beta +1\right) }\dint_{t}^{t+T}\dint_{\Omega
}a\left( x\right) \ln ^{\beta +1}\left( b+q\left( x\right) +s\right)
\left\vert \partial _{t}u\right\vert ^{r+1}dxds \\ 
\leq k_{3}\dint_{t}^{t+T}\dint_{\Omega }a\left( x\right) f\left( q\left(
x\right) +s\right) \left( \left\vert \partial _{t}u\right\vert
^{2}+\left\vert \partial _{t}u\right\vert ^{2r}\right) dxds,%
\end{array}
\label{Xt final estimate log}
\end{equation}%
for every $t\geq 0.$ Thus%
\begin{equation}
\begin{array}{l}
\medskip X\left( nT\right) +\frac{1}{8}\dint_{0}^{nT}\dint_{\Omega }f\left(
q\left( x\right) +s\right) \left( \left\vert \nabla u\right\vert
^{2}+\left\vert \partial _{t}u\right\vert ^{2}\right) dxds \\ 
\medskip -\dint_{0}^{nT}\dint_{\Omega }a\left( x\right) f_{1}^{\prime
}\left( q\left( x\right) +s\right) \left\vert u\right\vert ^{2}dxds-\frac{1}{%
2}\dint_{0}^{nT}\dint_{\Omega }a\left( x\right) f_{2}^{\prime }\left(
q\left( x\right) +s\right) \left\vert u\right\vert ^{r+1}dxds \\ 
\medskip +\frac{1}{8\left( \beta +1\right) }\dint_{0}^{nT}\dint_{\Omega
}a\left( x\right) \ln ^{\beta +1}\left( b+q\left( x\right) +s\right)
\left\vert \partial _{t}u\right\vert ^{r+1}dxds \\ 
\leq k_{3}\dint_{0}^{nT}\dint_{\Omega }a\left( x\right) f\left( q\left(
x\right) +s\right) \left( \left\vert \partial _{t}u\right\vert
^{2}+\left\vert \partial _{t}u\right\vert ^{2r}\right) dxds+X\left( 0\right)
,\text{ for all }n\in 
\mathbb{N}
.%
\end{array}
\label{Xt final estimate1 log}
\end{equation}%
Using proposition \ref{proposition weighted energy}, we deduce that%
\begin{equation}
X\left( 0\right) \leq CI_{0}  \label{proof theorem 1 XT bound log}
\end{equation}%
where $I_{0}$ is defined in the statement of theorem 1.

Combining $\left( \ref{Xt final estimate1 log}\right) $ and $\left( \ref%
{proof theorem 1 XT bound log}\right) ,$ we obtain%
\begin{equation}
\begin{array}{l}
\medskip X\left( nT\right) +\frac{1}{8}\dint_{0}^{nT}\dint_{\Omega }f\left(
q\left( x\right) +s\right) \left( \left\vert \nabla u\right\vert
^{2}+\left\vert \partial _{t}u\right\vert ^{2}\right) dxds \\ 
\medskip -\dint_{0}^{nT}\dint_{\Omega }a\left( x\right) f_{1}^{\prime
}\left( q\left( x\right) +s\right) \left\vert u\right\vert ^{2}dxds-\frac{1}{%
2}\dint_{0}^{nT}\dint_{\Omega }a\left( x\right) f_{2}^{\prime }\left(
q\left( x\right) +s\right) \left\vert u\right\vert ^{r+1}dxds \\ 
\medskip +\frac{1}{8\left( \beta +1\right) }\dint_{0}^{nT}\dint_{\Omega
}a\left( x\right) \ln ^{\beta +1}\left( b+q\left( x\right) +s\right)
\left\vert \partial _{t}u\right\vert ^{r+1}dxds \\ 
\leq k_{4}\left( \dint_{0}^{nT}\dint_{\Omega }a\left( x\right) f\left(
q\left( x\right) +s\right) \left( \left\vert \partial _{t}u\right\vert
^{2}+\left\vert \partial _{t}u\right\vert ^{2r}\right) dxds+I_{0}\right) ,%
\text{ for all }n\in 
\mathbb{N}
.%
\end{array}
\label{log aux}
\end{equation}%
for some $k_{4}>0.$ The next step is to control the first term of the RHS of
\ the estimate above by the last term of the LHS. We remind that%
\begin{equation*}
p=\left\{ 
\begin{array}{ll}
2\left( r+1\right) & \text{if }d\leq 3 \\ 
\frac{2d}{d-2} & \text{if }d\geq 4.%
\end{array}%
\right.
\end{equation*}%
We have $r+1<2r<p$, using interpolation inequality and Young's inequality,
we obtain%
\begin{equation*}
\begin{array}{l}
\dint_{0}^{nT}\dint_{\Omega }a\left( x\right) f\left( q\left( x\right)
+s\right) \left\vert \partial _{t}u\right\vert ^{2r}dxds \\ 
\medskip \leq \dint_{0}^{nT}f\left( s\right) \dint_{\Omega }a\left( x\right)
\left\vert \partial _{t}u\right\vert ^{2r}dxds \\ 
\medskip \leq \dint_{0}^{nT}f\left( s\right) \left( \dint_{\Omega }a\left(
x\right) \left\vert \partial _{t}u\right\vert ^{r+1}dx\right) ^{\frac{p-2r}{%
p-r-1}}\left( \dint_{\Omega }a\left( x\right) \left\vert \partial
_{t}u\right\vert ^{p}dx\right) ^{\frac{r-1}{p-r-1}}ds \\ 
\medskip \leq \left( \left\Vert a\right\Vert _{L^{\infty }}\left\Vert
\partial _{t}u\right\Vert _{L^{\infty }\left( 
\mathbb{R}
_{+},L^{p}\left( \Omega \right) \right) }^{p}\dint_{0}^{nT}\left( f\left(
s\right) \right) ^{\frac{p-r-1}{r-1}}\left( \ln \left( b+s\right) \right) ^{-%
\frac{\left( \beta +1\right) \left( p-2r\right) }{r-1}}ds\right) ^{\frac{r-1%
}{p-r-1}} \\ 
\medskip \times \left( \dint_{0}^{nT}\ln ^{\beta +1}\left( b+s\right)
\dint_{\Omega }a\left( x\right) \left\vert \partial _{t}u\right\vert
^{r+1}dxds\right) ^{\frac{p-2r}{p-r-1}} \\ 
\medskip \leq \frac{\epsilon ^{-\frac{p-2r}{r-1}}\left( r-1\right)
\left\Vert a\right\Vert _{L^{\infty }}\left\Vert \partial _{t}u\right\Vert
_{L^{\infty }\left( 
\mathbb{R}
_{+},L^{p}\left( \Omega \right) \right) }^{p}}{p-r-1}\dint_{0}^{+\infty
}\left( b+s\right) ^{-\frac{p-r-1}{r-1}}\left( \ln \left( b+s\right) \right)
^{\beta -\frac{p-2r}{r-1}}ds \\ 
+\frac{\epsilon \left( p-2r\right) }{p-r-1}\dint_{0}^{nT}\dint_{\Omega
}a\left( x\right) \ln ^{\beta +1}\left( b+q\left( x\right) +s\right)
\left\vert \partial _{t}u\right\vert ^{r+1}dxds,%
\end{array}%
\end{equation*}%
for all $\epsilon >0$. Thus using $\left( \ref{hight energy inequality}%
\right) $ and Sobolev imbedding $H^{1}\hookrightarrow L^{p}$, we get%
\begin{equation}
\begin{array}{l}
\medskip \dint_{0}^{nT}\dint_{\Omega }a\left( x\right) f\left( q\left(
x\right) +s\right) \left\vert \partial _{t}u\right\vert ^{2r}dxds \\ 
\medskip \leq \epsilon ^{-\frac{p-2r}{r-1}}C\left\Vert a\right\Vert
_{L^{\infty }}\left( \left\Vert u_{0}\right\Vert _{H^{2}}^{2}+\left\Vert
u_{1}\right\Vert _{H^{1}}^{2}+\left\Vert u_{1}\right\Vert
_{H^{1}}^{2r}\right) ^{\frac{p}{2}} \\ 
+\frac{\epsilon \left( p-2r\right) }{p-r-1}\dint_{0}^{nT}\dint_{\Omega
}a\left( x\right) \left( \ln \left( b+q\left( x\right) +s\right) \right)
^{\beta +1}\left\vert \partial _{t}u\right\vert ^{r+1}dxds,%
\end{array}
\label{log 1}
\end{equation}%
for all $\epsilon >0$. To estimate the last term, first we use Holder's
inequality%
\begin{equation*}
\begin{array}{l}
\medskip \dint_{0}^{nT}\dint_{\Omega }a\left( x\right) f\left( q\left(
x\right) +s\right) \left\vert \partial _{t}u\right\vert ^{2}dxds \\ 
\medskip \leq \left( \left\Vert a\right\Vert _{L^{\infty
}}\dint_{0}^{nT}\dint_{\Omega }\left( f\left( q\left( x\right) +s\right)
\right) ^{\frac{r+1}{r-1}}\ln ^{-\frac{2\left( \beta +1\right) }{r-1}}\left(
b+q\left( x\right) +s\right) dxds\right) ^{\frac{r-1}{r+1}} \\ 
\medskip \times \left( \dint_{0}^{nT}\dint_{\Omega }a\left( x\right) \ln
^{\beta +1}\left( b+q\left( x\right) +s\right) \left\vert \partial
_{t}u\right\vert ^{r+1}dxds\right) ^{\frac{2}{r+1}} \\ 
\medskip \leq \left( \left\Vert a\right\Vert _{L^{\infty
}}\dint_{0}^{+\infty }\dint_{\Omega }\left( b+q\left( x\right) +s\right) ^{-%
\frac{r+1}{r-1}}\ln ^{\beta -\frac{2}{r-1}}\left( b+q\left( x\right)
+s\right) dxds\right) ^{\frac{r-1}{r+1}} \\ 
\times \left( \dint_{0}^{nT}\dint_{\Omega }a\left( x\right) \ln ^{\beta
+1}\left( b+q\left( x\right) +s\right) \left\vert \partial _{t}u\right\vert
^{r+1}dxds\right) ^{\frac{2}{r+1}}.%
\end{array}%
\end{equation*}%
By Young's inequality, we end up with%
\begin{equation*}
\begin{array}{l}
\medskip \dint_{0}^{nT}\dint_{\Omega }a\left( x\right) f\left( q\left(
x\right) +s\right) \left\vert \partial _{t}u\right\vert ^{2}dxds \\ 
\medskip \leq \frac{\left( r-1\right) \epsilon ^{-\frac{2}{r-1}}\left\Vert
a\right\Vert _{L^{\infty }}}{r+1}\dint_{0}^{+\infty }\dint_{\Omega }\left(
b+q\left( x\right) +s\right) ^{-\frac{r+1}{r-1}}\ln ^{\beta -\frac{2}{r-1}%
}\left( b+q\left( x\right) +s\right) dxds \\ 
\medskip +\frac{2\epsilon }{r+1}\dint_{0}^{nT}\dint_{\Omega }a\left(
x\right) \ln ^{\beta +1}\left( b+q\left( x\right) +s\right) \left\vert
\partial _{t}u\right\vert ^{r+1}dxds \\ 
\medskip \leq C\left\Vert a\right\Vert _{L^{\infty }}\frac{\left( r-1\right)
\epsilon ^{-\frac{2}{r-1}}}{r+1}\dint_{0}^{+\infty }\dint_{0}^{+\infty }\ln
^{\beta -\frac{2}{r-1}}\left( b+y+s\right) \left( b+y+s\right) ^{-\frac{r+1}{%
r-1}+d-1}dyds \\ 
+\frac{2\epsilon }{r+1}\dint_{0}^{nT}\dint_{\Omega }a\left( x\right) \ln
^{\beta +1}\left( b+q\left( x\right) +s\right) \left\vert \partial
_{t}u\right\vert ^{r+1}dxds,%
\end{array}%
\end{equation*}%
for all $\epsilon >0$. In view of the fact that 
\begin{equation}
\begin{array}{ll}
-\frac{r+1}{r-1}+d<-1 & \text{if }1<r<1+\frac{2}{d} \\ 
\beta -\frac{2}{r-1}<-1\text{ and }-\frac{r+1}{r-1}+d=-1 & \text{if }r=1+%
\frac{2}{d},%
\end{array}
\label{proof log theorem 1}
\end{equation}%
we see that%
\begin{equation}
\begin{array}{l}
\medskip \dint_{0}^{nT}\dint_{\Omega }a\left( x\right) f\left( q\left(
x\right) +s\right) \left\vert \partial _{t}u\right\vert ^{2}dxds \\ 
\leq C\epsilon ^{-\frac{2}{r-1}}\left\Vert a\right\Vert _{L^{\infty }}+\frac{%
2\epsilon }{r+1}\dint_{0}^{nT}\dint_{\Omega }a\left( x\right) \ln ^{\beta
+1}\left( b+q\left( x\right) +s\right) \left\vert \partial _{t}u\right\vert
^{r+1}dxds,%
\end{array}
\label{log2}
\end{equation}%
for all $\epsilon >0$. We choose $\epsilon $ such that%
\begin{equation*}
\begin{array}{l}
\frac{1}{8\left( \beta +1\right) }-k_{4}\epsilon \left( \frac{p-2r}{p-r-1}+%
\frac{2}{r+1}\right) \geq \frac{1}{16\left( \beta +1\right) }.%
\end{array}%
\end{equation*}%
We conclude that there exists a positive constant $C_{1}$ such that%
\begin{equation*}
\begin{array}{l}
\medskip X\left( nT\right) +\frac{1}{8}\dint_{0}^{nT}\dint_{\Omega }f\left(
q\left( x\right) +s\right) \left( \left\vert \nabla u\right\vert
^{2}+\left\vert \partial _{t}u\right\vert ^{2}\right) dxds \\ 
\medskip -\dint_{0}^{nT}\dint_{\Omega }a\left( x\right) f_{1}^{\prime
}\left( q\left( x\right) +s\right) \left\vert u\right\vert ^{2}dxds-\frac{1}{%
2}\dint_{0}^{nT}\dint_{\Omega }a\left( x\right) f_{2}^{\prime }\left(
q\left( x\right) +s\right) \left\vert u\right\vert ^{r+1}dxds \\ 
+\frac{1}{16\left( \beta +1\right) }\dint_{0}^{nT}\dint_{\Omega }a\left(
x\right) \ln ^{\beta +1}\left( b+q\left( x\right) +s\right) \left\vert
\partial _{t}u\right\vert ^{r+1}dxds\leq C_{1}I_{0},\text{ for all }n\in 
\mathbb{N}
.%
\end{array}%
\end{equation*}%
Therefore we obtain%
\begin{equation*}
\begin{array}{l}
\medskip \frac{1}{8}\dint_{0}^{\infty }\dint_{\Omega }f\left( q\left(
x\right) +s\right) \left( \left\vert \nabla u\right\vert ^{2}+\left\vert
\partial _{t}u\right\vert ^{2}\right) dxds \\ 
\medskip -\dint_{0}^{+\infty }\dint_{\Omega }a\left( x\right) f_{1}^{\prime
}\left( q\left( x\right) +s\right) \left\vert u\right\vert ^{2}dxds-\frac{1}{%
2}\dint_{0}^{+\infty }\dint_{\Omega }a\left( x\right) f_{2}^{\prime }\left(
q\left( x\right) +s\right) \left\vert u\right\vert ^{r+1}dxds \\ 
+\frac{1}{16\left( \beta +1\right) }\dint_{0}^{+\infty }\dint_{\Omega
}a\left( x\right) \ln ^{\beta +1}\left( b+q\left( x\right) +s\right)
\left\vert \partial _{t}u\right\vert ^{r+1}dxds\leq C_{1}I_{0}.%
\end{array}%
\end{equation*}%
Now using the weighted energy estimate $\left( \ref{weighted energy}\right) $%
, we infer that%
\begin{equation*}
\begin{array}{l}
E_{\varphi }\left( u\right) \left( t\right) =\dint_{\Omega }\varphi \left(
q\left( x\right) +s\right) \left( \left\vert \nabla u\left( s\right)
\right\vert ^{2}+\left\vert \partial _{t}u\left( s\right) \right\vert
^{2}\right) \\ 
\leq E_{\varphi }\left( u\right) \left( 0\right) +\left( \beta +1\right)
\dint_{0}^{\infty }\dint_{\Omega }f\left( q\left( x\right) +s\right) \left(
\left\vert \nabla u\left( s\right) \right\vert ^{2}+\left\vert \partial
_{t}u\left( s\right) \right\vert ^{2}\right) dxds \\ 
\leq C_{0}I_{0},%
\end{array}%
\end{equation*}%
\ for some positive constant $C_{0}$. The sought estimate follows from the
estimate above and the fact that%
\begin{equation*}
\ln ^{\beta +1}\left( 2+t\right) \medskip E_{u}\left( t\right) \leq \medskip
E_{\varphi }\left( u\right) \left( t\right) .
\end{equation*}

\section{Proof of Theorem 2}

\subsection{Preliminary results}

Throughout this section we use the following notations: Let $\beta $ be a
real number such that $0<1+\beta <\tau ,$ where 
\begin{equation*}
\begin{array}{c}
\tau =\frac{r\delta _{0}^{\frac{r^{2}+1}{r}}\left( \lambda +1\right)
^{r-1}\left( r+1\right) ^{r}}{\delta _{0}^{\frac{r+1}{r}}+\delta
_{0}^{r}\left( \lambda +1\right) ^{r-1}\left( r+1\right) ^{r}\left( r\delta
_{0}\left( \lambda +1\right) \left( r+1\right) +\delta _{0}^{\frac{1}{r}%
}\right) \allowbreak },%
\end{array}%
\end{equation*}%
$\lambda $ any positive constant$~$and 
\begin{equation*}
\delta _{0}=\left( \lambda +1\right) ^{\frac{r^{2}}{r^{2}-1}}\left(
r+1\right) ^{-\frac{r}{r-1}}.
\end{equation*}

We take $\varphi \left( s\right) =\left( 1+\alpha s\right) ^{\beta +1}$
where 
\begin{equation*}
\alpha =\frac{rk^{r}\left( r+1\right) +\delta _{0}^{\frac{1}{r}}}{%
k^{r}\delta _{0}^{\frac{1}{r}}\left( r+1\right) \left( r-\tau \right) },
\end{equation*}%
and 
\begin{equation*}
k=\left( 1+\lambda \right) \left( r+1\right) \delta _{0}.
\end{equation*}

Finally, let $\psi \in C_{c}^{\infty }\left( 
\mathbb{R}
^{d}\right) $ such that $0\leq \psi \leq 1$ and 
\begin{equation*}
\psi \left( x\right) =\left\{ 
\begin{array}{ll}
1 & \text{for }\left\vert x\right\vert \leq L \\ 
0 & \text{for }\left\vert x\right\vert \geq 2L%
\end{array}%
\right. .
\end{equation*}

\begin{proposition}
We assume that Hyp A holds and $(\omega $,$T)$ geometrically controls $%
\Omega $. Let $\delta >0$, $R_{0}>L$ and $-1<\beta \leq 0$. There exists $%
C_{T,\delta }=C\left( T,\delta ,R_{0},\alpha ,\beta \right) >0$, such that
the following inequality%
\begin{equation}
\begin{array}{l}
\medskip \dint_{t}^{t+T}\dint_{\Omega \cap B_{R_{0}}}\left( 1+\alpha \left(
q\left( x\right) +s\right) \right) ^{\beta }\left( \left\vert u\right\vert
^{2}+\left\vert \nabla u\right\vert ^{2}+\left\vert \partial
_{t}u\right\vert ^{2}\right) dxds \\ 
\medskip \leq C_{T,\delta }\dint_{t}^{t+T}\dint_{\Omega }a\left( x\right)
\left( 1+\alpha \left( q\left( x\right) +s\right) \right) ^{\beta }\left(
\left\vert \partial _{t}u\right\vert ^{2}+\left\vert \partial
_{t}u\right\vert ^{2r}\right) dxds \\ 
\medskip +C_{T,\delta }\dint_{t}^{t+T}\dint_{\Omega }a\left( x\right) \left(
1+\alpha \left( q\left( x\right) +s\right) \right) ^{\beta -2}\left\vert
u\right\vert ^{2}dxds \\ 
+\delta \dint_{t}^{t+T}\dint_{\Omega }\left( 1+\alpha \left( q\left(
x\right) +s\right) \right) ^{\beta }\left( \left\vert \nabla u\right\vert
^{2}+\left\vert \partial _{t}u\right\vert ^{2}\right) dxds,%
\end{array}
\label{observability gradient}
\end{equation}%
holds for every $t\geq 0$ and for all $u$ solution of $\left( \ref{system}%
\right) $ with initial data $\left( u_{0},u_{1}\right) $ in $H_{0}^{1}\left(
\Omega \right) \cap H^{2}\left( \Omega \right) \times H_{0}^{1}\left( \Omega
\right) .$
\end{proposition}

\begin{proof}
We set%
\begin{equation*}
f\left( s\right) =\left( 1+\alpha s\right) ^{\beta }
\end{equation*}%
In view of $f\in L^{\infty }\left( 
\mathbb{R}
_{+}\right) ,$ we have $E_{f}\left( u\right) \left( 0\right) <\infty .$ On
the other hand, it is clear that $f^{\prime }\in L^{\infty }\left( 
\mathbb{R}
_{+}\right) $ and there exists a positive constant $K,$ such that 
\begin{equation*}
\begin{array}{l}
\underset{%
\mathbb{R}
_{+}}{\sup }\left\vert \frac{f^{\prime \prime }\left( t\right) }{f^{\prime
}\left( t\right) }\right\vert \leq K.%
\end{array}%
\end{equation*}%
In addition the function $t\longmapsto \left\vert \frac{f^{\prime }\left(
t\right) }{f\left( t\right) }\right\vert $ is decreasing and $\underset{%
t\rightarrow +\infty }{\lim }\left\vert \frac{f^{\prime }\left( t\right) }{%
f\left( t\right) }\right\vert =0$. Moreover there exists $C>0,$ such that%
\begin{equation*}
\begin{array}{l}
\frac{\left( f^{\prime }\left( t\right) \right) ^{2}}{f\left( t\right) }\leq
C\text{ }\left( -f^{\prime }\left( t\right) \right) ,\text{ for all }t\geq 0.%
\end{array}%
\end{equation*}%
Since%
\begin{equation*}
\partial _{t}u\in L^{\infty }\left( 
\mathbb{R}
_{+},H_{0}^{1}\left( \Omega \right) \right) ,
\end{equation*}%
then from Sobolev imbedding, we deduce that%
\begin{equation*}
\sqrt{a\left( x\right) \left( 1+\alpha \left( q\left( x\right) +s\right)
\right) ^{\beta }}\left\vert \partial _{t}u\right\vert ^{r}\in
L_{loc}^{2}\left( 
\mathbb{R}
_{+},L^{2}\left( \Omega \right) \right) .
\end{equation*}%
By taking into account of the results above, we can use proposition \ref%
{proposition onservability global} and we obtain $\left( \ref{observability
gradient}\right) .$This finishes the proof of the proposition.
\end{proof}

In order to prove theorem 2 we need the following result$.$

\begin{lemma}
\label{lemma xt poly}Let $u$ be a solution of $\left( \ref{system}\right) $
with initial data in $H_{0}^{1}\left( \Omega \right) \cap H^{2}\left( \Omega
\right) \times H_{0}^{1}\left( \Omega \right) $ such that%
\begin{equation*}
\begin{array}{c}
E_{\varphi }\left( u\right) \left( 0\right) =\left\Vert \left( 1+\alpha
q\right) ^{\frac{1+\beta }{2}}\nabla u_{0}\right\Vert
_{L^{2}}^{2}+\left\Vert \left( 1+\alpha q\right) ^{\frac{1+\beta }{2}%
}u_{1}\right\Vert _{L^{2}}^{2}<+\infty .%
\end{array}%
\end{equation*}%
We set $\chi =1-\psi ~$and%
\begin{equation}
\begin{array}{l}
\medskip X\left( t\right) =\dint_{\Omega }\left( 1+\alpha \left( q\left(
x\right) +t\right) \right) ^{\beta }\chi ^{2}\left( x\right) u\left(
t\right) \partial _{t}u\left( t\right) dx+\frac{k_{1}}{2}\dint_{\Omega
}\left( 1+\alpha \left( q\left( x\right) +t\right) \right) ^{\beta
-1}a\left( x\right) \left\vert u\left( t\right) \right\vert ^{2}dx \\ 
+\dint_{\Omega }a\left( x\right) \left( 1+\alpha \left( q\left( x\right)
+t\right) \right) ^{\beta -r+1}\left\vert u\left( t\right) \right\vert
^{r+1}dx+\frac{k}{2}\dint_{\Omega }\left( 1+\alpha \left( q\left( x\right)
+t\right) \right) ^{\beta +1}\left( \left\vert \nabla u\right\vert
^{2}+\left\vert \partial _{t}u\right\vert ^{2}\right) dx,%
\end{array}
\label{Xt definition polynomial}
\end{equation}%
where $k_{1}>0.$ Then%
\begin{equation}
\begin{array}{l}
\medskip X\left( t+T\right) -X\left( t\right) +\frac{1-k\alpha \left(
1+\beta \right) }{2}\dint_{t}^{t+T}\dint_{\Omega }\left( 1+\alpha \left(
q\left( x\right) +s\right) \right) ^{\beta }\left( \left\vert \nabla
u\right\vert ^{2}+\left\vert \partial _{t}u\right\vert ^{2}\right) dxds \\ 
\medskip +\left( \frac{k_{1}\alpha \left( 1-\beta \right) }{4}-\frac{\beta
^{2}\alpha ^{2}}{\epsilon _{0}\epsilon }\right) \dint_{t}^{t+T}\dint_{\Omega
}a\left( x\right) \left( 1+\alpha \left( q\left( x\right) +s\right) \right)
^{\beta -2}\left\vert u\left( t\right) \right\vert ^{2}dxds \\ 
\medskip +\lambda \delta _{0}\dint_{t}^{t+T}\dint_{\Omega }a\left( x\right)
\left( 1+\alpha \left( q\left( x\right) +s\right) \right) ^{\beta
+1}\left\vert \partial _{t}u\right\vert ^{r+1}dxds \\ 
\medskip \leq \left( 3+\left\Vert \nabla \chi ^{2}\right\Vert _{\infty
}\right) \dint_{t}^{t+T}\dint_{\Omega \cap B_{2L}}\left( 1+\alpha \left(
q\left( x\right) +s\right) \right) ^{\beta }\left( \left\vert u\right\vert
^{2}+\left\vert \nabla u\right\vert ^{2}+\left\vert \partial
_{t}u\right\vert ^{2}\right) dxds \\ 
+\left( \frac{2}{\epsilon _{0}}+\frac{8k_{1}}{\alpha \left( 1-\beta \right) }%
+\frac{8\beta ^{2}\alpha }{\epsilon _{0}^{2}k_{1}\left( 1-\beta \right) }%
\right) \dint_{t}^{t+T}\dint_{\Omega }a\left( x\right) \left( 1+\alpha
\left( q\left( x\right) +s\right) \right) ^{\beta }\left\vert \partial
_{t}u\right\vert ^{2}dxds,%
\end{array}
\label{X t estimate}
\end{equation}%
for all $t\geq 0,$ where $\lambda $ any positive constant.
\end{lemma}

\begin{proof}
We have%
\begin{equation*}
\begin{array}{l}
\dint_{\Omega }\varphi \left( q\left( x\right) \right) \left( \left\vert
\nabla u_{0}\right\vert ^{2}+\left\vert u_{1}\right\vert ^{2}\right)
dx<\infty .%
\end{array}%
\end{equation*}%
Then from $\left( \ref{weighted energy}\right) ,$ we infer 
\begin{equation*}
\begin{array}{l}
\medskip E_{\varphi }\left( u\right) \left( t+T\right)
+\dint_{t}^{t+T}\dint_{\Omega }a\left( x\right) \left( 1+\alpha \left(
q\left( x\right) +s\right) \right) ^{\beta +1}\left\vert \partial
_{t}u\right\vert ^{r+1}dxds \\ 
\leq E_{\varphi }\left( u\right) \left( t\right) +\left( \beta +1\right)
\alpha \dint_{t}^{t+T}\dint_{\Omega }\left( 1+\alpha \left( q\left( x\right)
+s\right) \right) ^{\beta }\left( \left\vert \nabla u\right\vert
^{2}+\left\vert \partial _{t}u\right\vert ^{2}\right) dxds.%
\end{array}%
\end{equation*}%
We set%
\begin{equation*}
\begin{array}{l}
\medskip X_{0}\left( t\right) =\dint_{\Omega }\left( 1+\alpha \left( q\left(
x\right) +t\right) \right) ^{\beta }\chi ^{2}\left( x\right) u\left(
t\right) \partial _{t}u\left( t\right) dx+\frac{k_{1}}{2}\dint_{\Omega
}\left( 1+\alpha \left( q\left( x\right) +t\right) \right) ^{\beta
-1}a\left( x\right) \left\vert u\left( t\right) \right\vert ^{2}dx \\ 
\medskip +\dint_{\Omega }a\left( x\right) \left( 1+\alpha \left( q\left(
x\right) +t\right) \right) ^{\beta -r+1}\left\vert u\left( t\right)
\right\vert ^{r+1}dx.%
\end{array}%
\end{equation*}%
Using the fact that $u$ is a solution of $\left( \ref{system}\right) ,$ we
deduce that%
\begin{equation}
\begin{array}{l}
\medskip \frac{d}{dt}X_{0}\left( t\right) =\dint_{\Omega }\left( \left\vert
\partial _{t}u\left( t\right) \right\vert ^{2}-\left\vert \nabla u\left(
t\right) \right\vert ^{2}-a\left( x\right) \left\vert \partial _{t}u\left(
t\right) \right\vert ^{r-1}u\left( t\right) \partial _{t}u\left( t\right)
\right) \chi ^{2}\left( x\right) \left( 1+\alpha \left( q\left( x\right)
+t\right) \right) ^{\beta }dx \\ 
\medskip -\dint_{\Omega }\left( 1+\alpha \left( q\left( x\right) +t\right)
\right) ^{\beta }u\left( t\right) \nabla \chi ^{2}\left( x\right) \nabla
u\left( t\right) +\beta \alpha \left( 1+\alpha \left( q\left( x\right)
+t\right) \right) ^{\beta -1}\chi ^{2}\left( x\right) u\left( t\right) \frac{%
x\cdot \nabla u\left( t\right) }{q\left( x\right) }dx \\ 
\medskip +\beta \alpha \dint_{\Omega }\left( 1+\alpha \left( q\left(
x\right) +t\right) \right) ^{\beta -1}\chi ^{2}\left( x\right) u\left(
t\right) \partial _{t}u\left( t\right) dx \\ 
\medskip +k_{1}\left( \dint_{\Omega }a\left( x\right) \left( 1+\alpha \left(
q\left( x\right) +t\right) \right) ^{\beta -1}u\left( t\right) \partial
_{t}u\left( t\right) dx+\frac{\beta -1}{2}\alpha \dint_{\Omega }a\left(
x\right) \left( 1+\alpha \left( q\left( x\right) +t\right) \right) ^{\beta
-2}\left\vert u\right\vert ^{2}dx\right) \\ 
+\left( \beta +1-r\right) \alpha \dint_{\Omega }a\left( x\right) \left(
1+\alpha \left( q\left( x\right) +t\right) \right) ^{\beta -r}\left\vert
u\right\vert ^{r+1}dx \\ 
+\left( r+1\right) \dint_{\Omega }a\left( x\right) \left( 1+\alpha \left(
q\left( x\right) +t\right) \right) ^{\beta -r+1}\left\vert u\right\vert
^{r-1}u\partial _{t}udx.%
\end{array}
\label{X0poly}
\end{equation}%
We note that $\left\Vert \chi \right\Vert _{\infty }\leq 1.$ Using Young's
inequality and the fact that the support of $\chi $ is contained in $\left\{
\left\vert x\right\vert \geq L\right\} $ and 
\begin{equation*}
a\left( x\right) >\epsilon _{0}>0\text{ for }\left\vert x\right\vert \geq L,
\end{equation*}%
we infer that%
\begin{equation*}
\begin{array}{l}
\medskip \left\vert \alpha \beta \dint_{\Omega }\left( 1+\alpha q\left(
x\right) +\alpha t\right) ^{\beta -1}\chi ^{2}\left( x\right) u\left(
t\right) \partial _{t}u\left( t\right) dx\right\vert \\ 
\leq \frac{k_{1}\alpha \left( 1-\beta \right) }{8}\dint_{\Omega }a\left(
x\right) \left( 1+\alpha \left( q\left( x\right) +t\right) \right) ^{\beta
-2}\left\vert u\left( t\right) \right\vert ^{2}dx+\frac{8\beta ^{2}\alpha }{%
\epsilon _{0}^{2}k_{1}\left( 1-\beta \right) }\dint_{\Omega }a\left(
x\right) \left( 1+\alpha \left( q\left( x\right) +t\right) \right) ^{\beta
}\left\vert \partial _{t}u\left( t\right) \right\vert ^{2}dx%
\end{array}%
\end{equation*}%
and%
\begin{equation*}
\begin{array}{l}
\medskip k_{1}\left\vert \dint_{\Omega }a\left( x\right) \left( 1+\alpha
\left( q\left( x\right) +t\right) \right) ^{\beta -1}u\left( t\right)
\partial _{t}u\left( t\right) dx\right\vert \\ 
\leq \frac{k_{1}\alpha \left( 1-\beta \right) }{8}\dint_{\Omega }a\left(
x\right) \left( 1+\alpha \left( q\left( x\right) +t\right) \right) ^{\beta
-2}\left\vert u\left( t\right) \right\vert ^{2}dx+\frac{8k_{1}}{\alpha
\left( 1-\beta \right) }\dint_{\Omega }a\left( x\right) \left( 1+\alpha
\left( q\left( x\right) +t\right) \right) ^{\beta }\left\vert \partial
_{t}u\left( t\right) \right\vert ^{2}dx.%
\end{array}%
\end{equation*}

Using the same arguments, we also deduce that%
\begin{equation*}
\begin{array}{l}
\medskip \left\vert \dint_{\Omega }\beta \alpha \left( 1+\alpha \left(
q\left( x\right) +t\right) \right) ^{\beta -1}\chi ^{2}\left( x\right)
u\left( t\right) \frac{x\cdot \nabla u\left( t\right) }{q\left( x\right) }%
dx\right\vert \\ 
\leq \frac{\beta ^{2}\alpha ^{2}}{\epsilon _{0}\epsilon }\dint_{\Omega
}a\left( x\right) \left( 1+\alpha \left( q\left( x\right) +t\right) \right)
^{\beta -2}\left\vert u\left( t\right) \right\vert ^{2}dx+\epsilon
\dint_{\Omega }\left( 1+\alpha \left( q\left( x\right) +t\right) \right)
^{\beta }\left\vert \nabla u\left( t\right) \right\vert ^{2}dx,%
\end{array}%
\end{equation*}%
for all $\epsilon >0.~$Using the fact that the support of $\psi $ is
contained in $\left\{ \left\vert x\right\vert \leq 2L\right\} $ and 
\begin{equation*}
a\left( x\right) >\epsilon _{0}>0\text{ for }\left\vert x\right\vert \geq L,
\end{equation*}%
we get%
\begin{equation*}
\begin{array}{l}
\dint_{\Omega }\chi ^{2}\left( x\right) \left( \left\vert \partial
_{t}u\left( t\right) \right\vert ^{2}-\left\vert \nabla u\left( t\right)
\right\vert ^{2}\right) \left( 1+\alpha \left( q\left( x\right) +t\right)
\right) ^{\beta }dx \\ 
=\dint_{\Omega }\left( 1-2\psi \left( x\right) +\psi ^{2}\left( x\right)
\right) \left( 1+\alpha \left( q\left( x\right) +t\right) \right) ^{\beta
}\left( \left\vert \partial _{t}u\left( t\right) \right\vert ^{2}-\left\vert
\nabla u\left( t\right) \right\vert ^{2}\right) dx \\ 
\leq \frac{2}{\epsilon _{0}}\dint_{\Omega }a\left( x\right) \left( 1+\alpha
\left( q\left( x\right) +t\right) \right) ^{\beta }\left\vert \partial
_{t}u\left( t\right) \right\vert ^{2}dx \\ 
-\dint_{\Omega }\left( 1+\alpha \left( q\left( x\right) +t\right) \right)
^{\beta }\left( \left\vert \partial _{t}u\left( t\right) \right\vert
^{2}+\left\vert \nabla u\left( t\right) \right\vert ^{2}\right) dx \\ 
+3\dint_{\Omega \cap B_{2L}}\left( 1+\alpha \left( q\left( x\right)
+t\right) \right) ^{\beta }\left( \left\vert \partial _{t}u\left( t\right)
\right\vert ^{2}+\left\vert \nabla u\left( t\right) \right\vert ^{2}\right)
dx.%
\end{array}%
\end{equation*}%
We note that the support of $\nabla \chi ^{2}$ is contained in $\left\{
\left\vert x\right\vert \leq 2L\right\} $, using Young's inequality, we
deduce that%
\begin{equation*}
\begin{array}{l}
\left\vert -\dint_{\Omega }\left( 1+\alpha \left( q\left( x\right) +t\right)
\right) ^{\beta }u\left( t\right) \nabla \chi ^{2}\left( x\right) \nabla
u\left( t\right) dx\right\vert \\ 
\leq \frac{1}{2}\left\Vert \nabla \chi ^{2}\right\Vert _{\infty
}\dint_{\Omega \cap B_{2L}}\left( 1+\alpha \left( q\left( x\right) +t\right)
\right) ^{\beta }\left( \left\vert u\left( t\right) \right\vert
^{2}+\left\vert \nabla u\left( t\right) \right\vert ^{2}\right) dx.%
\end{array}%
\end{equation*}%
Young's inequality combined with the fact that $\left\Vert \chi \right\Vert
_{\infty }\leq 1,$ gives%
\begin{equation*}
\begin{array}{l}
\medskip \left\vert \dint_{\Omega }a\left( x\right) \left( 1+\alpha \left(
q\left( x\right) +t\right) \right) ^{\beta }\chi ^{2}\left( x\right)
\left\vert \partial _{t}u\left( t\right) \right\vert ^{r-1}u\partial
_{t}u\left( t\right) dx\right\vert \\ 
\leq \frac{rk}{r+1}\dint_{\Omega }a\left( x\right) \left( 1+\alpha \left(
q\left( x\right) +t\right) \right) ^{\beta +1}\left\vert \partial
_{t}u\left( t\right) \right\vert ^{r+1}dx+\frac{k^{-r}}{r+1}\int_{\Omega
}a\left( x\right) \left( 1+\alpha \left( q\left( x\right) +t\right) \right)
^{\beta -r}\left\vert u\left( t\right) \right\vert ^{r+1}dx%
\end{array}%
\end{equation*}%
and%
\begin{equation*}
\begin{array}{l}
\left( r+1\right) \left\vert \dint_{\Omega }a\left( x\right) \left( 1+\alpha
\left( q\left( x\right) +t\right) \right) ^{\beta -r+1}\left\vert u\left(
t\right) \right\vert ^{r-1}u\partial _{t}u\left( t\right) dx\right\vert \\ 
\leq \delta _{0}\dint_{\Omega }a\left( x\right) \left( 1+\alpha \left(
q\left( x\right) +t\right) \right) ^{\beta +1}\left\vert \partial
_{t}u\left( t\right) \right\vert ^{r+1}dx+r\delta _{0}^{-\frac{1}{r}%
}\dint_{\Omega }a\left( x\right) \left( 1+\alpha \left( q\left( x\right)
+t\right) \right) ^{\beta -r}\left\vert u\left( t\right) \right\vert
^{r+1}dx.%
\end{array}%
\end{equation*}%
By taking into account of the estimates above, making some arrangement in $%
\left( \ref{X0poly}\right) $ and integrating the result between $t$ and $%
t+T, $ we obtain%
\begin{equation*}
\begin{array}{l}
\medskip X\left( t+T\right) -X\left( t\right) +\left( 1-\epsilon -\left(
1+\beta \right) k\alpha \right) \dint_{t}^{t+T}\dint_{\Omega }\left(
1+\alpha \left( q\left( x\right) +s\right) \right) ^{\beta }\left(
\left\vert \nabla u\right\vert ^{2}+\left\vert \partial _{t}u\right\vert
^{2}\right) dxds \\ 
\medskip +\left( \frac{k_{1}\alpha \left( 1-\beta \right) }{4}-\frac{\beta
^{2}\alpha ^{2}}{\epsilon _{0}\epsilon }\right) \dint_{t}^{t+T}\dint_{\Omega
}a\left( x\right) \left( 1+\alpha \left( q\left( x\right) +s\right) \right)
^{\beta -2}\left\vert u\right\vert ^{2}dxds \\ 
\medskip +\left( \left( \alpha -\delta _{0}^{-\frac{1}{r}}\right) r-\left(
\beta +1\right) \alpha -\frac{k^{-r}}{r+1}\right)
\dint_{t}^{t+T}\dint_{\Omega }a\left( x\right) \left( 1+\alpha \left(
q\left( x\right) +s\right) \right) ^{\beta -r}\left\vert u\right\vert
^{r+1}dxds \\ 
\medskip +\left( k-\frac{kr}{r+1}-\delta _{0}\right)
\dint_{t}^{t+T}\dint_{\Omega }a\left( x\right) \left( 1+\alpha \left(
q\left( x\right) +s\right) \right) ^{\beta +1}\left\vert \partial
_{t}u\right\vert ^{r+1}dxds \\ 
\medskip \leq \left( 3+\left\Vert \nabla \chi ^{2}\right\Vert _{\infty
}\right) \dint_{t}^{t+T}\dint_{\Omega \cap B_{2L}}\left( 1+\alpha \left(
q\left( x\right) +s\right) \right) ^{\beta }\left( \left\vert u\right\vert
^{2}+\left\vert \nabla u\right\vert ^{2}+\left\vert \partial
_{t}u\right\vert ^{2}\right) dxds \\ 
+\left( \frac{2}{\epsilon _{0}}+\frac{8k_{1}}{\alpha \left( 1-\beta \right) }%
+\frac{8\beta ^{2}\alpha }{\epsilon _{0}^{2}k_{1}\left( 1-\beta \right) }%
\right) \dint_{t}^{t+T}\dint_{\Omega }a\left( x\right) \left( 1+\alpha
\left( q\left( x\right) +s\right) \right) ^{\beta }\left\vert \partial
_{t}u\right\vert ^{2}dxds,%
\end{array}%
\end{equation*}%
for all $\epsilon >0.~$

We have 
\begin{equation*}
1-\left( \beta +1\right) k\alpha >1-\tau k\alpha =0.~
\end{equation*}%
So we can choose $\epsilon =\frac{1-\gamma k\alpha }{2}.$ It is easy to see
that 
\begin{equation*}
\begin{array}{c}
\left( \alpha -\delta _{0}^{-\frac{1}{r}}\right) r-\left( \beta +1\right)
\alpha -\frac{k^{-r}}{r+1} \\ 
>\left( \alpha -\delta _{0}^{-\frac{1}{r}}\right) r-\tau \alpha -\frac{k^{-r}%
}{r+1}=0%
\end{array}%
\end{equation*}%
and 
\begin{equation*}
\begin{array}{l}
\medskip k-\frac{kr}{r+1}-\delta _{0}=\lambda \delta _{0}.%
\end{array}%
\end{equation*}%
Collecting the estimates above, we get $\left( \ref{X t estimate}\right) .$
\end{proof}

\subsection{Proof of Theorem 2}

We assume that Hyp A holds and $\omega $ satisfies the GCC. Let $u$ be a
solution of $\left( \ref{system}\right) $ with initial data in $%
H_{0}^{1}\left( \Omega \right) \cap H^{2}\left( \Omega \right) \times
H_{0}^{1}\left( \Omega \right) $ such that 
\begin{equation*}
\begin{array}{c}
\left\Vert \left( 1+\alpha q\right) ^{\frac{\gamma }{2}}\nabla
u_{0}\right\Vert _{L^{2}}^{2}+\left\Vert \left( 1+\alpha q\right) ^{\frac{%
\gamma }{2}}u_{1}\right\Vert _{L^{2}}^{2}<+\infty .%
\end{array}%
\end{equation*}%
We set $\gamma =1+\beta .$ Using $\left( \ref{X t estimate}\right) $\ and $%
\left( \ref{observability gradient}\right) $ and arguing as in the proof of
theorem 1 we obtain%
\begin{equation}
\begin{array}{l}
\medskip X\left( t+T\right) -X\left( t\right) +\left( \frac{1-k\alpha \left(
1+\beta \right) }{2}-\left( 3+\left\Vert \nabla \chi ^{2}\right\Vert
_{\infty }\right) \delta \right) \dint_{t}^{t+T}\dint_{\Omega }\left(
1+\alpha \left( q\left( x\right) +s\right) \right) ^{\beta }\left(
\left\vert \nabla u\right\vert ^{2}+\left\vert \partial _{t}u\right\vert
^{2}\right) dxds \\ 
\medskip +\left( \frac{k_{1}\alpha \left( 1-\beta \right) }{4}-\frac{\beta
^{2}\alpha ^{2}}{\epsilon _{0}\epsilon }-\left( 3+\left\Vert \nabla \chi
^{2}\right\Vert _{\infty }\right) C_{T,\delta }\right)
\dint_{t}^{t+T}\dint_{\Omega }a\left( x\right) \left( 1+\alpha \left(
q\left( x\right) +s\right) \right) ^{\beta -2}\left\vert u\right\vert
^{2}dxds \\ 
\medskip +\lambda \delta _{0}\dint_{t}^{t+T}\dint_{\Omega }a\left( x\right)
\left( 1+\alpha \left( q\left( x\right) +s\right) \right) ^{\beta
+1}\left\vert \partial _{t}u\right\vert ^{r+1}dxds \\ 
\leq k_{2}\dint_{t}^{t+T}\dint_{\Omega }a\left( x\right) \left( 1+\alpha
\left( q\left( x\right) +s\right) \right) ^{\beta }\left( \left\vert
\partial _{t}u\right\vert ^{2r}+\left\vert \partial _{t}u\right\vert
^{2}\right) dxds,%
\end{array}
\label{proof thm 1 estimate prefinal}
\end{equation}%
for all $t\geq 0$, and for some $k_{2}>0.$

Using Young's inequality we get%
\begin{equation}
\begin{array}{l}
\medskip X\left( t\right) \leq \left( \frac{k_{1}}{2}+\frac{1}{2\epsilon
_{0}\epsilon }\right) \dint_{\Omega }a\left( x\right) \left( 1+\alpha \left(
q\left( x\right) +t\right) \right) ^{\beta -1}\left\vert u\left( t\right)
\right\vert ^{2}dx \\ 
\medskip +\left( k+\epsilon \right) \dint_{\Omega }\left( 1+\alpha \left(
q\left( x\right) +t\right) \right) ^{\beta +1}\left( \left\vert \nabla
u\left( t\right) \right\vert ^{2}+\left\vert \partial _{t}u\left( t\right)
\right\vert ^{2}\right) dx \\ 
+\dint_{\Omega }a\left( x\right) \left( 1+\alpha \left( q\left( x\right)
+t\right) \right) ^{\beta -r+1}\left\vert u\left( t\right) \right\vert
^{r+1}dx%
\end{array}
\label{k estimate sup}
\end{equation}%
and%
\begin{equation}
\begin{array}{l}
\medskip X\left( t\right) \geq \left( \frac{k_{1}}{2}-\frac{1}{\epsilon
_{0}\epsilon }\right) \dint_{\Omega }a\left( x\right) \left( 1+\alpha \left(
q\left( x\right) +t\right) \right) ^{\beta -1}\left\vert u\left( t\right)
\right\vert ^{2}dx \\ 
\medskip +\left( k-\epsilon \right) \dint_{\Omega }\left( 1+\alpha \left(
q\left( x\right) +t\right) \right) ^{\beta +1}\left( \left\vert \nabla
u\left( t\right) \right\vert ^{2}+\left\vert \partial _{t}u\left( t\right)
\right\vert ^{2}\right) dx \\ 
+\dint_{\Omega }a\left( x\right) \left( 1+\alpha \left( q\left( x\right)
+t\right) \right) ^{\beta -r+1}\left\vert u\left( t\right) \right\vert
^{r+1}dx,%
\end{array}
\label{k estimate inf}
\end{equation}%
for all $\epsilon >0.$ We choose (by taking into account of the order below)%
\begin{equation*}
\begin{array}{l}
\medskip \epsilon \text{ such that }k-\epsilon \geq \delta _{0} \\ 
\medskip \delta \text{ such that }\frac{1-k\alpha \left( 1+\beta \right) }{2}%
-\left( 3+\left\Vert \nabla \chi ^{2}\right\Vert _{\infty }\right) \delta
\geq \frac{1-k\alpha \left( 1+\beta \right) }{4} \\ 
k_{1}\text{ such that }\frac{k_{1}}{2}-\frac{1}{2\epsilon _{0}\epsilon }\geq
\delta _{0}\text{ and }\frac{k_{1}\left( 1-\beta \right) }{4}-\frac{2\beta
^{2}}{\epsilon _{0}\delta _{0}}-\left( 3+\left\Vert \nabla \chi
^{2}\right\Vert _{\infty }\right) C_{T,\delta }\geq \delta _{0}.%
\end{array}%
\end{equation*}%
Therefore%
\begin{equation}
\begin{array}{l}
\medskip X\left( t\right) \geq \delta _{0}\dint_{\Omega }a\left( x\right)
\left( 1+\alpha \left( q\left( x\right) +t\right) \right) ^{\beta
-1}\left\vert u\left( t\right) \right\vert ^{2}dx \\ 
\medskip +\delta _{0}\dint_{\Omega }\left( 1+\alpha \left( q\left( x\right)
+t\right) \right) ^{\beta +1}\left( \left\vert \nabla u\left( t\right)
\right\vert ^{2}+\left\vert \partial _{t}u\left( t\right) \right\vert
^{2}\right) dx \\ 
+\dint_{\Omega }a\left( x\right) \left( 1+\alpha \left( q\left( x\right)
+t\right) \right) ^{\beta -r+1}\left\vert u\left( t\right) \right\vert
^{r+1}dx.%
\end{array}
\label{k estimate inf final}
\end{equation}%
and%
\begin{equation*}
\begin{array}{l}
\medskip X\left( t+T\right) -X\left( t\right) +\frac{1-k\alpha \left(
1+\beta \right) }{4}\dint_{t}^{t+T}\dint_{\Omega }\left( 1+\alpha \left(
q\left( x\right) +s\right) \right) ^{\beta }\left( \left\vert \nabla
u\right\vert ^{2}+\left\vert \partial _{t}u\right\vert ^{2}\right) dxds \\ 
\medskip +\delta _{0}\dint_{t}^{t+T}\dint_{\Omega }a\left( x\right) \left(
1+\alpha \left( q\left( x\right) +s\right) \right) ^{\beta -2}\left\vert
u\right\vert ^{2}dxds \\ 
\medskip +\lambda \delta _{0}\dint_{t}^{t+T}\dint_{\Omega }a\left( x\right)
\left( 1+\alpha \left( q\left( x\right) +s\right) \right) ^{\beta
+1}\left\vert \partial _{t}u\right\vert ^{r+1}dxds \\ 
\leq k_{2}\dint_{t}^{t+T}\dint_{\Omega }a\left( x\right) \left( 1+\alpha
\left( q\left( x\right) +s\right) \right) ^{\beta }\left( \left\vert
\partial _{t}u\right\vert ^{2r}+\left\vert \partial _{t}u\right\vert
^{2}\right) dxds,%
\end{array}%
\end{equation*}%
for all $t\geq 0.$ Thus%
\begin{equation}
\begin{array}{l}
\medskip X\left( nT\right) +\frac{1-k\alpha \left( 1+\beta \right) }{4}%
\dint_{0}^{nT}\dint_{\Omega }\left( 1+\alpha \left( q\left( x\right)
+s\right) \right) ^{\beta }\left( \left\vert \nabla u\right\vert
^{2}+\left\vert \partial _{t}u\right\vert ^{2}\right) dxds \\ 
\medskip +\delta _{0}\dint_{0}^{nT}\dint_{\Omega }a\left( x\right) \left(
1+\alpha \left( q\left( x\right) +s\right) \right) ^{\beta -2}\left\vert
u\right\vert ^{2}dxds \\ 
\medskip +\lambda \delta _{0}\dint_{0}^{nT}\dint_{\Omega }a\left( x\right)
\left( 1+\alpha \left( q\left( x\right) +s\right) \right) ^{\beta
+1}\left\vert \partial _{t}u\right\vert ^{r+1}dxds \\ 
\leq k_{2}\dint_{0}^{nT}\dint_{\Omega }a\left( x\right) \left( 1+\alpha
\left( q\left( x\right) +s\right) \right) ^{\beta }\left( \left\vert
\partial _{t}u\right\vert ^{2r}+\left\vert \partial _{t}u\right\vert
^{2}\right) dxds+X\left( 0\right) ,\text{ for all }n\geq 1.%
\end{array}
\label{Xt final estimate1}
\end{equation}%
In view of proposition \ref{proposition weighted energy}%
\begin{equation*}
\medskip X\left( 0\right) \leq CI_{1}.
\end{equation*}%
where $I_{1}$ is defined in the statement of theorem 2. Therefore there
exists a positive constant $k_{3},$ such that%
\begin{equation}
\begin{array}{l}
\medskip X\left( nT\right) +\frac{1-k\alpha \left( 1+\beta \right) }{4}%
\dint_{0}^{nT}\dint_{\Omega }\left( 1+\alpha \left( q\left( x\right)
+s\right) \right) ^{\beta }\left( \left\vert \nabla u\right\vert
^{2}+\left\vert \partial _{t}u\right\vert ^{2}\right) dxds \\ 
\medskip +\delta _{0}\dint_{0}^{nT}\dint_{\Omega }a\left( x\right) \left(
1+\alpha \left( q\left( x\right) +s\right) \right) ^{\beta -2}\left\vert
u\right\vert ^{2}dxds \\ 
\medskip +\lambda \delta _{0}\dint_{0}^{nT}\dint_{\Omega }a\left( x\right)
\left( 1+\alpha \left( q\left( x\right) +s\right) \right) ^{\beta
+1}\left\vert \partial _{t}u\right\vert ^{r+1}dxds \\ 
\leq k_{3}\left( \dint_{0}^{nT}\dint_{\Omega }a\left( x\right) \left(
1+\alpha \left( q\left( x\right) +s\right) \right) ^{\beta }\left(
\left\vert \partial _{t}u\right\vert ^{2}+\left\vert \partial
_{t}u\right\vert ^{2r}\right) dxds+I_{1}\right) ,\text{ for all }n\in 
\mathbb{N}
.%
\end{array}
\label{poly aux}
\end{equation}%
As in the proof of theorem 1, we absorb the first term of the RHS of the
estimate above by the last term of the LHS. Proceeding as the proof of
theorem 1, we obtain 
\begin{equation}
\begin{array}{l}
\medskip \dint_{0}^{nT}\dint_{\Omega }a\left( x\right) \left( 1+\alpha
\left( q\left( x\right) +s\right) \right) ^{\beta }\left\vert \partial
_{t}u\right\vert ^{2r}dxds\leq C\epsilon ^{-\frac{p-2r}{r-1}}\left\Vert
a\right\Vert _{L^{\infty }}\left( \left\Vert u_{0}\right\Vert
_{H^{2}}^{2}+\left\Vert u_{1}\right\Vert _{H^{1}}^{2}+\left\Vert
u_{1}\right\Vert _{H^{1}}^{2r}\right) ^{\frac{p}{2}} \\ 
+\frac{\epsilon \left( p-2r\right) }{p-r-1}\dint_{0}^{nT}\dint_{\Omega
}a\left( x\right) \left( 1+\alpha \left( q\left( x\right) +s\right) \right)
^{\beta +1}\left\vert \partial _{t}u\right\vert ^{r+1}dxds.%
\end{array}
\label{poly1}
\end{equation}%
and%
\begin{equation}
\begin{array}{l}
\medskip \dint_{0}^{nT}\dint_{\Omega }a\left( x\right) \left( 1+\alpha
\left( q\left( x\right) +s\right) \right) ^{\beta }\left\vert \partial
_{t}u\right\vert ^{2}dxds \\ 
\leq C\epsilon ^{-\frac{2}{r-1}}+\frac{2\epsilon }{r+1}\dint_{0}^{nT}\dint_{%
\Omega }a\left( x\right) \left( 1+\alpha \left( q\left( x\right) +s\right)
\right) ^{\beta +1}\left\vert \partial _{t}u\right\vert ^{r+1}dxds%
\end{array}
\label{poly 2}
\end{equation}%
We choose $\epsilon $ such that%
\begin{equation*}
\begin{array}{l}
\lambda \delta _{0}-k_{3}\epsilon \left( \frac{p-2r}{p-r-1}+\frac{2}{r+1}%
\right) \geq \frac{\lambda \delta _{0}}{2}%
\end{array}%
\end{equation*}%
So there exists a positive constant $C_{0}$ such that%
\begin{equation*}
\begin{array}{l}
\medskip X\left( nT\right) +\frac{1-k\alpha \left( 1+\beta \right) }{4}%
\dint_{0}^{nT}\dint_{\Omega }\left( 1+\alpha \left( q\left( x\right)
+s\right) \right) ^{\beta }\left( \left\vert \nabla u\right\vert
^{2}+\left\vert \partial _{t}u\right\vert ^{2}\right) dxds \\ 
\medskip +\delta _{0}\dint_{0}^{nT}\dint_{\Omega }a\left( x\right) \left(
1+\alpha \left( q\left( x\right) +s\right) \right) ^{\beta -2}\left\vert
u\right\vert ^{2}dxds \\ 
\medskip +\frac{\lambda \delta _{0}}{2}\dint_{0}^{nT}\dint_{\Omega }a\left(
x\right) \left( 1+\alpha \left( q\left( x\right) +s\right) \right) ^{\beta
+1}\left\vert \partial _{t}u\right\vert ^{r+1}dxds\leq C_{0}I_{1},\text{ for
all }n\in 
\mathbb{N}
.%
\end{array}%
\end{equation*}%
Therefore we obtain%
\begin{equation*}
\begin{array}{l}
\medskip \frac{1-k\alpha \left( 1+\beta \right) }{4}\dint_{0}^{+\infty
}\dint_{\Omega }\left( 1+\alpha \left( q\left( x\right) +s\right) \right)
^{\beta }\left( \left\vert \nabla u\right\vert ^{2}+\left\vert \partial
_{t}u\right\vert ^{2}\right) dxds \\ 
\medskip +\delta _{0}\dint_{0}^{+\infty }\dint_{\Omega }a\left( x\right)
\left( 1+\alpha \left( q\left( x\right) +s\right) \right) ^{\beta
-2}\left\vert u\right\vert ^{2}dxds \\ 
\medskip +\frac{\lambda \delta _{0}}{2}\dint_{0}^{+\infty }\dint_{\Omega
}a\left( x\right) \left( 1+\alpha \left( q\left( x\right) +s\right) \right)
^{\beta +1}\left\vert \partial _{t}u\right\vert ^{r+1}dxds\leq C_{0}I_{1},%
\text{ for all }n\in 
\mathbb{N}
.%
\end{array}%
\end{equation*}%
Now using the weighted energy estimate $\left( \ref{weighted energy}\right) $%
, we infer that%
\begin{equation*}
\begin{array}{l}
E_{\varphi }\left( u\right) \left( t\right) =\frac{1}{2}\dint_{\Omega
}\left( 1+\alpha \left( q\left( x\right) +s\right) \right) ^{\beta +1}\left(
\left\vert \nabla u\right\vert ^{2}+\left\vert \partial _{t}u\right\vert
^{2}\right) dx \\ 
\leq E_{\varphi }\left( u\right) \left( 0\right) +\alpha \left( \beta
+1\right) \dint_{0}^{\infty }\dint_{\Omega }\left( 1+\alpha \left( q\left(
x\right) +s\right) \right) ^{\beta }\left( \left\vert \nabla u\left(
s\right) \right\vert ^{2}+\left\vert \partial _{t}u\left( s\right)
\right\vert ^{2}\right) dxds \\ 
\leq C_{1}I_{1},%
\end{array}%
\end{equation*}%
\ for some positive constant $C_{1}$. The sought estimate follows from the
estimate above and the fact that%
\begin{equation*}
\left( 1+\alpha t\right) ^{\beta +1}E_{u}\left( t\right) \leq E_{\varphi
}\left( u\right) \left( t\right) .
\end{equation*}%
This finishes the proof of theorem 2.

\section{Proof of Theorem 3}

\subsection{Preliminary results}

Throughout this section we use the following notations: Let $\beta $ be a
real number such that $0<1+\beta <\tau ,$ where 
\begin{equation*}
\begin{array}{c}
\tau _{1}=\frac{2r\delta _{0}^{\frac{r^{2}+1}{r}}\left( \lambda +1\right)
^{r-1}\left( r+1\right) ^{r}}{\delta _{0}^{\frac{r+1}{r}}+\delta
_{0}^{r}\left( \lambda +1\right) ^{r-1}\left( r+1\right) ^{r}\left( r\delta
_{0}\left( \lambda +1\right) \left( r+1\right) +2\delta _{0}^{\frac{1}{r}%
}\right) \allowbreak },%
\end{array}%
\end{equation*}%
$\lambda $ any positive constant such that $\lambda <1~$and 
\begin{equation*}
\delta _{0}=\left( \lambda +1\right) ^{\frac{r^{2}}{r^{2}-1}}\left(
r+1\right) ^{-\frac{r}{r-1}}.
\end{equation*}

Let $R>0,$ we take $\varphi \left( s\right) =\left( R+\alpha s\right)
^{\beta +1}$ where 
\begin{equation*}
\alpha =\frac{rk^{r}\left( r+1\right) +\delta _{0}^{\frac{1}{r}}}{%
k^{r}\delta _{0}^{\frac{1}{r}}\left( r+1\right) \left( r-\tau _{1}\right) },
\end{equation*}%
and 
\begin{equation*}
k=\left( 1+\lambda \right) \left( r+1\right) \delta _{0}.
\end{equation*}

Finally, let $\psi \in C_{c}^{\infty }\left( 
\mathbb{R}
^{d}\right) $ such that $0\leq \psi \leq 1$ and 
\begin{equation*}
\psi \left( x\right) =\left\{ 
\begin{array}{ll}
1 & \text{for }\left\vert x\right\vert \leq L \\ 
0 & \text{for }\left\vert x\right\vert \geq 2L%
\end{array}%
\right. .
\end{equation*}

From proposition \ref{proposition onservability global} we deduce the
following result.

\begin{proposition}
We assume that Hyp A holds and $(\omega $,$T)$ geometrically controls $%
\Omega $. Let $\delta >0$,$~R,R_{0}>L$ and $-1<\beta \leq 0$. There exists $%
C_{T,\delta }=C\left( T,\delta ,R_{0},R,\alpha ,\beta \right) >0$, such that
the following inequality%
\begin{equation}
\begin{array}{l}
\medskip \dint_{t}^{t+T}\dint_{\Omega \cap B_{R_{0}}}\left( R+\alpha
s\right) ^{\beta }\left( \left\vert u\right\vert ^{2}+\left\vert \nabla
u\right\vert ^{2}+\left\vert \partial _{t}u\right\vert ^{2}\right) dxds \\ 
\medskip \leq C_{T,\delta }\dint_{t}^{t+T}\dint_{\Omega }a\left( x\right)
\left( R+\alpha s\right) ^{\beta }\left( \left\vert \partial
_{t}u\right\vert ^{2}+\left\vert \partial _{t}u\right\vert ^{2r}\right) dxds
\\ 
\medskip +C_{T,\delta }\dint_{t}^{t+T}\dint_{\Omega }a\left( x\right) \left(
R+\alpha s\right) ^{\beta -2}\left\vert u\right\vert ^{2}dxds \\ 
+\delta \dint_{t}^{t+T}\dint_{\Omega }\left( R+\alpha s\right) ^{\beta
}\left( \left\vert \nabla u\right\vert ^{2}+\left\vert \partial
_{t}u\right\vert ^{2}\right) dxds,%
\end{array}%
\end{equation}%
holds for every $t\geq 0$ and for all $u$ solution of $\left( \ref{system}%
\right) $ with initial data $\left( u_{0},u_{1}\right) $ in $H_{0}^{1}\left(
\Omega \right) \cap H^{2}\left( \Omega \right) \times H_{0}^{1}\left( \Omega
\right) .$
\end{proposition}

As in the proof of theorem we need to define and to show an estimate for an
auxiliary function $X\left( t\right) .$

\begin{lemma}
\label{lemma xt poly copy(1)}Let $u$ be a solution of $\left( \ref{system}%
\right) $ with initial data in $H_{0}^{1}\left( \Omega \right) \cap
H^{2}\left( \Omega \right) \times H_{0}^{1}\left( \Omega \right) $ such
that. We set $\chi =1-\psi ~$and%
\begin{equation}
\begin{array}{l}
\medskip X\left( t\right) =\dint_{\Omega }\left( R+\alpha t\right) ^{\beta
}\chi ^{2}\left( x\right) u\left( t\right) \partial _{t}u\left( t\right) dx+%
\frac{k_{1}}{2}\dint_{\Omega }\left( R+\alpha t\right) ^{\beta -1}a\left(
x\right) \left\vert u\left( t\right) \right\vert ^{2}dx \\ 
+\dint_{\Omega }a\left( x\right) \left( R+\alpha t\right) ^{\beta
-r+1}\left\vert u\left( t\right) \right\vert ^{r+1}dx+\frac{k}{2}%
\dint_{\Omega }\left( R+\alpha t\right) ^{\beta +1}\left( \left\vert \nabla
u\right\vert ^{2}+\left\vert \partial _{t}u\right\vert ^{2}\right) dx,%
\end{array}%
\end{equation}%
where $k_{1}.$ Then%
\begin{equation}
\begin{array}{l}
\medskip X\left( t+T\right) -X\left( t\right) +\frac{2-k\alpha \left(
1+\beta \right) }{2}\dint_{t}^{t+T}\dint_{\Omega }\left( R+\alpha s\right)
^{\beta }\left( \left\vert \nabla u\right\vert ^{2}+\left\vert \partial
_{t}u\right\vert ^{2}\right) dxds \\ 
\medskip +\left( \frac{k_{1}\alpha \left( 1-\beta \right) }{4}-\frac{\beta
^{2}\alpha ^{2}}{\epsilon _{0}\epsilon }\right) \dint_{t}^{t+T}\dint_{\Omega
}a\left( x\right) \left( R+\alpha s\right) ^{\beta -2}\left\vert u\left(
t\right) \right\vert ^{2}dxds \\ 
\medskip +\lambda \delta _{0}\dint_{t}^{t+T}\dint_{\Omega }a\left( x\right)
\left( R+\alpha s\right) ^{\beta +1}\left\vert \partial _{t}u\right\vert
^{r+1}dxds \\ 
\medskip \leq \left( 3+\left\Vert \nabla \chi ^{2}\right\Vert _{\infty
}\right) \dint_{t}^{t+T}\dint_{\Omega \cap B_{2L}}\left( R+\alpha s\right)
^{\beta }\left( \left\vert u\right\vert ^{2}+\left\vert \nabla u\right\vert
^{2}+\left\vert \partial _{t}u\right\vert ^{2}\right) dxds \\ 
+\left( \frac{2}{\epsilon _{0}}+\frac{8k_{1}}{\alpha \left( 1-\beta \right) }%
+\frac{8\beta ^{2}\alpha }{\epsilon _{0}^{2}k_{1}\left( 1-\beta \right) }%
\right) \dint_{t}^{t+T}\dint_{\Omega }a\left( x\right) \left( R+\alpha
s\right) ^{\beta }\left\vert \partial _{t}u\right\vert ^{2}dxds,%
\end{array}%
\end{equation}%
for all $t\geq 0$ and any $\lambda >0.$
\end{lemma}

\begin{proof}
For the proof we have to argue as in the proof of Lemma \ref{lemma xt poly}
and to use the fact that 
\begin{equation*}
\begin{array}{l}
\medskip E_{\varphi }\left( u\right) \left( t+T\right)
+\dint_{t}^{t+T}\dint_{\Omega }a\left( x\right) \left( R+\alpha s\right)
^{\beta +1}\left\vert \partial _{t}u\right\vert ^{r+1}dxds \\ 
\leq E_{\varphi }\left( u\right) \left( t\right) +\frac{\left( \beta
+1\right) \alpha }{2}\dint_{t}^{t+T}\dint_{\Omega }\left( R+\alpha s\right)
^{\beta }\left( \left\vert \nabla u\right\vert ^{2}+\left\vert \partial
_{t}u\right\vert ^{2}\right) dxds.%
\end{array}%
\end{equation*}
\end{proof}

\subsection{Proof of Theorem 3}

For the proof we have to proceed as in the proof of theorem 2 and to use the
finite speed propagation property and the fact that the support of the
initial data is contained in $B_{R}$ to show that%
\begin{equation*}
\begin{array}{l}
\medskip \dint_{0}^{\infty }\left( R+\alpha s\right) ^{\beta }\dint_{\Omega
}a\left( x\right) \left\vert \partial _{t}u\right\vert ^{2r}dxds\leq
C\epsilon ^{-\frac{p-2r}{r-1}}\left\Vert a\right\Vert _{L^{\infty }}\left(
\left\Vert u_{0}\right\Vert _{H^{2}}^{2}+\left\Vert u_{1}\right\Vert
_{H^{1}}^{2}+\left\Vert u_{1}\right\Vert _{H^{1}}^{2r}\right) \\ 
+\frac{\epsilon \left( p-2r\right) }{p-r-1}\dint_{0}^{\infty }\dint_{\Omega
}a\left( x\right) \left( R+\alpha s\right) ^{\beta +1}\left\vert \partial
_{t}u\right\vert ^{r+1}dxds,%
\end{array}%
\end{equation*}%
and%
\begin{equation*}
\begin{array}{l}
\dint_{0}^{\infty }\left( R+\alpha s\right) ^{\beta }\dint_{\Omega }a\left(
x\right) \left\vert \partial _{t}u\right\vert ^{2}dxds\leq C\epsilon ^{-%
\frac{2}{r-1}}+\frac{2\epsilon }{r+1}\dint_{0}^{\infty }\left( R+\alpha
s\right) ^{\beta +1}\dint_{\Omega }a\left( x\right) \left\vert \partial
_{t}u\right\vert ^{r+1}dxds,%
\end{array}%
\end{equation*}%
for some positive constant $C$ and for all $\epsilon >0.$

\end{document}